\documentclass{article}
\usepackage{
amsmath,
amsfonts,
latexsym, 
amsrefs,
amssymb
}

\usepackage{bm}

\usepackage{tocloft}
\usepackage{cite}
\usepackage{subcaption}
\usepackage{graphicx}

\newcommand{\bla}{\bm{\lambda}}

\newcommand{\1}{\mathds{1}}
\usepackage{enumerate}
\usepackage{mathrsfs}
\usepackage{wrapfig}
\usepackage{hyperref}

\usepackage{color}

\numberwithin{equation}{section}

\newcommand{\rd}{{\rm d}}

\newcommand{\bx}{{\bf{x}}}

\newcommand{\be}{\begin{equation}}
\newcommand{\ee}{\end{equation}}

\newcommand{\e}{{\varepsilon}}

\newcommand{\la}{\lambda}

\usepackage{amsmath} %[intlimits]
\usepackage{amssymb}
\usepackage{amsthm}

\newcommand{\wt}{\widetilde}

\newcommand{\ii}{\mathrm{i}} %\newcommand{\mi}{\mathrm{i}}
\newcommand{\dd}{\mathrm{d}}

\renewcommand{\epsilon}{\varepsilon}
\renewcommand{\leq}{\leqslant}
\renewcommand{\geq}{\geqslant}

\renewcommand{\le}{\leq}

\newcommand{\E}{\mathbb{E}}

\DeclareMathOperator{\diag}{diag}

\DeclareMathOperator{\re}{Re}
\DeclareMathOperator{\im}{Im}

\DeclareMathOperator{\OO}{O}
\DeclareMathOperator{\oo}{o}

\theoremstyle{plain} %plain, definition, remark
\newtheorem{theorem}{Theorem}[section]
\newtheorem*{theorem*}{Theorem}
\newtheorem{lemma}[theorem]{Lemma}
\newtheorem*{lemma*}{Lemma}
\newtheorem{corollary}[theorem]{Corollary}
\newtheorem*{corollary*}{Corollary}
\newtheorem{proposition}[theorem]{Proposition}
\newtheorem*{proposition*}{Proposition}

\newtheorem*{definition*}{Definition}

\newtheorem*{example*}{Example}
\newtheorem{remark}[theorem]{Remark}

\newtheorem*{remark*}{Remark}
\newtheorem*{remarks*}{Remarks}

\makeatletter
\renewcommand{\subsection}{\@startsection
{subsection}%                   % the name
{2}%                         % the level
{0mm}%                       % the indent
{-\baselineskip}%            % the before skip
{0 \baselineskip}%          % the after skip
{\normalfont\bf\itshape}} % the style
\makeatother

\usepackage{dsfont}
\usepackage{stmaryrd}

\newcommand{\nc}{\normalcolor}

\def\@empty{}

\def\author#1{\par
    {\centering{\authorfont#1}\par\vspace*{0.05in}}
}

\def\titlefont{\fontsize{13}{15}\bfseries\boldmath\selectfont\centering{}}
\def\authorfont{\fontsize{13}{15}}
\def\abstractfont{\fontsize{8}{10}}

\let\affiliationfont\rhfont

\def\address#1{\par
    {\centering{\affiliationfont#1\par}}\par\vspace*{11pt}
}

\def\keywords#1{\par
    \vspace*{8pt}
    {\authorfont{\leftskip18pt\rightskip\leftskip
    \noindent{\it\small{Keywords}}\/:\ #1\par}}\vskip-12pt}

\def\title#1{
    \thispagestyle{plain}
    \vspace*{-14pt}
    \vskip 79pt
    {\centering{\titlefont #1\par}}%
    \vskip 1em
}

\renewenvironment{abstract}{\par%
    \vspace*{6pt}\noindent %{\bf Abstract}
    \abstractfont
    \noindent\leftskip18pt\rightskip18pt
}{%
  \par}

\usepackage{tocloft}

\makeatletter
\renewcommand{\section}{\@startsection
{section}%                   % the name
{1}%                         % the level
{0mm}%                       % the indent
{-2\baselineskip}%            % the before skip
{1\baselineskip}%          % the after skip
{\normalfont\large\scshape\centering}} % the style
\makeatother

\begin{document}

~\vspace{-1.5cm}

\title{The distribution of overlaps between eigenvectors of Ginibre matrices}

\vspace{1.2cm}
\noindent\begin{minipage}[b]{0.5\textwidth}

 \author{P. Bourgade}

\address{Courant Institute, New York University\\
  bourgade@cims.nyu.edu}
 \end{minipage}
\begin{minipage}[b]{0.5\textwidth}

 \author{G. Dubach}

\address{Courant Institute, New York University\\
 dubach@cims.nyu.edu}
 \end{minipage}
\begin{minipage}[b]{0.5\textwidth}

 \end{minipage}

%~\vspace{0.3cm}

\begin{abstract}
We study the overlaps between eigenvectors of nonnormal matrices. 
They quantify the stability of the spectrum, and characterize
the joint eigenvalues increments under Dyson-type dynamics.
Well known work by Chalker and Mehlig calculated the expectation of these overlaps for complex Ginibre matrices. For the same model, we extend their results by deriving the distribution of diagonal overlaps (the condition numbers), and their correlations. We prove:

\vspace{0.1cm}

{\addtolength{\leftskip}{1.5em} \setlength{\parindent}{0em}

\makebox[1.5em][l]{(i)} convergence of condition numbers for bulk eigenvalues to an inverse Gamma distribution; more generally, we decompose the quenched overlap (i.e. conditioned on eigenvalues)  as a product of independent random variables;
\par

\vspace{0.1cm}

\makebox[1.5em][l]{(ii)} asymptotic expectation of off-diagonal overlaps, both for microscopic or mesoscopic separation of the corresponding eigenvalues;\par

\vspace{0.1cm}

\makebox[1.5em][l]{(iii)} decorrelation of condition numbers associated to eigenvalues at mesoscopic distance, at polynomial speed in the dimension;
 \par

\vspace{0.1cm}

\makebox[1.5em][l]{(iv)}\, second moment asymptotics to identify the fluctuations order for off-diagonal overlaps, when the related eigenvalues are separated by any mesoscopic scale;\par

\vspace{0.1cm}

\makebox[1.5em][l]{(v)}\, a new formula for the correlation between overlaps for eigenvalues at microscopic distance, both diagonal and off-diagonal.\par}
\setlength{\parindent}{0em}

\vspace{0.1cm}

These results imply estimates on the extreme condition numbers, the volume of the pseudospectrum and the diffusive evolution of eigenvalues under Dyson-type dynamics, at equilibrium. 
\end{abstract}

\keywords{Nonnormal matrices, Ginibre ensemble, Eigenvectors overlaps,  condition number, pseudospectrum.}

\tableofcontents

{\let\thefootnote\relax\footnotetext{\noindent This work is supported by the NSF grant DMS\#1513587.}}

\newpage
\section{Introduction}

\subsection{The Ginibre ensemble.}\ 
Throughout this article we will essentially consider a complex Ginibre matrix $G_N =(G_{ij})_{i,j=1}^N$ where the $G_{ij}$'s are independent and identically distributed  complex Gaussian random variables, with distribution $\mu=\mu^{(N)}$:
\begin{equation}\label{eqn:Gini}
G_{ij} \stackrel{{\rm (d)}}{=} \mathscr{N}_{\mathbb{C}} \Big(0, \frac{1}{2N} \mathrm{Id} \Big), \qquad \mu(\dd \lambda) = \frac{N}{\pi} e^{- N|\lambda|^2} \dd m(\lambda),
\end{equation}
where $m$ is the Lebesgue measure on $\mathbb{C}$. As proved in  \cite{Gin1965},
the eigenvalues of $G_N$ have joint distribution
\begin{equation}\label{eqn:partitionfunction}
\rho_{N} (\lambda_1, \dots, \lambda_N) m^{\otimes N}(\rd\bla)= \frac{1}{Z_N} \prod_{j<k} |\lambda_j - \lambda_k|^2 
\prod_{k=1}^N\mu(\rd \lambda_k),
\end{equation}
where $Z_N=N^{-N(N-1)/2}\prod_{j=1}^N j!$.
The above measure is written $\mathbb{P}_N$, with corresponding expectation $\E_N$. The limiting empirical spectral measure converges to the circular law, i.e. $\frac{1}{N}\sum \delta_{\lambda_i}\to \frac{1}{\pi}\mathds{1}_{|\lambda|<1}\rd m(\lambda)$.

%We will also sometimes consider the real Ginibre ensemble, defined  through $G_{ij} \stackrel{{\rm (d)}}{=} \mathscr{N}_{\mathbb{R}} \left(0, \frac{1}{N}\right)$, with independent entries. With this scaling, the limiting circular law also holds.

The statistics of eigenvalues of Ginibre matrices have been studied in great details, and other non-Hermitian matrix models are known to be integrable, see e.g. \cite{KhoSom2011,For2010}. Much less is known about the statistical properties of eigenvectors of non-Hermitian ensembles.

\subsection{Overlaps.}\label{overlapssubsec}\ Almost surely, the eigenvalues of a Ginibre matrix are distinct and $G$ can be diagonalized with left eigenvectors denoted $(L_i)_{i=1}^N$, right eigenvectors $(R_i)_{i=1}^N$, defined by
$G R_i=\lambda_iR_i$, $L_i^{\rm t}G=\lambda_i L_i^{\rm t}$ (for a column vector $\bx$, we write $\bx^{\rm t}=(x_1,\dots,x_n)$, $\bx^{*}=(\overline{x_1},\dots,\overline{x_n})$ and 
 $\|\bx\|=(\bx^*\bx)^{1/2}$).
Right and left eigenvectors are biorthogonal basis sets, normalized by
\begin{equation}\label{eqn:normalized}
L_i^tR_j=\delta_{ij}.
\end{equation}
In other words, defining $X$ with $i$th column $R_i$, we have
$G= X \Delta X^{-1}$ with $\Delta={\rm diag}(\lambda_1,\dots,\lambda_N)$, and $L_i^{\rm t}$ is the $i$th row of $Y=X^{-1}$.
Because of the normalization (\ref{eqn:normalized}), the first interesting statistics to quantify non-orthogonality of the eigenbasis is
\begin{equation}\label{eqn:overlap}
\mathscr{O}_{ij}= (R_j^* R_i)(L_j^* L_i).
\end{equation}
These overlaps are invariant under the rescaling $R_i\to c_iR_i$, $L_i\to c_i^{-1} L_i$ and the
diagonal overlaps $\mathscr{O}_{ii}= \| R_i \|^2 \| L_i \|^2$ directly quantify the stability of the spectrum.
Indeed, if we assume all eigenvalues of $G$ are distinct and denote $\lambda_i(t)$ the eigenvalues of $G+tE$, standard perturbation theory yields (in this paper  $\|M\|=\sup_{\|\bx\|_2=1}\|M \bx\|_2$)
$$
\mathscr{O}_{ii}^{1/2}=\lim_{t\to 0}\sup_{\|E\|=1}t^{-1}|\lambda_i(t)-\lambda_i|,
$$
so that the $\mathscr{O}_{ii}^{1/2}$'s are also called condition numbers.
They also naturally appear through the formulas $\mathscr{O}_{ii}^{1/2}=\|R_iL_i^{\rm t}\|$
or  $\mathscr{O}_{ii}^{1/2}=\limsup_{z\to\lambda_i} \|(z-G)^{-1}\|\cdot |z-\lambda_i|$.
We refer to \cite[Sections 35 and 52]{Tre2005} for further discussion and references about the relevance of condition numbers to the perturbative theory of eigenvalues, and to estimates of the pseudospectrum.

  Eigenvector overlaps also play a fundamental role in non perturbative dynamical settings. 
First, the large off-diagonal $\mathscr{O}_{ij}$'s appear when $G$ is the generator of evolution in real or imaginary time, see \cite[Appendix B]{ChaMeh2000}. More generally, eigenvector correlations are as relevant as eigenvalue distributions in determining evolution at intermediate times,  a well known fact in hydrodynamic stability theory \cite{TreTreRedDri1993}.
Second, the overlaps also fully characterize the eigenvalue increments when all matrix entries undergo independent Brownian motions, as shown in Appendix A, for any deterministic initial condition.
For the Dyson Brownian motion on Hermitian matrices, the eigenvalues evolution is autonomous and coincides with Langevin dynamics for a one-dimensional log-gas. On the contrary, in the nonnormal setting, the Dyson and Langevin dynamics strongly differ.
More about the Dyson-type dynamics in the context of the Ginibre ensemble can be found in \cite{BurGreNowTarWar2014,GrelaWarchol}, and the Langevin equation related to (\ref{eqn:partitionfunction}) is studied in  \cite{BolChaFon2017}.

\subsection{Overlaps statistics.}\ The statistical study of overlaps started with the seminal work of Chalker and Mehlig \cite{ChaMeh1998,ChaMeh2000,MehCha1998}. They estimated the large $N$ limit of the expectation of diagonal and off-diagonal overlaps, for the complex Ginibre ensemble: for any $ |z_1|, |z_2|<1$,
\begin{align}
&\E\left(\mathscr{O}_{11}\mid \lambda_1=z_1\right)\underset{N\to\infty}{\sim}N(1-|z_1|^2),\label{eqn:O11}\\
&\E\left(\mathscr{O}_{12}\mid \lambda_1=z_1,\lambda_2=z_2\right)\underset{N\to\infty}{\sim}-\frac{1}{N}\frac{1-z_1\overline{z_2}}{|z_1-z_2|^4}\frac{1-(1+N|z_1-z_2|^2)e^{-N|z_1-z_2|^2}}{1-e^{-N|z_1-z_2|^2}}\label{eqn:O12},
\end{align}
with (\ref{eqn:O12}) uniformly in $|z_1-z_2|$ from the macroscopic up to the microscopic $N^{-1/2}$ scale\footnote{Our formula (\ref{eqn:O12}) differs from the analogues in \cite{ChaMeh1998,ChaMeh2000,MehCha1998,WalSta2015}
through the additional denominator, due to eigenvalues repulsion: we consider conditional expectation instead of averages.}.
In \cite{ChaMeh2000},  (\ref{eqn:O11}) and (\ref{eqn:O12}) were rigorously established for $z_1=0$,
and convincing heuristics extended them anywhere in the bulk of the spectrum. From (\ref{eqn:O11}), one readily
quantifies the instability of the spectrum, an order $N$ greater than for normal matrices in the bulk, and more stable closer to the edge. 

An increasing interest in the statistical properties of overlaps for nonnormal matrices followed in theoretical physics
\cite{JanNorNovPapZah1999,Rot2009,GoeSki2011,BelNowSpeTar2017,SchFraPatBee2000}, often
with interest in calculating overlaps averages beyond the Ginibre ensemble. For example, 
eigenvector overlaps appear to describe resonance shift if one perturbs a scattering system
\cite{FyoMeh2002,FyoSom2003,FyoSav2012}. This was experimentally verified  \cite{GroLegMorRicSav2014}. Remarkably, the exact statistics (\ref{eqn:O12}) appeared very recently in an experiment from \cite{DavyGenack} for microscopic separation of eigenvalues, suggesting some universality of this formula.
Unfortunately, many of the models considered in the physics literature are perturbative, and most of the examined statistics are limited to expectations.

In the mathematics community, the overlaps were recently studied in \cite{WalSta2015}. Walters and Starr  extended (\ref{eqn:O11}) to any $z_1$ in the bulk, established asymptotics for $z_1$ at the edge of the spectrum, and suggested an approach 
towards a proof of (\ref{eqn:O12}). They also studied the
connection between overlaps and mixed matrix moments. Concentration for such moments for more general matrix models was established in \cite{ErdKruRen2017}, together with applications to coupled differential equations with random coefficients.
We continue the rigorous analysis of overlaps by deriving the full distribution of the condition numbers for bulk eigenvalues of the complex Ginibre ensembles.
We also establish (\ref{eqn:O12}) and an explicit formula for the correlation between diagonal and off-diagonal overlaps, on any scale including microscopic. These formulas have 
consequences on the volume of the pseudospectrum and eigenvalues dynamics.

Motivated by our explicit distribution for the overlaps, Fyodorov \cite{Fyodorov2018} recently
derived the distribution of diagonal overlaps for real eigenvalues of real Ginibre matrices, as well as an alternative proof for the distribution of diagonal overlaps for the complex Ginibre ensemble. Fyodorov's method relies on the supersymmetry approach in random matrix theory,
while our technique is probabilistic, as described below.

\subsection{Main Results.}\ 
Equation (\ref{eqn:O11}) suggests that the overlaps have typical size of order $N$. 
For the complex Ginibre ensemble (like all results below), we confirm that this is indeed the typical
behavior, identifying the limiting distribution of $\mathscr{O}_{11}$. We recall that a Gamma random variable $\gamma_{\alpha}$ has density $\frac{1}{\Gamma(\alpha)}x^{\alpha-1}e^{-x}$ on $\mathbb{R}_+$.

\begin{theorem}[Limiting distribution of diagonal overlaps]\label{thm:diag} Let $\kappa>0$ be an arbitrarily constant. Uniformly\footnote{More precisely, for any smooth, bounded, compactly supported function $f$ and deterministic sequence $(z_N)$ such that $|z_N|<1-N^{-\frac{1}{2}+\kappa}$ we have $\E\left(f(\mathscr{O}_{11}/(N(1-|z_N|^2)))\mid \la_1=z_N\right)\to\E f(\gamma_2^{-1})$.} in $|z|<1-N^{-\frac{1}{2}+\kappa}$, the following holds.
Conditionally on $\lambda_1=z$, the rescaled diagonal overlap $\mathscr{O}_{11}$ converges in distribution to an inverse Gamma random variable with parameter 2 as $N\to\infty$, namely
\begin{equation}\label{eqn:law}
\frac{\mathscr{O}_{11}}{N(1-|z|^2)} \overset{(\rm d)}{\to}  \frac{1}{\gamma_2}.
\end{equation}
\end{theorem}

\noindent Our proof  also gives convergence of the expectation, in the complex Ginibre case, so that it extends (\ref{eqn:O11}).
Equation (\ref{eqn:law}) means that for any continuous bounded function $f$ we have
\begin{equation}\label{eqn:density}
\E\left(
f\left( \frac{ \mathscr{O}_{11}}{N(1-|z|^2)}\right)
\mid \la_1=z\right)
\to
\int_0^\infty f(t)\, \frac{e^{-\frac{1}{t}}}{t^3} \rd t.
\end{equation}
This $t^{-3}$ asymptotic density was calculated for $N=2$ in \cite[Section V.A.2]{ChaMeh2000}, where this heavy tail was suggested to remain in the large $N$ limit.

Theorem \ref{thm:diag} requires integrating over all the randomness of the Ginibre ensemble, in this sense this is an {\it annealed} result. It derives from a {\it quenched} result, when conditioning on all eigenvalues: the overlap $\mathscr{O}_{11}$ can then be decomposed as a product of independent random variables, see Theorem \ref{thm:quenched}. Very similar results have been recently established for the Quaternionic Ginibre Ensemble \cite{DubachQGE}, as well as for the Spherical and Truncated Unitary Ensembles \cite{DubachSpherical}.

We observe that the limiting density in (\ref{eqn:density}) vanishes  exponentially fast at $0$, so that it is extremely unlikely to find any bulk overlap of polynomial order smaller than $N$: the spectrum is uniformly unstable. This is confirmed by the following bound on the extremes of condition numbers.

\begin{corollary}[Bounds on the condition numbers]\label{cor:extremes} Let $\kappa,\e>0$, $\kappa<\kappa_0\leq 1/2$ be fixed  and $\Omega_N\subset\{1-N^{-\frac{1}{2}+\kappa_0}\leq |z|<1-N^{-\frac{1}{2}+\kappa}\}$ be deterministic, measurable.
Then with probability tending to 1 as $N\to\infty$, the following event holds: for any $\lambda_i\in\Omega_N$,
$$
N^{\frac{1}{2}+\kappa-\e} \leq \mathscr{O}_{ii}\leq N^{1+\kappa_0+\e}m(\Omega_N)^{1/2}.
$$
\end{corollary}
\noindent In particular, all bulk overlaps are in $[N^{1-\e},N^{3/2+\e}]$ with large probability.
In terms of polynomial scales in $N$, the above lower bound is clearly optimal, and we believe the upper bound is also the best possible.

The next result is a rigorous proof of (\ref{eqn:O12})  in the bulk of the spectrum. It answers Conjecture 4.5 in \cite{WalSta2015} and gives firm grounds to the heuristic arguments of Chalker and Mehlig \cite{ChaMeh2000}. Different heuristics towards Theorem \ref{thm:expoff} for more general ensembles recently appeared in \cite{NowTar2018}, based on diagrammatics. Another recent approach \cite{CrawfordRosenthal} allows to compute the conditional expectation of more general multi-index overlaps when eigenvalues are conditioned to be at macroscopic distance.

\begin{theorem}[Expectation of off-diagonal overlaps, microscopic and mesoscopic scales]\label{thm:expoff} 
For any $\kappa\in(0,1/2)$, any $\e>0$ and  $C>0$ the following holds.
Uniformly in $z_1,z_2$ such that $|z_1|<1-N^{-\frac{1}{2}+\kappa}$, $\omega=\sqrt{N}|z_1-z_2|\in[N^{-C},N^{\kappa-\e}]$, 
 we have
\begin{align}
\E\left(\mathscr{O}_{12}\mid \lambda_1=z_1,\lambda_2=z_2\right)&=-N\frac{1-z_1\overline{z_2}}{|\omega|^4}\frac{1-(1+|\omega|^2)e^{-|\omega|^2}}{1-e^{-|\omega|^2}} \left(1+\OO(N^{-2\kappa+\e})\right) .\label{eqn:diagok}
\end{align}
%Let $\kappa,\e>0$ be any small constants, and $D>0$ be fixed arbitrarily large. Uniformly in $|z_1|, |z_2|<1-N^{-\frac{1}{2}+\kappa}$, $N^{-1/2+\kappa}<|z_1-z_2|<N^{-\kappa}$, the following holds.
%Conditionally on $(\la_1,\la_2)=(z_1,z_2)$, we have
%\begin{align}
%\E\left(\mathscr{O}_{12}\right)&=-\frac{1-z_1\overline{z_2}}{N|z_1-z_2|^4}+\OO(N^{-D})\label{eqn:diagok1}
%\end{align}
\end{theorem}

\noindent In particular, under the same hypothesis, in the mesoscopic regime $|w|\to\infty$, equation (\ref{eqn:diagok}) simplifies to
\begin{equation}\label{eqn:mesoasymp}
\E\left(\mathscr{O}_{12}\mid \lambda_1=z_1,\lambda_2=z_2\right)=-\frac{1-z_1\overline{z_2}}{N|z_1-z_2|^4}(1+\oo(1)),
\end{equation}
showing that the expectation of off-diagonal overlaps decreases with the separation of eigenvalues. 
Our next result, about second moments at any scale, allows to identify the natural size of off-diagonal overlaps, and 
gives polynomial decay of correlations between condition numbers.

\begin{theorem}[Correlations of overlaps: microscopic and mesoscopic scales]\label{thm:meso2} 
Let $\kappa\in(0,1/2)$ and $\sigma\in(0,\kappa)$. Let $\e>0$. Then
uniformly in $z_1,z_2$ such that $|z_1|<1-N^{-\frac{1}{2}+\kappa}$, $\omega=\sqrt{N}|z_1-z_2|\in[N^{-\frac{\kappa}{2}+\e},N^{\sigma}]$, 
 we have
\begin{align}
\E\left(|\mathscr{O}_{12}|^2\mid \lambda_1=z_1,\lambda_2=z_2\right)&= \frac{N^2(1-|z_1|^2)(1-|z_2|^2)}{|\omega|^4} \left(1+\OO(N^{2(\sigma-\kappa)+\e})\right),\\
\E\left(\mathscr{O}_{11}\mathscr{O}_{22}\mid \lambda_1=z_1,\lambda_2=z_2\right)&= \frac{N^2(1-|z_1|^2)(1-|z_2|^2)}{|\omega|^4}\frac{1+|\omega|^4- e^{-|\omega|^2}}{1-e^{-|\omega|^2}}\left(1+\OO(N^{2(\sigma-\kappa)+\e})\right)\label{eqn:decorrel}.
\end{align}
\end{theorem}

%\begin{theorem}[Joint statistics of overlaps: the mesoscopic scale]\label{thm:meso1} Let $\kappa,\e>0$ be any small constants, and $D>0$ be fixed arbitrarily large. Uniformly in $|z_1|, |z_2|<1-N^{-\frac{1}{2}+\kappa}$, $N^{-1/2+\kappa}<|z_1-z_2|<N^{-\kappa}$, the following holds.
%Conditionally on $(\la_1,\la_2)=(z_1,z_2)$, we have
\noindent For the mesoscopic scales $|w|\to\infty$, the above asymptotics become
\begin{align}
\E\left(|\mathscr{O}_{12}|^2\mid \lambda_1=z_1,\lambda_2=z_2\right)&\sim \frac{(1-|z_1|^2)(1-|z_2|^2)}{|z_1-z_2|^4},\label{eqn:diagok2}\\
\E\left(\mathscr{O}_{11}\mathscr{O}_{22}\mid \lambda_1=z_1,\lambda_2=z_2\right)&\sim\E\left(\mathscr{O}_{11}\mid \lambda_1=z_1\right)\E\left(\mathscr{O}_{22}\mid \lambda_2=z_2\right)\label{eqn:indep}.
\end{align}
%\end{theorem}

\noindent Equations (\ref{eqn:mesoasymp}) and (\ref{eqn:diagok2}) suggest that for any mesoscopic separation of eigenvalues, $\mathscr{O}_{12}$ does not concentrate, because $\E(|\mathscr{O}_{12}|^2)$ is of larger order than $\E(\mathscr{O}_{12})^2$. Contrary to (\ref{eqn:O12}), (\ref{eqn:diagok2}) therefore identifies the size of off-diagonal overlaps, at mesoscopic scales.

The  covariance bounds from  Theorem \ref{thm:meso2} yield effective estimates on the volume of the pseudospectrum, defined through
$
\sigma_\e(G)=\left\{z: \|z-G\|^{-1}>\e^{-1}\right\}.
$
We state the result when the pseudospectrum is intersected with a mesoscopic ball, although it clearly holds on any domain within the bulk that is regular enough.

\begin{corollary}[Volume of the pseudospectrum]\label{cor:pseudo} Let $\kappa>a>0$ be any constants and $\mathscr{B}_N\subset \{|z|<1-N^{-\frac{1}{2}+\kappa}\}$ be a ball with radius at least $N^{-\frac{1}{2}+a}$, at most $N^{-\frac{1}{2}+\kappa-a}$.
Then the volume of the pseudospectrum in $\mathscr{B}_N$ is deterministic at first order: for any $c>0$, 
\begin{equation}\label{eqn:pseudo}
\lim_{N \to \infty} \lim_{\e\to0}\mathbb{P}\left(1-c<\frac{m\left(\sigma_{\e}(G)\cap \mathscr{B}_N\right)}{\e^2N^2 {\int_{\mathscr{B}_N}} (1-|z|^2)\rd m(z)}<1+c\right)= 1.
\end{equation}
\end{corollary}

Finally, we remark that our results  shed some light on natural matrix dynamics on nonnormal matrices, the Dyson-type evolution where all matrix entries follow independent complex Ornstein-Uhlenbeck processes. 
Under this evolution, the eigenvalues follow the dynamics (see Proposition \ref{prop:dynamics})
$$ \dd \lambda_k(t) = \rd M_k(t)  -  \frac{1}{2} \lambda_k(t)\dd t  $$
where the martingales $(M_k)_{1\leq k\leq N}$ have brackets $\langle M_i,M_j\rangle=0$ and
$
\rd\langle M_i,\overline{M_j}\rangle_t=\mathscr{O}_{ij}(t)\frac{\rd t}{N}.
$ 
Based on this observation, Theorem \ref{thm:diag}, Theorem \ref{thm:expoff} and some bound from \cite{Fyodorov2018}, we show that the eigenvalues propagate with diffusive scaling, at equilibrium, but 
slower when close to the boundary.

\begin{corollary}[Diffusive exponent for eigenvalues dynamics]\label{cor:diffusive} Let $c,a>0$ be arbitrarily. Consider the matrix dynamics (\ref{eqn:dynamics}) with initial condition $G(0)$ distributed as (\ref{eqn:Gini}). 
Let  $\mathscr{B}\subset \{|z|<1-c\}$ be a ball, and
$t<N^{-c}$. Then as $N\to\infty$ we have 
\begin{align}
&\E(|\lambda_1(t)-\lambda_1(0)|^2\mathds{1}_{\lambda_1(0)\in \mathscr{B}})=t\int_{\mathscr{B}}(1-|z|^2)\frac{\rd m(z)}{\pi}(1+\oo(1)),\label{eqn:scale}\\
&\E\left((\lambda_1(t)-\lambda_1(0))\overline{(\lambda_2(t)-\lambda_2(0))}\mathds{1}_{\{\lambda_1(0)\in \mathscr{B}\}\cap\{|\la_1(0)-\la_2(0)|<N^{-a}\}}\right)=\oo(tN^{-2a}).\label{eqn:indepen}
\end{align}
\end{corollary}

\noindent
Given the time scale in (\ref{eqn:scale}), we expect that, conditionally on $\{\lambda_i(0)=z\}$, the process
$$\frac{(\lambda_i(ts)-z)_{0\leq s\leq 1}}{\sqrt{t(1-|z|^2)}}$$
converges in distribution to a normalized  complex Brownian motion as $N\to\infty$ ($t$ is any scale $N^{-1+c}<t<N^{-c}$), i.e. 
$
\frac{1}{\sqrt{2}}(B_s^{(i)}+\ii \wt B_s^{(i)})_{0\leq s\leq 1}
$ with $B^{(i)}$ and  $\wt B^{(i)}$ independent standard Brownian motions.
Moreover, from (\ref{eqn:indepen}), we expect that these limiting processes associated to different eigenvalues are independent.

%\noindent In Appendix A, we will also prove that increments are uncorrelated, in the sense that
%\begin{align}
%&\E((\lambda_i(t)-\lambda_i(0))(\lambda_j(t)-\lambda_j(0))\mathds{1}_{\lambda_i(0),\lambda_{j}(0)\in \Omega})=\oo(t),\\
%&\E((\lambda_i(t)-\lambda_i(0))(\overline{\lambda_j(t)-\lambda_j(0)})\mathds{1}_{\lambda_i(0),\lambda_{j}(0)\in \Omega})=\oo(t).
%\end{align}

\subsection{About the proofs.}\
Our analysis of the condition numbers starts with the observation of Chalker and Mehlig: the overlaps coincide with those of the Schur form of the original operator, a particularly simple decomposition when  the input is a Ginibre matrix. We refer the reader to (\ref{eqn:T}) for this key structure at the source of our inductions.

In Section \ref{sec:diag} our method to prove Theorem \ref{thm:diag}  follows from a simple, remarkable identity in law: the {\it quenched} overlap (i.e. conditionally on eigenvalues and integrating only over the complementary randomness of the Ginibre ensemble) is a product of $N-1$ independent random variables, see Theorem \ref{thm:quenched}. In the specific case $\lambda_1=0$, for the complex Ginibre ensemble, this split of the distribution of $\mathscr{O}_{11}$ remains in the {\it annealed} setting, as a consequence of a theorem of Kostlan: the radii of eigenvalues of complex Ginibre matrices are independent random variables. For the Ginibre ensemble with conditioned eigenvalue, we give the extension of Kostlan's theorem in Section \ref{eqn:Sec}. This concludes a short probabilistic proof of Theorem \ref{thm:diag} in the case $\lambda_1=0$. 

The extension to general $\lambda_1$ in the bulk proceeds by decomposition of 
$\mathscr{O}_{11}$ into a long-range factor, which gives the deterministic $1-|z|^2$ coefficient, and a short-range factor, responsible for the $\gamma_2^{-1}$ fluctuations. Concentration of the long-range contribution relies on rigidity results for eigenvalues, from \cite{BouYauYin2014II,BreDui2014}. To prove that the short-range contribution is independent of the position of $\lambda_1$, we need strong  form of invariance for our measure, around the conditioned point $\lambda_1$.
While translation invariance for $\mathbb{P}$ follows easily from the explicit form of the Ginibre determinantal kernel, the  invariance of the {\it conditioned} measure requires more involved tools such as
the negative association property for determinantal point processes, see Section \ref{sec:translation}. Corollary \ref{cor:extremes} directly follows from our estimates on the speed of convergence to the inverse Gamma distribution, as explained in Section \ref{sec:diag}.

The proof of theorems \ref{thm:expoff} and \ref{thm:meso2} in Section \ref{sec:offdiag} follows the same scheme: first a quenched identity, then an explicit formula obtained for $z_1=0$ and a localization procedure to isolate the short and long-range contributions. The main difference with the proof of Theorem \ref{thm:diag} concerns the special case $z_1=0$, the other step being more robust.
Due to conditioning on  $z_2$, rotational invariance of the remaining eigenvalues is broken and there is no analogue of Kostlan's theorem to obtain the explicit joint distribution 
of $\mathscr{O}_{12}, \mathscr{O}_{11}$ and $\mathscr{O}_{22}$.
Despite this lack of rotational invariance, Chalker and Mehlig had already obtained a closed-form formula for $\mathbb{E}(\mathscr{O}_{12})$ when $z_1=0$. Remarkably, the second moments are also explicit (although considerably more involved) even for finite $N$, see Proposition \ref{prop:theformula}.
Our correlation estimates on overlaps imply Corollary \ref{cor:pseudo} by a second moment method.
It is plausible yet unclear that Theorem  \ref{thm:meso2} admits generalizations to joint  moments with an arbitrary number of conditioned eigenvalues. In fact, 
our second moment calculation involves a P-recursive sequence (see (\ref{eqn:recf})), for which explicit solutions are not expected in general.
We hope to address the general study of relevant holonomic sequences  in future work.

In Appendix A, we consider eigenvalues dynamics. After deriving the relevant stochastic differential equations in Proposition \ref{prop:dynamics}, we show that Corollary \ref{cor:diffusive} follows from Theorem \ref{thm:diag} and Theorem \ref{thm:expoff}.  The diagonal overlaps dictates the eigenvalues quadratic variation, the off-diagonal overlaps their correlations.

Finally, although this article focuses on the eigenvalues condition numbers, the Schur decomposition technique also readily answers the natural question of angles between normalized Ginibre eigenvectors, as explained in Appendix B.

\begin{figure}[b!]
\centering
\begin{subfigure}{.3\textwidth}
\includegraphics[width=5cm]{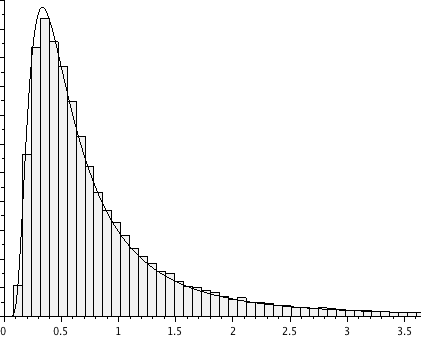}
\caption{Complex Ginibre ensemble, entries have density (\ref{eqn:Gini}).}
\end{subfigure}%
\hspace{0.2cm}
\begin{subfigure}{.3\textwidth}
\includegraphics[width=5cm]{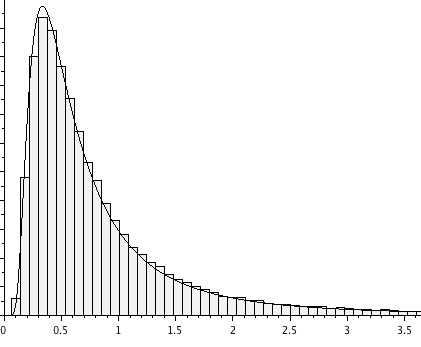}
\caption{Complex Bernoulli,  independent $\pm 1$ real and imaginary parts.}
\end{subfigure}
\hspace{0.2cm}
\begin{subfigure}{.3\textwidth}
\vspace{0cm}\includegraphics[width=5cm]{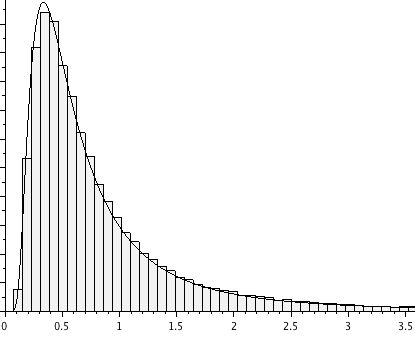}
\caption{Complex Uniform, entries are uniform on $\{|z|<1\}$.}
\end{subfigure}
\caption{The histogram of the overlaps $\frac{\mathscr{O}_{ii}}{N(1-|\la_i|^2)}$ associated to bulk eigenvalues for different densities of the matrix entries. The average is performed over all bulk eigenvalues of a $600\times 600$ matrix, sampled 100 times. The curve gives the density of the inverse Gamma distribution with parameter 2.}
\label{Fig1}
\end{figure}

\subsection{Numerical test for universality of the distribution of overlaps.}\ Universality of eigenvector statistics recently attracted a lot of attention for random Hermitian matrices.
For the Gaussian orthogonal and unitary ensembles, the eigenvectors basis is Haar distributed on  the corresponding unitary group. As a consequence, projections of eigenvectors on
deterministic directions are asymptotically normal, as the projection of the uniform measure on high dimensional spheres (a result due to L\'evy and Borel). These eigenvector statistics are now known to be universal \cite{TaoVu2012,KnoYin2013,BouYau2017}, holding for generalized Wigner matrices.

The situation is quite different for eigenvectors of dense random non Hermitian matrices: orders of magnitude such as delocalization are known \cite{RudVer2015} , but universality seems out of reach with current techniques. In Figure \ref{Fig1}, we numerically investigate whether the inverse Gamma distribution from Theorem \ref{thm:diag} describes the typical behavior of condition numbers.
This leads us to conjecture that  for any complex random matrix with i.i.d. entries with substantial real and imaginary parts, the normalized condition numbers converge to the inverse Gamma distribution with parameter two.

\vspace{-.2cm}

\subsection{Notations and conventions.}\ 
The partial sums of the  exponential are written
\begin{equation}\label{eqn:expo}
e_k^{(\ell)}(x)= \sum_{i=k}^{\ell} {x^i \over i!}
\end{equation}
and we abbreviate  $e_k=e_k^{(\infty)}$, $e^{(N)}=e_0^{(N)}$. Throughout the paper, $0 \leq \chi \leq 1$ is a smooth cut-off function  on $\mathbb{R}_+$ such that
$\chi(x)=1$ for $x<1/2$ and $0$ for $x>1$.
We write $f=\OO(\phi)$ if $|f|<C|\phi|$ for some $C>0$ which does not depend on $N$, and  $f=\oo(\phi)$ if $|f|/|\phi|\to 0$ as $N\to\infty$.
Along this work, the constants $C$ and $c$ are some universal (resp. typically large and small) constants that may vary from line to line.\\

\noindent{\bf Acknowledgement.} The authors thank the referees for particularly precise and pertinent suggestions which helped improving this article.

%Moreover, in the following sections all results and proofs will be about the complex Ginibre ensemble, unless we explicitly mention we consider real Ginibre matrices.

\section{Diagonal overlaps}\label{sec:diag}

This section first gives a remarkable identity in law for the diagonal overlap conditioned on the eigenvalues, Theorem \ref{thm:quenched}. Eigenvalues
are then integrated, first for $\lambda_1$ at the center of the spectrum thanks to a variant of Kostlan's theorem, then anywhere in the bulk.

\subsection{The quenched diagonal overlap.}\ 
For all $i\neq j$ we  denote
$\alpha_{ij} = \frac{1}{\lambda_i - \lambda_j},$
and $\alpha_{ii}=0$.
These numbers satisfy
\begin{align}  
\alpha_{ij} +  \alpha_{ji} =0,\  \alpha_{ij} +  \alpha_{jk} =\frac{ \alpha_{ij} \alpha_{jk}}{\alpha_{ik}}. \label{eqn:ward}
\end{align}
We first recall the analysis of the overlaps given by Chalker and Mehlig, and include the short proof for completeness.

\begin{proposition}[from \cite{ChaMeh1998,ChaMeh2000,MehCha1998}]\label{fund} The following joint equality in distribution holds: 
$$ \mathscr{O}_{11} =  \sum_{i=1}^N |b_i|^2,  \qquad  \mathscr{O}_{12}= - \overline{b_2} \sum_{i=2}^N b_i \overline{d_i},  \qquad  \mathscr{O}_{22} = (1+|b_2|^2) \sum_{i=2}^N |d_i|^2,$$
where $b_1= d_2=1$, $d_1=0$, and the next terms are defined by the recurrence
$$  b_i =\alpha_{1i} \sum_{k=1}^{i-1} b_k T_{ki}\ {\rm for}\  i\geq 2,\ \  d_i =\alpha_{2i}\sum_{k=1}^{i-1} d_k T_{ki}\ {\rm for}\  i\geq 3.$$
The $T_{ij}$ are independent complex Gaussian variables, centered and with variance $N^{-1}$ ($T_{ij}
\stackrel{{\rm (d)}}{=} 
G_{ij}$). 
\end{proposition}

\begin{proof} We defined $X$ and $Y=X^{-1}$ such that $G=X \Delta Y$ (see subsection \ref{overlapssubsec}). The QR algorithm applied to $X$ yields $X=UR$, $Y=R^{-1} U^*$, and thus the Schur form $G=UTU^*$ with $T= R \Delta R^{-1}$. It is straightforward to check that $U$ is independent of $T$ and uniformly distributed on the unitary group. The overlaps are unchanged by an unitary change of basis, and therefore we can study directly the overlaps of the matrix $T$. As proved in \cite[Appendix 35]{Meh1991}, this matrix $T$ has the eigenvalues of $G$ as diagonal entries, and independently the upper triangle consists in uncorrelated Gaussian random variables, $T_{ij}
\stackrel{{\rm (d)}}{=} 
G_{ij}$:
\begin{equation}\label{eqn:T}
T=  \left( \begin{array}{cccc}
\lambda_1 & T_{12} & \dots & T_{1N} \\
0 & \lambda_2 & \dots & T_{2N} \\
 \vdots & \ddots & \ddots & \vdots \\
0 & \dots & 0 & \lambda_N \\
 \end{array} \right).
\end{equation}
For $T$, the right eigenvectors are of type
 $$ R_1=(1,0, \dots, 0)^t, \qquad R_2=(a,1,0,\dots, 0)^t,$$
and the left eigenvectors are denoted
$$L_1=(b_1, \dots, b_N)^t, \qquad L_2=(d_1,\dots, d_N)^t.$$
Biorthogonality relations give $b_1=1$, $d_1=0$, $d_2=1$ and $a=- b_2$. The formulas given for the overlaps follow,  and the recurrence formulas proceed from the definition of the eigenvectors, i.e. $L_j^t T = \lambda_j L_j^t$.  
 \end{proof}
Proposition \ref{fund} shows that the eigenvectors, and thus the overlaps, are obtained according to a very straightforward random process. Indeed, let us consider the sequences of column vectors:
 \begin{align*}
B_k &= (1, b_2,\dots ,b_k)^{\rm t}  \qquad \text{so that } L_1=B_N, \\
D_k &= (0, 1, d_3,\dots , d_k)^{\rm t} \qquad \text{so that } L_2=D_N, \\
T_k &= (T_{1,k+1},\dots ,T_{k,k+1})^{\rm t} \qquad \text{(subset of the $k+1$ th column of $T$)}.
\end{align*}
For any $k$, $T_k$ is a $k$-dimensional centered Gaussian vector with independent coordinates and variance $1/N$. We denote the corresponding $\sigma$-algebras
$$
\mathcal{F}_n=\sigma\left(T_k,1\leq k\leq n\right)=\sigma\left(T_{i,j+1},1\leq i\leq j\leq n\right).
$$
In particular, $b_2=\frac{T_{12}}{\lambda_1-\lambda_2}\in\mathcal{F}_1$.
The recurrence formula from Proposition \ref{fund} becomes
\begin{align}
b_{n+1} = \alpha_{1,n+1}\,  B_n^{\rm t} T_n,\ d_{n+1} = \alpha_{2, n+1}\, D_n^{\rm t} T_n,\ n\geq 1.\label{eqn:rec22}
\end{align}

%
%
%In this section we will use the distribution of $N$ Ginibre points conditioned at $\lambda_1 = z$. The remaining eigenvalues have joint distribution
%$$ \rho_{N}^{(z)} (\lambda_2, \dots, \lambda_{N})m^{\otimes N-1}(\rd\bla) = \frac{1}{ Z_N^{(z)}} \prod_{2 \leq i <j \leq N} |\lambda_i - \lambda_j|^2 \prod_{i=2}^{N} |z - \lambda_i|^2\dd \mu(\lambda_i)$$
%where the normalizing constant is obtained thanks to Theorem \ref{cond1}, i.e. 
%\begin{equation}\label{constantZ1}
%Z_N^{z}
%=
%\left(\prod_{k=1}^{N-1}k!\right)N^{-\frac{N(N-1)}{2}}e^{(N-1)}(N|z|^2).
%\end{equation}
%

\begin{theorem}\label{thm:quenched} The following equality in law holds  conditionally on $\{\lambda_1,\dots,\lambda_N\}$:
$$ \mathscr{O}_{11} \overset{(\rm d)}{=} \prod_{n=2}^{N} \Big(1 + {|X_n|^2 \over N |\lambda_1-\lambda_n |^2}\Big) $$
where the $X_n$'s are independent standard complex Gaussian random variables ($X_n\stackrel{{\rm (d)}}{=}
\frac{1}{\sqrt{2}}(\mathscr{N}_1+\ii\mathscr{N}_2)$ with standard real Gaussians $\mathscr{N}_1$, $\mathscr{N}_2$).
\end{theorem}

\begin{remark}
In particular
\begin{equation}\label{eqn:O11Exp}
\E_T (\mathscr{O}_{11}) = \prod_{n=2}^{N} \Big(1 + \frac{1}{ N |\lambda_1-\lambda_n |^2}\Big),
\end{equation}
where  $\E_T$ is partial integration in the upper-diagonal variables $(T_{ij})_{j>i}$. We therefore recover the result by Chalker and Mehlig \cite{ChaMeh1998,ChaMeh2000,MehCha1998}.
\end{remark}
\begin{proof}
For fixed $N$ and $1\leq n\leq N$  we shall use the notations
\begin{equation}\label{eqn:Ogen}
\mathscr{O}_{11}^{(n)} = \| B_n \|^2,\ 
\mathscr{O}_{12}^{(n)} = - \overline{b_2}\, B_n^{\rm t} \bar D_n,\ 
\mathscr{O}_{22}^{(n)} = (1+|b_2|^2) \| D_n \|^2.
\end{equation}
Note that
$$
\mathscr{O}_{11}^{(1)} =1,\  \mathscr{O}_{11}^{(2)} = 1+ |b_2|^2 = 1 + \Big| {T_{12} \over \lambda_1 - \lambda_2} \Big|^2
$$
and, with (\ref{eqn:rec22}),
$$
\mathscr{O}_{11}^{(n+1)} = \| B_{n+1} \|^2 
= \| B_n \|^2 + |b_{n+1}|^2
= \| B_n \|^2 \Big( 1+ { | \alpha_{1,n+1}\,B_n^{\rm t} T_n |^2 \over  \| B_n \|^2} \Big)
= \mathscr{O}_{11}^{(n)} \Big( 1+ {|X_{n+1}|^2 \over N | \lambda_1 - \lambda_{n+1}|^2}  \Big),
$$
where $X_{n+1}:= \sqrt{N} {B_n^{\rm t}T_n  \over  \| B_n \|}$ is a $\mathcal{F}_{n}$-measurable Gaussian with variance $1$, independent of $\mathcal{F}_{n-1}$. We have therefore proved the expected factorization
with independent $X_n$'s, by an immediate induction.
\end{proof}

%
%\subsection{Preliminaries about the Ginibre radii}
%The next step uses as an essential tool the following property of the Ginibre radii: they behave like independent and unordered gamma variable.
%
%
%
%\begin{theorem}[Kostlan, \cite{Kos1992}]\label{kost} The set $N\{ |\lambda_1|^2, \dots, |\lambda_N|^2 \} $ is distributed as $\{\gamma_1,\dots\gamma_N\}$, the set of (unordered) independent Gamma variables of parameters $1,2, \dots, N$.\label{cor:Kos}
%\end{theorem}
%
%\begin{theorem}\label{kostcond} Conditioned on $\lambda_1 =0$, $N\{ |\lambda_2|^2, \dots, |\lambda_N|^2 \} $ is distributed as $\{\gamma_2,\dots\gamma_N\}$, a set of (unordered) independent Gamma variables of parameters $2,3, \dots, N$.
%\end{theorem}
%
%\noindent Elementary proofs can be found in Section \ref{eqn:Sec}. A general theorem about radially symmetric determinantal point processes can be found in \cite{DPP}.

\subsection{The annealed diagonal overlap at the origin.}\ 
We recall that a Gamma random variable $\gamma_{\alpha}$ has density $\frac{1}{\Gamma(\alpha)}x^{\alpha-1}e^{-x}$ on $\mathbb{R}_+$, and a Beta random variable $\beta_{a,b}$ has density
$\frac{\Gamma(a+b)}{\Gamma(a)\Gamma(b)}x^{a-1}(1-x)^{b-1}\mathds{1}_{[0,1]}(x)$.

\begin{proposition}\label{prop:at0}
Conditionally on $\{ \lambda_1 =0 \}$, the following equality in distribution holds:
\begin{equation} \mathscr{O}_{11} \overset{(\rm d)}{=} \frac{1}{\beta_{2,N-1}}.\label{eqn:law2}
\end{equation}
In particular,
$
\E \left( \mathscr{O}_{11}\mid \lambda_1=0 \right)=N$ and 
$N^{-1} \mathscr{O}_{11}$ converges weakly to $\gamma_2^{-1}$.
\end{proposition}

\begin{proof} 
With the notations from Theorem \ref{thm:quenched}, we have $(|X_2|^2,\dots,|X_N|^2) \overset{(\rm d)}{=} (\gamma_1^{(2)},\dots\gamma_1^{(N)})$, a collection of $N-1$ independent
Gamma random variables with parameter 1. 
Moreover, still conditionally on $\lambda_1=0$,  from Corollary \ref{kostbis} we have $\{N |\lambda_2|^2,\dots, N |\lambda_N|^2\}\overset{(\rm d)}{=}\{\gamma_2,\dots,\gamma_N\}$, a set of independent
Gamma random variables with corresponding parameters. Theorem \ref{thm:quenched} therefore yields, conditionally on $\lambda_1=0$,
$$ \mathscr{O}_{11}\overset{(\rm d)}{=} \prod_{j=2}^N \left(1 + { \gamma_1^{(j)} \over \gamma_j}\right),$$
where all random variable are independent. Equation (\ref{eqn:law2}) then follows immediately from Lemma \ref{lemmebeta1} below.
This readily implies 
\begin{equation}\label{eqn:expec}
\E \left( \mathscr{O}_{11}\mid  \lambda_1=0 \right) = \E \left(\beta_{2,N-1}^{-1}\right) = {\Gamma(1) \Gamma(N+1) \over \Gamma(2) \Gamma(N) }=N.
\end{equation}
The convergence in distribution follows from a simple change of variables: for any bounded test function $f$,
\begin{equation}\label{eqn:betaasymp}
\E\left(f\left(\frac{1}{N\beta_{2,N-1}}\right)\right)
=
\int_0^1f\left(\frac{1}{Nx}\right)\frac{N!}{(N-2)!}x(1-x)^{N-2}\rd x
=
\int_{N^{-1}}^\infty f\left(t\right)\frac{N-1}{N}\left(1-\frac{1}{Nt}\right)^{N-2}\frac{\rd t}{t^3},
\end{equation}
which clearly converges to the right hand side of (\ref{eqn:density}) as $N\to\infty$.
\end{proof}

\begin{lemma}\label{lemmebeta1} The following equalities in distribution hold, where all random variables with different indexes are independent:
\begin{align}
&{\gamma_a \over \gamma_a + \gamma_b}  \overset{(\rm d)}{=} \beta_{a,b},\label{eqn:Gamma}\\
&\prod_{j=2}^{N} \beta_{j,1}  \overset{(\rm d)}{=} \beta_{2,N-1}\label{eqn:Beta}.
\end{align}
\end{lemma}

\begin{proof}
Equation (\ref{eqn:Gamma}) is standard, see e.g. 
\cite[Chapter 25]{JohKot1970}. The equality (\ref{eqn:Beta}) follows by immediate induction from the following property \cite{DiaFre1987}: if $\beta_{p,q}$ and $\beta_{p+q,r}$ are independent, their product has distribution $\beta_{p,q+r}$.
\end{proof}

\subsection{The annealed diagonal overlap in the bulk.}\ \label{subsec:1stpatch}
With the following theorem, we first recall how  the expectation of diagonal overlaps is accessible, following \cite{ChaMeh1998,ChaMeh2000,MehCha1998},\cite{WalSta2015}.  We then
prove Theorem \ref{thm:diag}.

The following was proved by Chalker and Mehlig for $z=0$. They gave convincing arguments for any $z$ in the bulk, a result then proved by Walters and Starr. Explicit formulae have also been recently obtained in \cite{Akemannetal} for the conditional expectation of diagonal and off-diagonal overlaps with respect to any number of eigenvalues. We include the following statement and its short proof for the sake of completeness.

\begin{theorem}[from \cites{ChaMeh1998,ChaMeh2000,MehCha1998,WalSta2015}]\label{meandiag} For any $z \in \mathbb{D}$, we have
\begin{align*} \E \left( \mathscr{O}_{11}\mid \lambda_1=z \right)&  =  N  (1-|z|^2) + \OO\left(\sqrt{N}\frac{e^{- c(z) N }}{1-|z|^2}\right)) \qquad \text{where} \ c(z)= |z|^2-1- \log (|z|^2)>0.\end{align*}
\end{theorem}

\begin{proof}
From (\ref{eqn:O11Exp}), we can write
$$\E\left(\mathscr{O}_{11}\mid  \lambda_1=z\right)= \E \left(\prod_{k=2}^{N} \left(1+\frac{1}{ N| z - \lambda_k|^2} \right) \mid \lambda_1= z\right). $$
Theorem \ref{cond1} with $g(\lambda) =1+\frac{1}{ N|z - \lambda|^2}$ then gives
$$\E \left( \prod_{k=2}^{N} g (\lambda_k) \mid  \lambda_1 = z \right)= {\det (f_{ij})_{1\leq i,j\leq N-1} \over Z_N^{(z)}} \ \text{where} \ f_{ij} = \frac{1}{ i!} \int \lambda^{i-1} \bar{\lambda}^{j-1} \left(\frac{1}{N}+ |z -\lambda |^2 \right)  \mu(\dd \lambda).$$
This is the determinant of a tridiagonal matrix, with entries (we use (\ref{eqn:momentgam}))
$$
f_{ii}=\frac{1}{N^i}+\frac{N^{-1}+|z|^2}{iN^{i-1}},\ 
f_{i,i+1}=-\frac{\overline{z}}{N^i},\ 
f_{i,i-1}=- \frac{z}{i N^{i-1}}.
$$
Denoting $x=N |z|^2$ and $d_k=\det((M_{ij})_{1\leq i,j\leq k})$, with the convention $d_0=1$ we have
\begin{align*}
d_1&=\frac{x+2}{N},\\
d_k&=\left(1+\frac{x+1}{k}\right)\frac{1}{N^k}d_{k-1}-\frac{x}{k}\frac{1}{N^{2k-1}}d_{k-2},
\end{align*}
so that $a_k=d_k N^{\frac{k(k+1)}{2}}$ satisfies $a_0=1$, $a_1=x+2$,
$$
a_k=\left(1+\frac{x+1}{k}\right)a_{k-1}-\frac{x}{k}a_{k-2}.
$$
This gives $a_k=(k+1)e^{(k+1)}(x)-x e^{(k)}(x)$ by an immediate induction.
Thus, we conclude 
\begin{equation}\E\left(\mathscr{O}_{11}\mid \lambda_1=z\right) = N\frac{e^{(N)}(x)}{e^{(N-1)}(x)}-x.\label{eqn:exact}\end{equation}
From the asymptotics
$$e^{(N)}(x)=e^x-\sum_{\ell>N}\frac{x^\ell}{\ell!}=e^x+\OO\left(
\frac{x^N}{N!(1-|z|^2)}\right)=e^x\left(1+\OO\left(\frac{e^{N(1-|z|^2+\log |z|^2)}}{\sqrt{2\pi N}(1-|z|^2)}\right)
\right),$$
the expected formula follows.
\end{proof}

%
%
%
%\begin{lemma}\label{BBC} For any $\epsilon < 1/4$, let $(a_{iN})_{i=1}^{N}$ be an array of coefficients satisfying the condition
%$$ \forall N, \ \forall i \leq N \quad | a_{iN} | < C N^{\epsilon} $$
%and let $(X_{iN})$ be an array of i.i.d. random variables with mean $0$, variance $\sigma_2$ and fourth moment $\sigma_4 < \infty$.
%$$ Y_N:=\frac{1}{ N} \sum_{i=1}^N a_{iN} X_{iN} \xrightarrow[]{a.s.} 0 $$
%\end{lemma}
%
%\begin{proof} The classical fourth-moment technique gives
%$$
%\E Y_N^4 = N^{-4} \sum_{i=1}^N a_{iN}^{4} \sigma_4 + 2 N^{-4} \sum_{i \neq j} a_{iN}^{2} a_j^{2} \sigma_2^2 
%\leq N^{-3+4\epsilon} \sigma_4 + 2 N^{-2+4\epsilon} \sigma_2^2 = \OO(N^{-2+4 \epsilon}).
%$$
%The Markov inequality gives
%$$\forall \eta>0 \quad \PP[Y_N^4 > \eta] \ \leq \ \eta^{-1} \E Y_N^4 \ = \ O(N^{-2+4 \epsilon}) $$
%which is summable, as $\epsilon < 1/4$. We then boldly use Borel-Cantelli:
%$$\forall \eta>0 \quad \PP \big[ \ \limsup_N \{ Y_N^4 > \eta \} \ \big] =0$$
%from which we deduce a.s. convergence.   \end{proof} 
%

The following proofs make use of special integrability when the conditioned particle is  at the center, together with 
a separation of the short-range and long-range eigenvalues. This separation of scales idea is already present in \cite{ChaMeh2000}, though not rigorous. 
To illustrate the main ideas, we first give an alternative proof of Theorem \ref{meandiag} (with deteriorated error estimate) which does not rely on explicit formulas, but rather on rigidity and translation invariance. We then prove the main result of this section, Theorem \ref{thm:diag}.

\begin{proof}[Alternative proof of Theorem \ref{meandiag}] We denote $\nu_z$ the measure (\ref{eqn:partitionfunction}) conditioned to $\lambda_1=z$. Note that $\nu_z$ is a determinantal measure as it has density $\prod_{2\leq i<j\leq N}|\la_i-\la_j|^2e^{-\sum_2^NV(\lambda_i)}$ for some external potential $V$ (which depends on $z$). With a slight abuse of language, we will abbreviate $\E_{\nu_z}(X)=\E(X\mid \la_1=z)$
even for $X$ a function of the overlaps.

The proof consists in three steps: we first show that we can afford a small cutoff of our test function around the singularity, then we decompose our product into smooth long-range and a short-range parts. The long range concentrates, and the short range is invariant by translation.\\

\noindent{\it First step: small cutoff.} Let $g_z(\lambda)=1+ \frac{1}{N|z-\lambda|^2}$. Remember that from (\ref{eqn:O11Exp}),
$
\E_{\nu_z}(\mathscr{O}_{11}) = \E_{\nu_z}\left(\prod_{i=2}^{N} g_z(\lambda_i)\right).
$
We denote $h_z(\lambda)=g_z(\lambda) \mathds{1}_{{|z-\lambda|}>N^{-A}}$, with $A=A(\kappa)$ a large enough chosen constant, and first prove the following elementary equality:
\begin{equation}\label{eqn:elementary}
\E_{\nu_z}\left(\prod_{i=2}^{N} g_z(\lambda_i)\right)= \E_{\nu_z}\left(\prod_{i=2}^{N} h_z(\lambda_i)\right)+\OO(N^{-3}).
\end{equation}
Note that the exponent $N^{-3}$ here is just chosen for the sake of concreteness. 
The left hand side coincides with
$
\E_{\nu_z}\left(\prod_{i=2}^N(|z-\lambda_i|^2+N^{-1})\right)\left(\E_{\nu_z}\left(\prod_{i=2}^N(|z-\lambda_i|^2)\right)\right)^{-1},
$
so that (\ref{eqn:elementary}) follows if we can prove
\begin{align}
&\E_{\nu_z}\left(\prod_{i=2}^N(|z-\lambda_i|^2+N^{-1})e^{-N(|z|^2-1)})\right)
-
\E_{\nu_z}\left(\prod_{i=2}^N(|z-\lambda_i|^2+N^{-1}\mathds{1}_{|z-\lambda_i|>N^{-A}})e^{-N(|z|^2-1)}\right)=\OO(N^{-B}),
\label{eqn:elem1}
\\
&\E_{\nu_z}\left(\prod_{i=2}^N(|z-\lambda_i|^2)e^{-N(|z|^2-1)}\right)\geq N^{-C}\label{eqn:elem2},
\end{align}
for a constant $B$ sufficiently larger than $C$.
Equation (\ref{eqn:elem2}) follows from Lemma \ref{lem:generalmoments}.
For equation  (\ref{eqn:elem1}), note that the left hand side has size order
$$
\E_{\nu_z}\left(\prod_{i=2}^N(|z-\lambda_i|^2+N^{-1})e^{-N(|z|^2-1)})\mathds{1}_{\exists i: |\lambda_i-z|<N^{-A}}\right)
=\OO(N^{-A/10})\E_{\nu_z}\left(\prod_{i=2}^N(|z-\lambda_i|^2+N^{-1})^2e^{-2N(|z|^2-1)}\right)^{1/2},
$$
by Cauchy-Schwarz inequality, union bound, and considering that $A$ can be taken as large as needed.
This last expectation is bounded by Lemma \ref{lem:generalmoments}, which concludes the proof of
(\ref{eqn:elementary}) by choosing $A$ large enough.\\

\noindent{\it Second step:  the long-range contribution concentrates.} 
We smoothly separate the short-range from a long-range contributions on the right hand side of (\ref{eqn:elementary}). For this, we define:
\begin{align}
\chi_{z,\delta} (\la)& = \chi \Big(N^{\frac{1}{2}-\delta}|z-\la| \Big) \quad \text{with} \quad \delta \in (0, \kappa) \label{eqn:chii}\\
f^{\ell}_z(\la)& =\frac{1}{ N |z - \lambda|^2}  (1 - \chi_{z,\delta} (\lambda)) \\
\bar f_z & =(N-1)\int_{\mathbb{D}} {1 - \chi_{z,\delta} (\lambda)  \over N|z - \lambda|^2} \frac{\dd m(\lambda)}{\pi} \\
h_z(\lambda)&=e^{h_z^{s}(\lambda)+h_z^{\ell}(\lambda)},\notag\\
h_z^{s}(\lambda)&=  \log \Big( 1 + \frac{1}{ N |z - \la|^2}\mathds{1}_{{|z-\lambda|}>N^{-A}} \Big) \chi_{z,\delta} (\la),\label{eqn:short}\\
h_z^{\ell}(\la)&=  \log \Big( 1 + \frac{1}{ N |z - \la|^2} \Big) (1 - \chi_{z,\delta} (\la)),\notag
\end{align}
Note that
\begin{equation}\label{eqn:est}
\left|\sum_{i=2}^Nh^\ell_z(\lambda_i)-\bar f_z\right|\leq \left|\sum_{i=2}^Nf^\ell_z(\lambda_i)-\bar f_z\right|+\sum_{i=2}^N\frac{1}{N^2|z-\lambda_i|^4}(1 - \chi_{z,\delta} (\lambda_i)).
\end{equation}
To bound the first term on the right hand side, 
we rely on  \cite[Lemma 3.2]{BreDui2014}: for any $\alpha$ such that $\alpha \|f^{\ell}_z\|_\infty<1/3$ 
(in practice $\|f^{\ell}_z\|_\infty<N^{-2\delta}$ so that we will choose $\alpha=c N^{2\delta}$ for some fixed small $c$), we have
\begin{equation}\label{eqn:concentration}
\E_{\nu_z}\left(e^{\alpha(\sum_{i=2}^Nf^\ell_z(\lambda_i)-\E_{\nu_z}\left(\sum_{i=2}^Nf^\ell_z(\lambda_i)\right))}\right)\leq e^{C\alpha^2 {\rm Var}_{\nu_z}(\sum_{i=2}^N f^\ell_z(\lambda_i))}
\end{equation}
for some $C$ which does not depend on $N$. We first bound the above variance. Introduce a partition of type $1=\chi+\sum_{k\geq 1}\xi(2^{-k}x)$ for any $x>0$, with $\xi$ smooth, compactly supported.
Let $f^{\ell}_{z,k}(\lambda)=f^{\ell}_{z}(\lambda)\xi(2^{-k}N^{1/2-\delta}|z-\la|)$ and $K=\min\{k\geq 1:2^kN^{-1/2+\delta}>C\}$ where $C$ and therefore $K$ only depend on $\xi$. Then $\sum_if^{\ell}_z(\lambda_i)=\sum_{k=1}^{K}\sum_if^{\ell}_{z,k}(\lambda_i)$ with probability $1-e^{-cN}$ (here we use that there are no eigenvalues $|\lambda_k|>1+\e$ with probability $1-e^{-c(\e)N}$, thanks to the Corollary $\ref{Kost1}$). Moreover,  from \cite[Theorem 1.2]{BouYauYin2014II}, for any  $\e>0$ and  $D>0$, there exists $N_0>0$ such that for any $N\geq N_0$, $|z|<1$ and   $1\leq k\leq K$ we have
$$
\mathbb{P}_{N-1}\left(
\left|\sum f^{\ell}_{z,k}(\lambda_i)-(N-1)\int_{|\la|<1} f^{\ell}_{z,k}\right|>N^{-2\delta+\e}
\right)
\leq N^{-D}.
$$
This implies
the same estimate for the conditioned measure by a simple Cauchy-Schwarz inequality:
\begin{multline*}
\mathbb{P}_{\nu_z}\left(
\left|\sum f^{\ell}_{z,k}(\lambda_i)-(N-1)\int_{|z|<1} f^{\ell}_{z,k}\right|>N^{-2\delta+\e}
\right)\\
\leq N^C
\mathbb{P}_{N-1}\left(
\left|\sum f^{\ell}_{z,k}(\lambda_i)-(N-1)\int_{|z|<1} f^{\ell}_{z,k}\right|>N^{-2\delta+\e}
\right)^{1/2}
\E_{\nu_z}\left(\prod_{i=2}^N|z-\lambda_i|^4e^{-2N(|z|^2-1)}\right)^{1/2}\leq N^{-\frac{D}{2}+2C}.
\end{multline*}
where we used  Lemma \ref{lem:generalmoments}. We conclude that for any  $\e>0$ and  $D$ we have
$$
\mathbb{P}_{\nu_z}\left(
|\sum f^{\ell}_z(\lambda_i)-\bar f_z|>N^{-2\delta+\e}
\right)\leq N^{-D},
$$
so that ${\rm Var}_{\nu_z}(\sum_{i=2}^N f_z^\ell(\lambda_i))=\OO(N^{-4\delta+\e})$ and $\mathbb{E}_{\nu_z}(\sum_{i=2}^N f_z^\ell(\lambda_i))=\bar f_z+\OO(N^{-2\delta+\e})$.
As a consequence, (\ref{eqn:concentration}) becomes 
\begin{equation}\label{eqn:est1}
\E_{\nu_z}\left(e^{\alpha(\sum_{i=2}^Nf^\ell_z(\lambda_i)-\bar f_z)}\right)\leq e^{C\alpha^2 N^{-4\delta+\e}+C|\alpha| N^{-2\delta+\e}}.
\end{equation}
The same reasoning yields 
\begin{equation}\label{eqn:est2}
\E_{\nu_z}\left(e^{\alpha\sum_{i=2}^N\frac{1}{N^2|z-\lambda_i|^4}(1-\chi_{z,\delta}(\lambda_i))}\right)\leq e^{C\alpha^2 N^{-8\delta+\e}+C|\alpha| N^{-2\delta+\e}}.
\end{equation}
The choice $\alpha=\pm c N^{2\delta}$ in (\ref{eqn:est1}), (\ref{eqn:est2}) together with (\ref{eqn:est}) implies
$$
\mathbb{P}_{\nu_z}(A)<e^{-c N^{\e}},\ \mbox{where}\  A=\left\{
\left|
\sum_{i=2}^Nh^\ell_z(\lambda_i)-\bar f_z
\right|>N^{-2\delta+\e}
\right\}.
$$
This yields, for some $p=1+c(\kappa)$, $c(\kappa)>0$ and some $q,r>1$,
\begin{equation}\label{eqn:conc}
\mathbb{E}_{\nu_z}\left(e^{\sum_i(h^s_z(\lambda_i)+h^\ell_z(\lambda_i))}\mathds{1}_{A}\right)
\leq 
\mathbb{E}_{\nu_z}\left(e^{p\sum_ih^s_z(\lambda_i)}\right)^{1/p}
\mathbb{E}_{\nu_z}\left(e^{q\sum_i h^\ell_z(\lambda_i)}\right)^{1/q}
\mathbb{P}_{\nu_z}(A)^{1/r}\leq e^{-c N^\e}.
\end{equation}
Here we used that the third term has size order $e^{-c N^{\e}}$, the second one is of order  $e^{q\bar f_z}=\OO(N^{C})$ from  (\ref{eqn:est1}), (\ref{eqn:est2}), and so is the first one from Lemma \ref{lem:moment} (we needed the initial small cutoff changing $g$ into $h$ in order to apply this Lemma).
Moreover,
\begin{multline*}
\mathbb{E}_{\nu_z}\left(e^{\sum_i(h^s_z(\lambda_i)+h^\ell_z(\lambda_i))}\mathds{1}_{A^c}\right)
=
(1+\OO(N^{-2\delta+\e}))\mathbb{E}_{\nu_z}\left(e^{\sum_ih^s_z(\lambda_i)+\bar f_z}\right)-
(1+\OO(N^{-2\delta+\e}))\mathbb{E}_{\nu_z}\left(e^{\sum_ih^s_z(\lambda_i)+\bar f_z}\mathds{1}_{A}\right),
\end{multline*}
and this last expectation is of order $e^{-c N^{\e}}$ for the same reason as (\ref{eqn:conc}). To summarize, with the previous two equations we have proved (up to exponentially small error terms)
$$
\mathbb{E}_{\nu_z}\left(\prod_{i=2}^Nh_z(\lambda_i)\right)
=
(1+\OO(N^{-2\delta+\e}))e^{\bar f_z}\ \mathbb{E}_{\nu_z}\left(e^{\sum_ih^s_z(\lambda_i)}\right)
.
$$
%From (\ref{eqn:scales}) and Cauchy-Schwarz,
%we have
%$$
%\E(\mathscr{O}_{11}\mid\lambda_1=z) 
%=
%\E\left(e^{\sum_{n=2}^N h^{\text{loc}}(\lambda_k)+(N-1)\bar h}\right)
%+
%$$
%
%
%%Let us define $a_{kN} = {1 - \chi_{\epsilon} (\lambda_k) \over |z_1 - \lambda_k|^2} < N^{1-2 \epsilon}$. We were careful to take $\epsilon>{3 \over 8}$ so that $1-2\epsilon < \frac{1}{ 4}$, and $ |X_k|^2 $ are independent $\gamma_1$ variable with finite moments; we can therefore use lemma  \ref{BBC} and state the convergence 
%%$$ \frac{1}{ N} \sum_{k=1}^N  {1 - \chi_{\epsilon} (\lambda_k) \over |z_1 - \lambda_k|^2}  |X_k|^2 \xrightarrow[]{a.s.} \frac{1}{ N} \sum_{k=1}^N  {1 - \chi_{\epsilon} (\lambda_k) \over |z_1 - \lambda_k|^2}.$$
%%The error term converges as $1-2\epsilon < \frac{1}{ 4}$. 
%
%We then use rigidity (see appendix D) to justify
%$$ \sum_{k=2}^N    {1 - \chi_{z, \theta}(\lambda_k) \over N |z - \lambda_k|^2}  = \frac{1}{ \pi} \int_{\mathbb{D}} {1 - \chi \big({ { |\lambda - z|^2 \over \theta^2} \big)} \over |z - \lambda|^2} \dd m(\lambda) + \OO_{\prec} (N^{-2 \epsilon})$$
%and lemma \ref{integral1} gives us the asymptotics of this integral. We conclude
%$$\sum_{k=1}^N h_k^{\text{glob}} =  (1- 2 \epsilon) \log (N) + \log (1 - |z|^2 )  + C +O(?). $$
%$$e^{\sum_{k=1}^N h_k^{\text{glob}}} = e^C (1 - |z|^2 ) N^{1-2\epsilon} $$

\noindent{\it Third step: the local part is invariant.} With $p=1$ in (\ref{eqn:useful}), we have
$$
\mathbb{E}_{\nu_z}\left(e^{\sum_ih_z^s(\lambda_i)}\right)=\mathbb{E}_{\nu_0}\left(e^{\sum_ih_0^s(\lambda_i)}\right)+\OO\left(e^{-c N^{2\kappa}}\right).
$$
This yields
$$
\mathbb{E}_{\nu_z}\left(\prod_{i=2}^Nh_z(\lambda_i)\right)
=
(1+\OO(N^{-2\delta+\e}))e^{\bar f_z-\bar f_0}\ \mathbb{E}_{\nu_0}\left(\prod_{i=2}^Nh_0(\lambda_i)\right).$$
From Lemma \ref{integral1}, $e^{\bar f_z-\bar f_0}=1-|z|^2$, and from
(\ref{eqn:expec}), (\ref{eqn:elementary}) we have  $\mathbb{E}_{\nu_0}\left(\prod_{i=2}^Nh_0(\lambda_i)\right)=N+\OO(N^{-2})$.
This concludes the proof.
\end{proof}

\begin{proof}[Proof of Theorem \ref{thm:diag}]
We follow the same method as in the previous proof, except that we won't need a small a priori cutoff: we are interested in convergence in distribution, not in ${\rm L}^1$.\\

\noindent{\it First step:  the long-range contribution concentrates.} 
We smoothly separate the short-range from a long-range contributions in Theorem \ref{thm:quenched}:
\begin{align}
 \mathscr{O}_{11} &\overset{(\rm d)}{=}e^{\sum_2^N g^s_z(\lambda_i,X_i)+\sum_2^N g^\ell_z(\lambda_i,X_i)},\label{eqn:decompgener}\\
g_z^{s}(\lambda,x)&=  \log \Big( 1 + {|x|^2 \over N |z - \la|^2} \Big) \chi_{z,\delta} (\la),\notag\\
g_z^{\ell}(\la,x)&=  \log \Big( 1 + {|x|^2 \over N |z - \la|^2} \Big) (1 - \chi_{z,\delta} (\la))\notag,
\end{align}
For the convenience of the reader we recall the notations defined above :
\begin{align*}
\chi_{z,\delta} (\la) &= \chi \Big(N^{\frac{1}{2}-\delta}|z-\la| \Big) \quad \text{with} \quad  \delta \in (0, \kappa) \\
f^{\ell}_z(\la,x) & ={|x|^2 \over N |z - \lambda|^2}  (1 - \chi_{z,\delta} (\lambda)) \\
\bar f_z & =(N-1)\int_{\mathbb{D}} {1 - \chi_{z,\delta} (\lambda)  \over N|z - \lambda|^2} \frac{\dd m(\lambda)}{\pi}
\end{align*} 
Let $\mathcal{G}$ be the distribution of $(X_2,\dots,X_N)$. For any $\e>0$, by Gaussian tail we have
$$
\mathcal{G}(B)\geq 1-e^{-cN^{\e/10}},\ \mbox{where}\ B=\{|X_i|^2\leq N^{\e/10}\ \mbox{for all}\ 2\leq i\leq N\}.
$$
Moreover,
\begin{equation}\label{eqn:estII}
|\sum_{i=2}^Ng^\ell_z(\lambda_i,X_i)-\bar f_z|\mathds{1}_{B}\leq |\sum_{i=2}^Nf^\ell_z(\lambda_i,X_i)-\bar f_z|+C N^{\e/2}\sum_{i=2}^N\frac{1}{N^2|z-\lambda_i|^4}(1 - \chi_{z,\delta} (\lambda_i)).
\end{equation}
To bound the first term on the right hand side, we first integrate over the Gaussian variables:
\begin{multline*}
\mathbb{E}_{\nu_z\times \mathcal{G}}\left(
e^{\alpha (\sum_{i=2}^N f_z^{\ell}(\lambda_i,X_i)-\bar f_z)}
\right)
=
\mathbb{E}_{\nu_z}\left(
\prod_{i=2}^N
\frac{1}{1-\alpha\frac{1-\chi_{z,\delta}(\lambda_i)}{N|z-\la_i|^2}}
e^{-\alpha \bar f_z}
\right)\\
\leq 
\mathbb{E}_{\nu_z}\left(
e^{\alpha \left(\sum_{i=2}^N f_z^{\ell}(\lambda_i)-\bar f_z\right)+C\alpha^2\sum_{i=2}^N\frac{1}{N^2|z-\lambda_i|^4}(1 - \chi_{z,\delta} (\lambda_i))^2}
\right)\leq e^{C\alpha^2 N^{-4\delta+\e}+C\alpha N^{-2\delta+\e}},
\end{multline*}
where, for the last inequality, we used (\ref{eqn:est1}) and (\ref{eqn:est2}) together with the Cauchy-Schwarz  inequality, with $\alpha=c N^{2\delta}$ for some fixed small enough $c$ being admissible. With  (\ref{eqn:estII}), we obtain
\begin{equation}\label{eqn:goodset}
\mathbb{P}_{\nu_z\times \mathcal{G}}(A)<e^{-c N^{\e}},\ \mbox{where}\  A=\left\{
\left|
\sum_{i=2}^N g_z^{\ell}(\lambda_i,X_i)-\bar f_z
\right|>N^{-2\delta+\e}
\right\}.
\end{equation}
Let $\xi\in \mathbb{R}$ be fixed.  From the above bound we have 
\begin{multline*}
\mathbb{E}\left(\left(\frac{\mathscr{O}_{11}}{N(1-|z|^2)}\right)^{\ii\xi}\mid \la_1=z\right)=
\mathbb{E}_{\nu_z\times \mathcal{G}}\left(\left(\frac{e^{\sum_2^N g^s_z(\lambda_i,X_i)+\sum_2^N g^\ell_z(\lambda_i,X_i)}}{N(1-|z|^2)}\right)^{\ii\xi} \mathds{1}_{A^c}\right)+\OO(e^{-c N^{\e}})\\
=
\mathbb{E}_{\nu_z\times \mathcal{G}}\left(\left(\frac{e^{\sum_2^N g^s_z(\lambda_i,X_i)+\bar f_z}}{N(1-|z|^2)}\right)^{\ii\xi} \mathds{1}_{A^c}\right)+\OO\left(|\xi| N^{-2\delta+\e}\right)
=
\mathbb{E}_{\nu_z\times \mathcal{G}}\left(\left(\frac{e^{\sum_2^N g^s_z(\lambda_i,X_i)+\bar f_0}}{N}\right)^{\ii\xi} \right)+\OO(|\xi| N^{-2\delta+\e}),
\end{multline*}
where we used $e^{\bar f_z}/(1-|z|^2)=e^{\bar f_0}$, from Lemma \ref{integral1}.

We now define the function $a_z$ (omitting the dependence in $\xi$ in the notation) through
$$
e^{a_z(\la)}=\mathbb{E}_{\mathcal{G}}\left(e^{\ii\xi g_z^s(\la,X_i)}\right).
$$
Note that $a_z$ does not depend on $i$ because the $X_i$'s are identically distributed.
We want to apply Lemma  \ref{lem:invariance2}.
Note that $a_z$ is supported on $|z-\la|<C N^{-\frac{1}{2}+\delta}$ and $\re(a_z)\leq 0$, so that (\ref{eqn:strange2}) and (\ref{eqn:strange3}) are automatically satisfied and  $(N\|\nu\|_1)^r\leq C N^{2r\delta}$, 
hence (\ref{eqn:strange1}) holds for the choice $r=3$, $\delta=\kappa/10$. For this choice of $\delta$, we therefore have
\begin{equation}\label{eqn:FourierGood}
\mathbb{E}_{\nu_z}\left(e^{\sum_{i=2}^N a_z(\la_i)}\right)=
\mathbb{E}_{\nu_0}\left(e^{\sum_{i=2}^N a_0(\la_i)}\right)+\OO\left(e^{-cN^{2\kappa}}\right)
\end{equation}
uniformly in  $\xi$.
This proves
\begin{multline}
\mathbb{E}\left(\left(\frac{\mathscr{O}_{11}}{N(1-|z|^2)}\right)^{\ii\xi}\mid \la_1=z\right)=
\mathbb{E}_{\nu_0\times \mathcal{G}}\left(\left(\frac{e^{\sum_2^N g^s_0(\lambda_i,X_i)+\bar f_0}}{N}\right)^{\ii\xi} \right)+\OO(|\xi| N^{-2\delta+\e})\label{eqn:identity1}\\
=\mathbb{E}\left(\left(\frac{\mathscr{O}_{11}}{N}\right)^{\ii\xi}\mid \la_1=0\right)+\OO(|\xi| N^{-2\delta+\e}).
\end{multline}
Together with Proposition \ref{prop:at0}, this concludes the proof of Theorem \ref{thm:diag}.
\end{proof}

\begin{proof}[Proof of Corollary \ref{cor:extremes}]
We start with the lower bound. 
From (\ref{eqn:goodset}) we have
$$
\mathbb{P}\left(
\frac{\mathscr{O}_{11}}{1-|z|^2}<N^{1-\e}\mid \la_1=z
\right)
\leq 
\mathbb{P}_{\nu_z\times \mathcal{G}}\left(
e^{\sum_2^N g^s_z(\lambda_i,X_i)+\bar f_0}
<N^{1-\frac{\e}{2}}
\right)
+\OO(e^{-c N^\e}).
$$
We now apply Lemma \ref{lem:FourierInversion} to justify that $z$ can essentially be replaced by $z=0$ in the above left hand side. From (\ref{eqn:FourierGood}), the Fourier transforms of $g_z,g_0$ are exponentially close uniformly in the Fourier parameter $\xi$. By choosing in Lemma \ref{lem:FourierInversion} $R=T=e^{N^{\e/10}}$ and $F$ smooth bounded equal to $1$ on  the interval $(-\infty,(1-\e/2\log N-\bar f_0)]$, $0$ on $[(1-\e/2\log N-\bar f_0+1,\infty)$, we have
\begin{multline}
\mathbb{P}_{\nu_z\times \mathcal{G}}\left(
e^{\sum_2^N g^s_z(\lambda_i,X_i)+\bar f_0}
<N^{1-\frac{\e}{2}}
\right)\leq 
\mathbb{P}_{\nu_0\times \mathcal{G}}\left(
e^{\sum_2^N g^s_0(\lambda_i,X_i)+\bar f_0}
<N^{1-\frac{\e}{4}}
\right)+\OO(e^{-N^{\e/10}})\\
\leq
\mathbb{P}\left(\mathscr{O}_{11}<N^{1-\e/8}\mid \la_1=0
\right)+\OO(e^{-N^{\e/10}})=\OO(e^{-N^{\e/10}}),\label{eqn:boundby0}
\end{multline}
where this last probability  was estimated thanks to Proposition \ref{prop:at0}. For $|z|<1-N^{-\frac{1}{2}+\kappa}$, this yields 
$$
\mathbb{P}\left(
\mathscr{O}_{11}<N^{-\frac{1}{2}+\kappa}N^{1-\e}\mid \la_1=z
\right)=\OO(e^{-N^{\e/10}}),
$$
and we conclude by a union bound (an error bound $\oo(N^{-1})$ above would be enough).
For the upper estimate, in the same way as previously, for any $x>0$ we obtain
\begin{equation}\label{eqn:tail}
\mathbb{P}\left(
\frac{\mathscr{O}_{11}}{1-|z|^2}>x\mid \la_1=z
\right)
\leq
\mathbb{P}\left(
\mathscr{O}_{11}>x/2\mid \la_1=0
\right)+\OO(e^{-N^{\e/10}}).
\end{equation}
For $x\gg N$, the following is easy to justify:
$$
\mathbb{P}\left(
\mathscr{O}_{11}>x\mid \la_1=0
\right)
=
N(N-1)\int_0^{1/x}u(1-u)^{N-2}
\rd u
=
\frac{N-1}{N}\int_{x/N}^\infty\left(1-\frac{1}{Nt}\right)^{N-2}\frac{\rd t}{t^3}
\sim
\int_{x/N}^\infty\frac{\rd t}{t^3}=\frac{N^2}{2 x^2}.
$$
We obtained
\begin{multline*}
\sum_{i=1}^N\mathbb{P}\left(\lambda_i\in\Omega_N,\mathscr{O}_{ii}\geq N^{1+\kappa_0+\e}m(\Omega_N)^{1/2}\right)
=N\mathbb{P}\left(\lambda_i\in\Omega_N\right)
\mathbb{P}\left(\mathscr{O}_{11}\geq N^{1+\kappa_0+\e}m(\Omega_N)^{1/2}\mid \lambda_1\in\Omega_N\right)\\
\leq 
N\mathbb{P}\left(\lambda_i\in\Omega_N\right)
\mathbb{P}\left(\frac{\mathscr{O}_{11}}{1-|\lambda_1|^2}\geq N^{\frac{3}{2}+\e}m(\Omega_N)^{1/2}\mid \lambda_1\in\Omega_N\right)
\leq N m(\Omega_N)\frac{N^2}{(N^{\frac{3}{2}+\e}m(\Omega_N)^{1/2})^2}\leq N^{-2\e},
\end{multline*}
which concludes the proof by a union bound.
\end{proof}

\begin{remark} One may wonder about the true asymptotics of the greatest overlap over the whole spectrum.
The above bounds could easily be refined to prove that for any $C>0$ and $ N\ll x\ll N^C$,
$$
\sum_{i=1}^N\mathbb{P}\left(\mathscr{O}_{ii}\geq x\right)\sim N\int_{\mathbb{D}}\frac{N^2(1-|z|^2)^2}{2 x^2}\frac{\rd m(z)}{\pi}=\frac{N^3}{6 x^2}.
$$
If the overlaps are sufficiently independent (a fact suggested by (\ref{eqn:indep})), this hints towards convergence of the maximum to a Fr\'echet distribution: for any fixed $y>0$, as $N\to\infty$
$$\mathbb{P}\left(\max_{1\leq i\leq N}\frac{\mathscr{O}_{ii}}{N^{3/2}}<y\right)\to e^{-\frac{1}{6y^2}}.$$
\end{remark}

Remember that
$0 \leq \chi \leq 1$
is a  smooth cut-off function  on $\mathbb{R}_+$ such that
$\chi(x)=1$ for $x<1/2$ and $0$ for $x>1$, and we denote
$$\chi_{z,\delta} (\la)= \chi \Big( N^{\frac{1}{2}-\delta}|z-\la|\Big).$$

The following three lemmas were used in the previous proofs.
\begin{lemma}\label{integral1}  There exists a constant $c(\chi)$ such that for any $|z|<1-N^{-\frac{1}{2}+\delta}$ we have
$$ \frac{1}{ \pi} \int_{\mathbb{D}} {1 - \chi_{z,\delta} (\lambda)  \over |z - \lambda|^2} \dd m(\lambda) =  (1- 2 \delta) \log (N) + \log (1 - |z|^2 )  + c(\chi). $$
\end{lemma}

\begin{proof}
For $z=0$, this is an elementary calculation in polar coordinates, so that we only need to show that for any given $0<\e<1-|z|$ we have (here $\mathbb{D}_a$ is the disk with center $a$, radius $\e$)
\begin{equation}\label{eqn:aim22}
 \frac{1}{ \pi} \int_{\mathbb{D}-\mathbb{D}_z} {1  \over |z - \lambda|^2} \dd m(\lambda) - \frac{1}{ \pi} \int_{\mathbb{D}-\mathbb{D}_0} {1  \over |\lambda|^2} \dd m(\lambda) =\log(1-|z|^2).
\end{equation}
Denote $x+\ii y=r e^{\ii\theta}$ and $a=|z|$. Note that $\frac{1}{r^2}=\partial_x(x\frac{\log r}{r^2})+\partial_y(y\frac{\log r}{r^2})$,
so that by Green's theorem we have
$$
\frac{1}{ \pi} \int_{\mathbb{D}-\mathbb{D}_z} {1  \over |z - \lambda|^2} \dd m(\lambda)=\frac{1}{ \pi} 
\left(\int_{\partial\mathbb{D}}-\int_{\partial\mathbb{D}_z}\right)\frac{\log|re^{\ii\theta}-a|}{|re^{\ii\theta}-a|^2}\left((x-a)\rd y-y\rd x\right).
$$
The second integral  clearly does  not depend on $a$. The first integral can be split into
$$
\frac{1}{ \pi} \int_{\partial\mathbb{D}}\frac{\log|re^{\ii\theta}-a|}{|e^{\ii\theta}-a|^2}\rd\theta
-
\frac{a}{ 2\pi} \int_{\partial\mathbb{D}}\frac{\log|re^{\ii\theta}-a|}{|e^{\ii\theta}-a|^2}(e^{\ii\theta}+e^{-\ii\theta})\rd\theta.
$$
To calculate the first integral above, we expand $\log|e^{\ii\theta}-a|=\Re\sum_{p\geq 1}\frac{1}{p}(ae^{-\ii\theta})^p$,
$|e^{\ii\theta}-a|^{-2}=\sum_{k,\ell\geq 0}a^{k+\ell}e^{\ii(k-\ell)\theta},$ and obtain
$$
\frac{1}{ \pi} \int_{\partial\mathbb{D}}\frac{\log|re^{\ii\theta}-a|}{|e^{\ii\theta}-a|^2}\rd\theta=2\sum_{p\geq 1,k=p+\ell}\frac{a^{p+k+\ell}}{p}=2\frac{\log(1-a^2)}{1-a^2}.
$$
In the same way, we have
$$
\frac{a}{ 2\pi} \int_{\partial\mathbb{D}}\frac{\log|re^{\ii\theta}-a|}{|e^{\ii\theta}-a|^2}(e^{\ii\theta}+e^{-\ii\theta})\rd\theta
=
a\left(\sum_{p\geq 1,k+1=p+\ell}+\sum_{p\geq 1,k-1=p+\ell}\right)\frac{a^{k+\ell+p}}{p}
=
\frac{\log(1-a^2)}{1-a^2}+a^2\frac{\log(1-a^2)}{1-a^2}
$$
To summarize we have proved that
$$
 \frac{1}{ \pi} \int_{\mathbb{D}-\mathbb{D}_z} {1  \over |z - \lambda|^2} \dd m(\lambda)=\log(1-a^2)+c
$$
where $c$ does not depend on $z$,
and (\ref{eqn:aim22}) follows.
\end{proof}

\begin{lemma}\label{lem:moment}
Let $h_z^s$ be given by (\ref{eqn:short}).
For any $\kappa\in(0,1/2)$, there exists $c(\kappa),\, C(\kappa)>0$ such that for any $|z|<1-N^{-\frac{1}{2}+\kappa}$ and $p\in[1,1+c(\kappa)]$, we have
$$
\mathbb{E}_{\nu_z}\left(e^{p\sum_2^N h_z^s(\lambda_i)}\right)\leq N^{C(\kappa)}.
$$
\end{lemma}

\begin{proof}
First, the result is true for $z=0$. Indeed,
\begin{equation}
\mathbb{E}_{\nu_0}\left(e^{p\sum_2^N h_0^s(\lambda_i)}\right)
\leq 
\mathbb{E}_{\nu_0}\left(\prod_{i=2}^N\left(1+\frac{1}{N|\lambda_i|}\right)^p\right)
=\prod_{k=2}^N\E\left(\left(1+\frac{1}{\gamma_k}\right)^p\right)
=\OO(N^{C}),\label{eqn:useful1}
\end{equation}
where we used Corollary \ref{kostbis}.

%\begin{multline}
%\mathbb{E}_{\nu_0}\left(e^{p\sum_2^N h_0^s(\lambda_i)}\right)
%\leq 
%\mathbb{E}_{\nu_0}\left(\prod\left(1+\frac{1}{N|\lambda_i|}\right)^p\right)
%\leq
%\mathbb{E}_{\nu_0}\left(\prod\left(1+\frac{|X_i|^2}{N|\lambda_i|}\right)^p\right)\\
%=
%\mathbb{E}\left(\mathscr{O}_{11}^p\mid\lambda_1=0\right)
%=
%\mathbb{E}\left(\frac{1}{(\beta_{2,N-1})^p}\right)
%=
%\frac{\Gamma(2-p)}{\Gamma(2)}\frac{\Gamma(N+1)}{\Gamma(N+1-p)}=\OO(N^{C}),\label{eqn:useful1}
%\end{multline}
%where the $X_i$'s are independent standard complex Gaussian ($\E(|X_i|^2)=1$), we used Jensen's inequality   $(1+c \E(|X|^2))^p\leq \E((1+c |X|^2)^p)$, and finally the remarkable identity (\ref{eqn:law}).

We want to apply Lemma \ref{lem:invariance2} to conclude the proof. For this we need to check conditions (\ref{eqn:strange1}), (\ref{eqn:strange2}) and (\ref{eqn:strange3}) for our function $f=ph^s_z$, and $\nu=e^f-1$.
First note that
$\|\nu\|_1\leq C\int_{N^{-A}<|\la|<1}\frac{1}{(N|\lambda^2|)^p}\leq C N^{-p}N^{A(2p-2)}$, so that (\ref{eqn:strange1}) holds by choosing $p=1+c(\kappa)$ with $c(\kappa)$ small enough.
To prove (\ref{eqn:strange2}) and (\ref{eqn:strange3}), we rely on Lemma \ref{lem:invariance1}:
$$
 \E_{\nu_z}\left(e^{p\sum_{i=2}^Nh^s_z(\la_i)}\right)= \E_{\nu_z}\left(e^{p\sum_{i=2}^Nh^s_0(\la_i)}\right)(1+\oo(1))\leq
 \mathbb{E}\left(
 \left(
 1+\frac{\mathds{1}_{\gamma_1>N^{-A}}}{\gamma_1}
 \right)^p
 \right)
\prod_{k=2}^N\E\left(\left(1+\frac{1}{\gamma_k}\right)^p\right)\leq N^C
$$
where we used (\ref{eqn:strange1}) for the first equation, and
the calculation 
$
\int_{N^{-A}}^{\infty}
\frac{1}{x^p}e^{-x}\leq 
N^C.
$
We therefore obtained
\begin{equation}\label{eqn:useful}
\mathbb{E}_{\nu_z}\left(e^{p\sum_2^N h_z^s(\lambda_i)}\right)
=
\mathbb{E}_{\nu_0}\left(e^{p\sum_2^N h_z^s(\lambda_i)}\right)
+\OO\left(e^{-c N^{2\kappa}}\right)
\end{equation}
for any $1\leq p\leq 1+c(\kappa)$ with $c(\kappa)$ small enough. Equations (\ref{eqn:useful1}) and (\ref{eqn:useful}) conclude the proof.
\end{proof}

To quantitatively invert the Fourier transform, we use the following crude bound, see \cite[Lemma 2.6]{ArgBelBou2017} following from \cite[Corollary 11.5]{BhaRao76}.
\begin{lemma}
\label{lem:FourierInversion} There exists a constant $c$
such that if $\mu$ and $\nu$ are probability measures on $\mathbb{R}$
with Fourier transforms $\hat{\mu}\left({t}\right)=\int e^{\ii{t}{x}}\mu\left({\rm d}x\right)$
and $\hat{\nu}\left({t}\right)=\int e^{\ii{t}{x}}\nu\left({ d}x\right)$,
then for any $R,T>0$ and any function $f:\mathbb{R}\to\mathbb{R}$
with Lipschitz constant $C$,
\begin{equation}
\left|\mu\left(f\right)-\nu\left(f\right)\right|\le c\frac{C}{T}+c\|f\|_{\infty}\left\{ RT\|\1_{\left(-T,T\right)}\left(\hat{\mu}-\hat{\nu}\right)\|_{\infty}+\mu\left([-R,R]^{c}\right)+\nu\left([-R,R]^{c}\right)\right\} .\label{eq: fourier inversion bound}
\end{equation}
\end{lemma}

The following crude a priori estimates are used in this paper. Note that for $z$ strictly in the bulk of the spectrum ($|z|<1-\e$ for fixed $\e>0$), the first statement is a simple consequence of 
the main result in \cite{WebWon2017}.

\begin{lemma}\label{lem:generalmoments}
For any $p,\kappa>0$, there exists $C>0$ such that, for large enough $N$, uniformly in $|z_1|, |z_2|<1-N^{-\frac{1}{2}+\kappa}$ we have
\begin{align}
&N^{-C}\leq \E_N\left(\prod_{i=1}^N |z_1-\la_i|^{2p}e^{-pN(|z_1|^2-1)}\right)\leq N^{C},\label{eqn:written}\\
&N^{-C}\leq \E_N\left(\prod_{i=1}^N \left(|z_1-\la_i|^{2}+\frac{1}{N}\right)^pe^{-pN(|z_1|^2-1)}\right)\leq N^{C},\label{eqn:written2}\\
&N^{-C}\leq \E_N\left(\prod_{i=1}^N |z_1-\la_i|^2 |z_2-\la_i|^2e^{-N(|z_1|^2-1)}e^{-N(|z_2|^2-1)}\right)\leq N^{C}.\label{eqn:written3}
\end{align}
\end{lemma}

\begin{proof}
We start with the lower bounds, which are elementary: as $\E(e^X)>e^{\E(X)}$, we have
$$
\E_N\left(\prod_{i=1}^N |z_1-\la_i|^{2p}e^{-pN(|z_1|^2-1)}\right)\geq \exp\left(2pN\int \log(|z_1-\lambda|)(\rho_1^N(\lambda)-\frac{\mathds{1}_{|\lambda|<1}}{\pi})\rd m(z)\right)\geq \exp(\OO(1)),
$$
where for the last inequality we used that the density of states for the Ginibre ensemble is close to the uniform measure on the disk with high accuracy (se e.g. \cite[Lemma 4.5]{BouYauYin2014II}). This proves the lower bounds in (\ref{eqn:written}) and the lower bounds for (\ref{eqn:written2}), (\ref{eqn:written3}) hold  by the same argument.

For the upper bounds,  we only need to prove  (\ref{eqn:written2}), as (\ref{eqn:written}) will follow by monotonicity, and (\ref{eqn:written3}) by the Cauchy-Schwarz inequality from (\ref{eqn:written}).
Remember the notation (\ref{eqn:chii}) and abbreviate $\log_N(x)=\log (|x|^2+1/N)$. We can bound
\begin{multline}
 \E_N\left(\prod_{i=1}^N \left(|z_1-\la_i|^{2}+\frac{1}{N}\right)^pe^{-pN(|z_1|^2-1)}\right)
 \leq
 \E_N\left(
 e^{2p\sum \log_N(z_1-\lambda_i)\chi_{z_1,\delta}(\lambda_i)-2pN\int \log|z_1-\lambda|\chi_{z_1,\delta}(\lambda)}
 \right)^{\frac{1}{2}}\\
 \times \E_N\left(
 e^{2p\sum \log_N(z_1-\lambda_i)(1-\chi_{z_1,\delta}(\lambda_i))-2pN\int \log|z_1-\lambda|(1-\chi_{z_1,\delta}(\lambda))}.
 \right)^{\frac{1}{2}}\label{eqn:expeverywhere}
\end{multline}

For the first expectation corresponding to the short range, we apply Lemma \ref{lem:invariance1}, observing that $N\|\nu\|_1=\OO(N^{2\delta})$ is negligible for $\delta$ small enough. We obtain that this first expectation is equivalent to
$$
\E_N\left(
 e^{2p\sum \log_N(\lambda_i)\chi_{0,\delta}(\lambda_i)-2pN\int \log|\lambda|\chi_{0,\delta}(\lambda)}
 \right)\leq N^C,
$$
where the above inequality follows from Corollary \ref{Kost1}.

The second expectation in (\ref{eqn:expeverywhere})  is the Laplace transform of smooth linear statistics, so that the loop equations techniques apply to prove it is of polynomial order, see \cite[Theorem 1.3]{Lambert}.
More precisely, \cite{Lambert} applies to the smooth function $(1-\chi_{z_1,\delta})\log$ instead of $(1-\chi_{z_1,\delta})\log_N$, but we can decompose $\log_N(\lambda)=2\log|\lambda|+\log (1+\frac{1}{N|\lambda^2|})$. With the Cauchy-Schwarz inequality we separate contribution from these two functions, then 
 the analogue (for the unconditioned measure) of (\ref{eqn:est1}) shows  the Laplace transform of linear statistics of $(1-\chi_{z_1,\delta})\log (1+\frac{1}{N|\lambda^2|})$ is $\OO(1)$, and finally \cite[Theorem 1.3]{Lambert}  bounds the contribution of $(1-\chi_{z_1,\delta})\log|\lambda|$ by $\OO(N^C)$.
\end{proof}

\section{Off-diagonal overlaps}\label{sec:offdiag}

In this section we consider the distribution of $N$ Ginibre points conditioned to $\{ \lambda_1 = z_1, \lambda_2 = z_2\}$. We will successively prove identities for
the quenched off-diagonal overlaps,  for all $z_1,z_2$, and then get explicit relations for $z_1=0$ in the annealed setting. Finally, these new correlation identities
are extended to any $z_1, z_2$ in the bulk of the spectrum by a decomposition of short and long range contributions.
  
%
%This distribution has joint density
%$$ \rho_{N}^{(0,z)} (\lambda_3, \dots, \lambda_{N}) m^{\otimes N-2}(\rd\bla)= \frac{1}{ Z^{(0,z)}_{N}} \prod_{3 \leq i <j \leq N} |\lambda_i - \lambda_j|^2 \prod_{i=3}^{N} |\lambda_i|^2 |z - \lambda_i|^2 \rd\mu(\lambda_i)$$
%where the normalization constant is given by Theorem \ref{cond2}, remembering the notation (\ref{eqn:expo}): 
%\begin{equation}\label{constantZ2}
%Z^{(0,z)}_{N}
%=
%\left(\prod_{k=1}^{N-1}k!\right)N^{-\frac{(N-2)(N+1)}{2}}\frac{e_1^{(N-1)}(N|z|^2)}{N|z|^2}.
%\end{equation}

\subsection{The quenched off-diagonal overlap.}\ 
Contrary to the diagonal overlap, the factorization here doesn't involve independent variables.

\begin{proposition} The following equality in law holds, conditionally on $\{ \la_1, \dots, \la_N\}$:\label{prop:offdiag}
\begin{align*}
 \mathscr{O}_{12}& \overset{({\rm d})}{=}  - \Big| {T_{12} \over \lambda_1 - \lambda_2} \Big|^2\prod_{n=3}^N \Big(1 + {Z_{n} \over N (\lambda_1-\lambda_{n}) (\overline{\lambda_2 - \lambda_{n}})} \Big),
\end{align*}
where, conditionally on $ \mathcal{F}_{n-2}$, $Z_n$ is a product of two (correlated) complex Gaussian random variables, and $\E(Z_{n}\mid \mathcal{F}_{n-2})=1$.
\end{proposition}

\begin{proof}
As for the diagonal overlap, we simply compute, with the notation (\ref{eqn:Ogen}),
$$
\mathscr{O}_{12}^{(2)} = - |b_2|^2 = - \Big| {T_{12} \over \lambda_1 - \lambda_2} \Big|^2
$$
and 
\begin{multline*}
\mathscr{O}_{12}^{(n+1)}=  - \overline{b_2}\, B_{n+1}^{\rm t} \bar D_{n+1} 
=  - \overline{b_2} ( B_{n}^{\rm t} \bar D_{n} + b_{n+1} \overline{d_{n+1}}) 
= - \overline{b_2} (B_{n}^{\rm t} \bar D_{n} +  \alpha_{1,n+1} B_n^{\rm t}T_n  \overline{\alpha_{2, n+1} D_n^{\rm t} T_n}) \\
=  - \overline{b_2} B_{n}^{\rm t} \bar D_{n}  \Big( 1+ \alpha_{1,n+1} \overline{\alpha_{2, n+1}} { B_n^{\rm t} T_n \overline{D_n^{\rm t} T_n} \over B_{n}^{\rm t} \bar D_{n}}  \Big) 
= \mathscr{O}_{12}^{(n)} \Big( 1+ {Z_{n+1} \over N (\lambda_1 - \lambda_{n+1}) \overline{( \lambda_2 - \lambda_{n+1})}} \Big),
\end{multline*}
where 
\begin{equation}\label{eqn:Zn}
Z_{n+1}=N {B_n^{\rm t} T_n\,  \overline{D_n^{\rm t} T_n} \over B_{n}^{\rm t} \bar D_{n}}.
\end{equation}
Clearly, conditionally on $\mathcal{F}_{n-1}$, $Z_{n+1}$ is a product of two complex Gaussian random variables, a distribution which depends on
$\mathscr{O}_{12}^{(n)}$. Moreover, $B_n,D_n\in\mathcal{F}_{n-1}$ and $T_n$ is independent of $\mathcal{F}_{n-1}$, so that $\E(Z_{n+1}\mid \mathcal{F}_{n-1})=1$.  
\end{proof}

\begin{remark}
By successive conditional expectations with respect to $\mathcal{F}_{n-2},\dots,\mathcal{F}_1$,  Proposition \ref{prop:offdiag} implies
\begin{equation}
\E_T\left(\mathscr{O}_{12}\right) = - \frac{1}{ N |\lambda_1-\lambda_2|^2} \prod_{k=3}^{N} \Big(1 + \frac{1}{ N (\lambda_1-\lambda_k) (\overline{\lambda_2 - \lambda_k})} \Big)\label{eqn:ExpO12},
\end{equation}
an important fact already proved in \cite{ChaMeh1998,ChaMeh2000,MehCha1998}. 
\end{remark}

\subsection{The annealed off-diagonal overlap: expectation.}\ \label{subsec:2ndpatch}
Remarkably, the works \cite{ChaMeh1998,ChaMeh2000,MehCha1998} also explicitly integrated the random variable (\ref{eqn:ExpO12}) over $\lambda_3,\dots,\lambda_N$, in the specific case $\lambda_1=0$. 
We state the resulting asymptotics and add the proof from Chalker and Mehlig, for completeness.

\begin{corollary}[Chalker, Mehlig \cite{ChaMeh1998,ChaMeh2000,MehCha1998}]\label{meandiag2} For any $\e>0$, there exists $c>0$ such that uniformly in  $|z|<N^{-\e}$,
$$\E \left(\mathscr{O}_{12}\mid \lambda_1=0, \lambda_2=z \right)=-\frac{1}{ N |z|^4} \frac{1 - (1 + N |z|^2) e^{-N |z|^2}}{1-e^{-N|z|^2}} + \OO(e^{-c N}).$$
\end{corollary}

\begin{proof}
From (\ref{eqn:ExpO12}), we want to evaluate
$$\E \left(\mathscr{O}_{12}\mid \lambda_1=0, \lambda_2=z \right)=  - \frac{1}{ N|z|^2} \E \left(
\prod_{k=3}^{N} \Big(1 - \frac{1}{ N \lambda_k (\overline{z - \lambda_k})} \Big)
\mid\lambda_1=0, \lambda_2= z \right)$$
From Theorem \ref{cond2} with $g (\lambda) =1 - \frac{1}{ N\lambda ( \overline{z - \lambda})}$, we find that
\begin{align*}
\E\left(\prod_{k=2}^{N} g (\lambda_k) \mid  \lambda_1=0, \lambda_2 = z\right) & = 
\frac{1}{Z_N^{(0,z)}}\det(f_{i,j})_{i,j = 1}^{N-2}
\end{align*}
where 
$$f_{i,j} = \frac{1}{ (i+1)!} \int \lambda^{i-1} \bar{\lambda}^{j-1} |\lambda|^2 |z - \lambda |^2 g (\lambda)  \mu(\dd \lambda)
= \frac{1}{ (i+1)!} \int \lambda^{i-1} \bar{\lambda}^{j-1} \Big( |\lambda|^2 |z-\lambda|^2 -\frac{1}{N} \overline{ \lambda} (z - \lambda) \Big)  \mu(\dd \lambda).
$$
This matrix is tridiagonal with entries
$$
f_{ii}=\frac{1}{N^{i+1}}+\frac{|z|^2}{(i+1) N^i}+\frac{1}{(i+1) N^{i+1}},\ 
f_{i,i+1}=-\frac{\overline{z}}{N^{i+1}},\ 
f_{i,i-1}=- \frac{z}{iN^{i}}.
$$
Let $d_k=\det(f_{i,j})_{i,j = 1}^{k}$ and $x= N|z|^2$.
With the convention $d_0=1$ we have
\begin{align*}
d_1&=\frac{1}{N^2}\left(\frac{3}{2}+\frac{x}{2}\right),\\
d_k&=\left(1+\frac{x+1}{k+1}\right)\frac{1}{N^{k+1}}d_{k-1}-\frac{x}{k}\frac{1}{N^{2k+1}}d_{k-2}.
\end{align*}
so that $a_k=d_k N^{\frac{k(k+3)}{2}}$ satisfies $a_0=1$, $a_1=3/2+x/2$,
$$
a_k=\left(1+\frac{x+1}{k+1}\right)a_{k-1}-\frac{x}{k}a_{k-2}.
$$
An immediate induction gives
$a_k=(k+2)x^{-2}e_2^{(k+2)}(x)$.
Thus, we conclude
$$\E \left( \mathscr{O}_{12} \mid \ \lambda_1=0, \lambda_2=z \right)=  - \frac{N}{x^2}\frac{e_2^{(N)}(x)}{e_1^{(N)}(x)}.$$
The proof is then concluded by standard asymptotics.
\end{proof}

With Corollary \ref{meandiag2}, the expectation of $\mathscr{O}_{12}$ is known for $\la_1=0$. To extend the result to anywhere in the bulk of the spectrum, we mimic the alternative proof of Theorem \ref{meandiag},
from Subsection \ref{subsec:1stpatch}.

\begin{proof}[Proof of Theorem \ref{thm:expoff}]
 We denote $\nu_{z_1,z_2}$ the measure (\ref{eqn:partitionfunction}) conditioned to $\lambda_1=z_1,\la_2=z_2$. Note that $\nu_{z_1,z_2}$ is a determinantal measure.
With a slight abuse of language, we will abbreviate $\E_{\nu_{z_1,z_2}}(X)=\E(X\mid \la_1=z_1,\la_2=z_2)$
even for $X$ a function of the overlaps. 

 We follow the same three steps as in the alternative proof of Theorem \ref{meandiag}. Strictly speaking, if we were to impose $|z_1-z_2|>N^{-C}$ for some fixed $C$, we would not need the first step below,
 as the singularity $1/|z|$ is integrable, contrary to our previous singularity $1/|z|^2$. However, in Theorem \ref{thm:expoff} we allow $z_1$ and $z_2$ to be arbitrarily close, so we first perform an initial small cutoff.\\

\noindent{\it First step: small cutoff.} Let $g_{z_1,z_2}(\lambda)= 1+\frac{1}{N(z_1-\lambda)(\overline{z_2-\lambda})}$. Remember that, from (\ref{eqn:ExpO12})
$$
\E_{\nu_{z_1,z_2}}(\mathscr{O}_{12}) =-\frac{1}{N|z_1-z_2|^2} \E_{\nu_{z_1,z_2}}\left(\prod_{n=3}^{N} g_{z_1,z_2}(\lambda_i)\right).
$$
We denote $h_{z_1,z_2}(\lambda)=g_{z_1,z_2}(\lambda) \mathds{1}_{\la\not\in\mathscr{B}}$ where $\mathscr{B}=\{|\la-z_1|<N^{-A}\}\cup\{|\la-z_2|<N^{-A}\}$. For $A$ a large enough constant, the analogue of 
(\ref{eqn:elementary}) holds:
\begin{equation}\label{eqn:elementaryII}
\E_{\nu_{z_1,z_2}}\left(\prod_{i=3}^{N} g_{z_1,z_2}(\lambda_i)\right)= \E_{\nu_{z_1,z_2}}\left(\prod_{i=3}^{N} h_{z_1,z_2}(\lambda_i)\right)+\OO(N^{-3}).
\end{equation}
Indeed, by making explicit the above conditional measures, (\ref{eqn:elementaryII}) follows from
\begin{equation}\label{eqn:elem1II}
\E_{N-2}\left(\prod_{i=3}^N(|z_1-\lambda_i|^2|z_2-\lambda_i|^2+N^{-1}\overline{(z_1-\la_i)}(z_2-\la_i))e^{-N(|z_1|^2-1)-N(|z_2|^2-1)}\right)
\qquad \qquad \qquad
\end{equation}
$$
- \E_{N-2}\left(\prod_{i=3}^N(|z_1-\lambda_i|^2|z_2-\lambda_i|^2+N^{-1}\overline{(z_1-\la_i)}(z_2-\la_i))\mathds{1}_{\la_i\not\in\mathscr{B}}e^{-N(|z_1|^2-1)-N(|z_2|^2-1)}\right)=\OO(N^{-B}),
$$
and
\begin{equation}\label{eqn:elem2II}
\E_{N-2}\left(\prod_{i=3}^N(|z_1-\lambda_i|^2|z_2-\lambda_i|^2)e^{-N(|z_1|^2-1)-N(|z_2|^2-1)}\right)\geq N^{-C_1}.
\end{equation}
with $B$ much larger than $C_1$.
Lemma \ref{lem:generalmoments} gives (\ref{eqn:elem2II}).
The left hand side of (\ref{eqn:elem1II}) has size order
$$
\E_{N-2}\left(\prod_{i=3}^N(|z_1-\lambda_i|^2+N^{-1})(|z_2-\lambda_i|^2+N^{-1})e^{-N(|z_1|^2-1)-N(|z_2|^2-1)}\mathds{1}_{\exists i: \la_i\in\mathscr{B}}\right)
=\OO(N^{-A+C_2})
$$
by the Cauchy-Schwarz inequality and Lemma \ref{lem:generalmoments}, for some $C_2$ which does not depend on $A$.
This concludes the proof of
(\ref{eqn:elementaryII}) by choosing $A$ large enough.\\

\noindent{\it Second step:  the long-range contribution concentrates.} 
We smoothly separate the short-range from a long-range contributions on the right hand side of (\ref{eqn:elementaryII}):\
\begin{align*}
h_{z_1,z_2}(\lambda)&=e^{h_{z_1,z_2}^{s}(\lambda)+h_{z_1,z_2}^{\ell}(\lambda)},\\
h_{z_1,z_2}^{s}(\lambda)&=  \log \Big( 1 + \frac{1}{N (z_1 - \la)\overline{(z_2-\la)}}\mathds{1}_{ \la \not\in\mathscr{B}} \Big) \chi_{z,\delta} (\la),\\
h_{z_1,z_2}^{\ell}(\la)&=  \log \Big( 1 + \frac{1}{N (z_1 - \la)\overline{(z_2-\la)}} \Big) (1 - \chi_{z,\delta} (\la)),
\end{align*}
and we denote $z=(z_1+z_2)/2$, recall $|z_1-z_2|<N^{-\frac{1}{2}+\kappa-\e}$, $\chi_{z,\delta} (\la)= \chi \Big(N^{\frac{1}{2}-\delta}|z-\la| \Big)$, and choose  $\delta \in (\kappa-\e, \kappa)$.
In the definition of $h_{z_1,z_2}^{s}$, we can choose any branch for the logarithm, this won't have any impact on the rest of the proof.\
In the long-range contribution $h^{\ell}_{z_1,z_2}$, the logarithm is defined by continuity from $\log(1)=0$.
Let $f^{\ell}_{z_1,z_2}(\la)= \frac{1}{N (z_1 - \la)\overline{(z_2-\la)}} (1 - \chi_{z,\delta} (\lambda))$ and $\bar f_{z_1,z_2}=\frac{N-2}{N}\frac{1}{ \pi} \int_{\mathbb{D}} {1 - \chi_{z,\delta} (\lambda)  \over (z_1-\la)\overline{(z_2-\la)}} \dd m(\lambda)$. Note that
\begin{equation}\label{eqn:estIII}
|\sum_{i=3}^Nh^\ell_{z_1,z_2}(\lambda_i)-\bar f_{z_1,z_2}|\leq |\sum_{i=3}^Nf^\ell_{z_1,z_2}(\lambda_i)-\bar f_{z_1,z_2}|+\frac{1}{2}\left(\sum_{i=3}^N\frac{1}{N^2|z_1-\lambda_i|^4}(1 - \chi_{z,\delta} (\lambda_i))+
\sum_{i=3}^N\frac{1}{N^2|z_2-\lambda_i|^4}(1 - \chi_{z,\delta} (\lambda_i))
\right).
\end{equation}
The last two sums  are bounded as in (\ref{eqn:est2}).
For the first term on the right hand side, we  bound the real and imaginary parts separately:
similarly to  (\ref{eqn:concentration}), we have
\begin{equation}\label{eqn:concentrationIII}
\E_{\nu_{z_1,z_2}}\left(e^{\alpha(\sum_{i=3}^N {\rm Re}f^\ell_{z_1,z_2}(\lambda_i)-\E_{\nu_{z_1,z_2}}\left(\sum_{i=3}^N{\rm Re} f^\ell_{z_1,z_2}(\lambda_i)\right))}\right)\leq e^{C\alpha^2 {\rm Var}_{\nu_{z_1,z_2}}(\sum_{i=3}^N {\rm Re}f^\ell_{z_1,z_2}(\lambda_i))}
\end{equation}
where $\alpha=c N^{2\delta}$ for some fixed $c$
and  $C$ which does not depend on $N$. We first bound the above variance. Remember we have a partition of type $1=\chi+\sum_{k\geq 1}\xi(2^{-k}x)$ for any $x>0$, with $\xi$ smooth, compactly supported.
Let $f^{\ell}_{z_1,z_2,k}(\lambda)=f^{\ell}_{z_1,z_2}(\lambda)\xi(2^{-k}N^{1/2-\delta}|z-\la|)$ and $K=\min\{k\geq 1:2^kN^{-1/2+\delta}>10\}$. Then $\sum_if^{\ell}_z(\lambda_i)=\sum_{k=1}^{K}\sum_if^{\ell}_{z,k}(\lambda_i)$ with probability $1-e^{-cN}$, and with  \cite[Theorem 1.2]{BouYauYin2014II}, for any  $\e>0$ and  $D>0$, there exists $N_0>0$ such that for any $N\geq N_0$ and   $1\leq k\leq K$ we have
(we now omit to write the real part, being understanding that $f$ is either ${\rm Re} f$ or ${\rm Im} f$)
$$
\mathbb{P}_{N-2}\left(
\left|\sum f^{\ell}_{z_1,z_2,k}(\lambda_i)-(N-2)\int_{|z|<1} f^{\ell}_{z_1,z_2,k}\right|>N^{-2\delta+\e}
\right)
\leq N^{-D}.
$$
The same estimate holds for the conditioned measure by  the Cauchy-Schwarz inequality:
\begin{multline*}
\mathbb{P}_{\nu_{z_1,z_2}}\left(
|\sum f^{\ell}_{z_1,z_2,k}(\lambda_i)-(N-2)\int_{|z|<1} f^{\ell}_{z_1,z_2,k}|>N^{-2\delta+\e}
\right)\\
\leq
\mathbb{P}_{N-2}\left(
|\sum f^{\ell}_{z_1,z_2,k}(\lambda_i)-(N-2)\int_{|z|<1} f^{\ell}_{z_1,z_2,k}|>N^{-2\delta+\e}
\right)^{1/2}N^C\\
\E_{N-2}\left(\prod_{i=2}^N|z_1-\lambda_i|^4e^{-2N(|z_1|^2-1)}\prod_{i=2}^N|z_2-\lambda_i|^4e^{-2N(|z_2|^2-1)}\right)^{1/2}
\leq N^{-\frac{D}{2}+2C}.
\end{multline*}
for some $C$ which only depends on $\kappa$,
where we used Lemma \ref{lem:generalmoments}. We conclude that for any small $\e>0$ and $D$ we have
$$
\mathbb{P}_{\nu_{z_1,z_2}}\left(
|\sum f^{\ell}_{z_1,z_2}(\lambda_i)-\bar f_{z_1,z_2}|>N^{-2\delta+\e}
\right)\leq N^{-D},
$$
so that ${\rm Var}_{\nu_{z_1,z_2}}(\sum_{i=3}^N f_{z_1,z_2}^\ell(\lambda_i))=\OO(N^{-4\delta+\e})$ and $\mathbb{E}_{\nu_{z_1,z_2}}(\sum_{i=3}^N f_{z_1,z_2}^\ell(\lambda_i))=\bar f_{z_1,z_2}+\OO(N^{-2\delta+\e})$.
As a consequence, (\ref{eqn:concentrationIII}) becomes 
\begin{equation}\label{eqn:est1II}
\E_{\nu_{z_1,z_2}}\left(e^{\alpha(\sum_{i=3}^Nf^\ell_{z_1,z_2}(\lambda_i)-\bar f_{z_1,z_2})}\right)\leq e^{C\alpha^2 N^{-4\delta+\e}+C\alpha N^{-2\delta+\e}}.
\end{equation}
With $\alpha=c N^{2\delta}$, we obtain
$$
\mathbb{P}_{\nu_{z_1,z_2}}(A)<e^{-c N^{\e}},\ \mbox{where}\  A=\left\{
\left|
\sum_{i=2}^Nh^\ell_{z_1,z_2}(\lambda_i)-\bar f_{z_1,z_2}
\right|>N^{-2\delta+\e}
\right\}.
$$
This yields, for some $p=1+c(\kappa)$, $c(\kappa)>0$ and some $q,r>1$,
\begin{equation}\label{eqn:concII}
\mathbb{E}_{\nu_{z_1,z_2}}\left(e^{\sum_i(h^s_{z_1,z_2}(\lambda_i)+h^\ell_{z_1,z_2}(\lambda_i))}\mathds{1}_{A}\right)
\leq 
\mathbb{E}_{\nu_{z_1,z_2}}\left(e^{p\sum_ih^s_{z_1,z_2}(\lambda_i)}\right)^{1/p}
\mathbb{E}_{\nu_{z_1,z_2}}\left(e^{q\sum_i h^\ell_{z_1,z_2}(\lambda_i)}\right)^{1/q}
\mathbb{P}(A)^{1/r}\leq e^{-c N^\e}.
\end{equation}
Here we used that the third term has size order $e^{-c N^{\e}}$, the second one is of order  $e^{q\bar f_{z_1,z_2}}=\OO(N^{C})$, and so is the first one from Lemma \ref{lem:momentII} and 
$|1+\frac{1}{(z_1-\la)\overline{(z_2-\la)}}|\leq (1+\frac{1}{N|z_1-\la|^2})(1+\frac{1}{N|z_2-\la|^2})$.
Moreover,
\begin{multline}
\mathbb{E}_{\nu_{z_1,z_2}}\left(e^{\sum_i(h^s_{z_1,z_2}(\lambda_i)+h^\ell_{z_1,z_2}(\lambda_i))}\mathds{1}_{A^c}\right)
=
(1+\OO(N^{-2\delta+\e}))\mathbb{E}_{\nu_{z_1,z_2}}\left(e^{\sum_ih^s_{z_1,z_2}(\lambda_i)+\bar f_{z_1,z_2}}\right)\\-
(1+\OO(N^{-2\delta+\e}))\mathbb{E}_{\nu_{z_1,z_2}}\left(e^{\sum_ih^s_{z_1,z_2}(\lambda_i)+\bar f_{z_1,z_2}}\mathds{1}_{A^c}\right),\label{eqn:useless}
\end{multline}
and this last expectation is of order $e^{-c N^{\e}}$ for the same reason as (\ref{eqn:concII}). To summarize, with the previous two equations we have proved (up to exponentially small additive error terms)
$$
\mathbb{E}_{\nu_{z_1,z_2}}\left(\prod_{i=3}^Nh_{z_1,z_2}(\lambda_i)\right)
=
(1+\OO(N^{-2\delta+\e}))e^{\bar f_{z_1,z_2}}\ \mathbb{E}_{\nu_{z_1,z_2}}\left(e^{\sum_ih^s_{z_1,z_2}(\lambda_i)}\right)
.
$$
%From (\ref{eqn:scales}) and Cauchy-Schwarz,
%we have
%$$
%\E(\mathscr{O}_{11}\mid\lambda_1=z) 
%=
%\E\left(e^{\sum_{n=2}^N h^{\text{loc}}(\lambda_k)+(N-1)\bar h}\right)
%+
%$$
%
%
%%Let us define $a_{kN} = {1 - \chi_{\epsilon} (\lambda_k) \over |z_1 - \lambda_k|^2} < N^{1-2 \epsilon}$. We were careful to take $\epsilon>{3 \over 8}$ so that $1-2\epsilon < \frac{1}{ 4}$, and $ |X_k|^2 $ are independent $\gamma_1$ variable with finite moments; we can therefore use lemma  \ref{BBC} and state the convergence 
%%$$ \frac{1}{ N} \sum_{k=1}^N  {1 - \chi_{\epsilon} (\lambda_k) \over |z_1 - \lambda_k|^2}  |X_k|^2 \xrightarrow[]{a.s.} \frac{1}{ N} \sum_{k=1}^N  {1 - \chi_{\epsilon} (\lambda_k) \over |z_1 - \lambda_k|^2}.$$
%%The error term converges as $1-2\epsilon < \frac{1}{ 4}$. 
%
%We then use rigidity (see appendix D) to justify
%$$ \sum_{k=2}^N    {1 - \chi_{z, \theta}(\lambda_k) \over N |z - \lambda_k|^2}  = \frac{1}{ \pi} \int_{\mathbb{D}} {1 - \chi \big({ { |\lambda - z|^2 \over \theta^2} \big)} \over |z - \lambda|^2} \dd m(\lambda) + \OO_{\prec} (N^{-2 \epsilon})$$
%and lemma \ref{integral1} gives us the asymptotics of this integral. We conclude
%$$\sum_{k=1}^N h_k^{\text{glob}} =  (1- 2 \epsilon) \log (N) + \log (1 - |z|^2 )  + C +O(?). $$
%$$e^{\sum_{k=1}^N h_k^{\text{glob}}} = e^C (1 - |z|^2 ) N^{1-2\epsilon} $$

\noindent{\it Third step: the local part is invariant.} For our test function $h^s_{z_1,z_2}$, the reader can easily check  the conditions of Lemma \ref{lem:invariance2II}, so that
$$
\mathbb{E}_{\nu_{z_1,z_2}}\left(e^{\sum_ih_{z_1,z_2}^s(\lambda_i)}\right)=\mathbb{E}_{\nu_{0,z_2-z_1}}\left(e^{\sum_ih_{0,z_2-z_1}^s(\lambda_i)}\right)+\OO\left(e^{-c N^{2\kappa}}\right).
$$
This yields
$$
\mathbb{E}_{\nu_{z_1,z_2}}\left(\prod_{i=3}^Nh_{z_1,z_2}(\lambda_i)\right)
=
(1+\OO(N^{-2\delta+\e}))e^{\bar f_{z_1,z_2}-\bar f_{0,z_2-z_1}}\ \mathbb{E}_{\nu_{0,z_2-z_1}}\left(\prod_{i=3}^Nh_{0,z_2-z_1}(\lambda_i)\right).$$
From Lemma \ref{lem:integral2}, $e^{\bar f_{z_1,z_2}-\bar f_{0,z_2-z_1}}=(1-z_1\overline{z_2})^{\frac{N-2}{N}}$. Together with Corollary \ref{meandiag2}, this concludes the proof.
\end{proof}

\begin{lemma}\label{lem:integral2} For any $\lambda_1, \lambda_2 \in \mathbb{D}$,   
$$ \frac{1}{ \pi} \int_{\mathbb{D}}  \frac{1}{ ( \lambda_1 - z) (\overline{\lambda_2 - z}) } \dd m(z)  = \log \Big({ 1 - \lambda_1 \overline{\lambda_2} \over |\lambda_1-\lambda_2|^2} \Big).$$
%In particular, denoting $z=(z_1+z_2)/2$, for any $\delta<\kappa$ there exists $A=A(N,\delta,|z_2-z_1|)$ such that
%$$ \frac{1}{ \pi} \int_{\mathbb{D}}  {1-\chi_{z,\delta}(w) \over ( \lambda_1 - w) (\overline{\lambda_2 - w}) } \dd m(w)  = \log \Big({ 1 - \lambda_1 \overline{\lambda_2} \over |\lambda_1-\lambda_2|^2} \Big)+A.$$
\end{lemma}

\begin{proof} We consider the following domains, assuming $0< |\lambda_1| < |\lambda_2| <1$ and $\epsilon >0$ is small enough. The following computation still holds if $|\lambda_1| = |\lambda_2|$, as long as  $\lambda_1 \neq \lambda_2$. Integrability is clear, as the poles are simple and isolated. Moreover, under these conditions, the integral cancels on the disks $D_1$ and $D_2$.
\begin{minipage}{0.55\linewidth}
   \includegraphics[width=8cm]{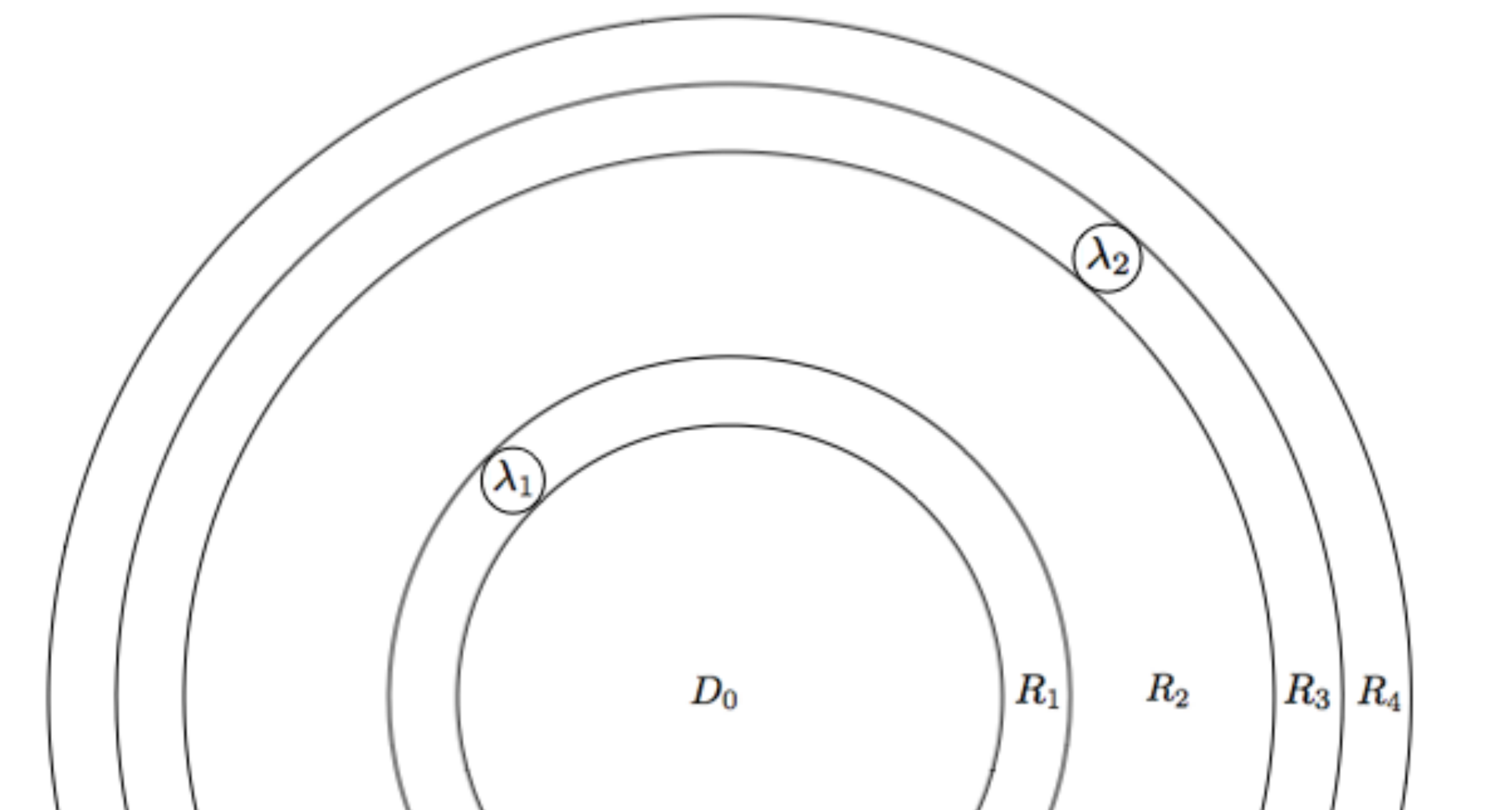}
\end{minipage}\hfill
\begin{minipage}{0.4\linewidth}
\begin{align*} D_0 &= D(0, |\lambda_1|-\epsilon), \\
D_1 &= D(\lambda_1, \epsilon), \\
D_2 &= D(\lambda_2,\epsilon), \\
R_1 & =D(0,|\lambda_1|+\epsilon)-D(0,|\lambda_1|-\epsilon), \\
R_2 & =D(0,|\lambda_2|-\epsilon)-D(0,|\lambda_1|+\epsilon), \\
R_3 & =D(0,|\lambda_2| + \epsilon)-D(0,|\lambda_2|-\epsilon), \\
R_4 & =D(0,1)-D(0,|\lambda_2|+\epsilon) .
\end{align*}
\end{minipage} \\ \\
\noindent Integration over the domain $D_0$ yields
\begin{multline}\label{integr1} \frac{1}{ \pi} \int_{D_0} \frac{1}{  ( \lambda_1 - z) (\overline{\lambda_2 - z})}  \dd m(z)  
 =  \frac{1}{ \pi} \iint_{r=0}^{|\lambda_1|- \epsilon} \frac{1}{  \lambda_1 \overline{\lambda_2}} {r \dd r \dd \theta \over (1 - {r e^{i \theta}  \over \lambda_1}) (1- \overline{r e^{i \theta} \over \lambda_2 }) }
 = 2  \int_{r=0}^{|\lambda_1|- \epsilon} \frac{1}{  \lambda_1 \overline{\lambda_2}} \sum_k  \Big(
{r^2  \over \lambda_1 \overline{\lambda_2}} \Big)^k
r \dd r \\
 =  \int_{r=0}^{|\lambda_1|- \epsilon} {2 r \over  \lambda_1 \overline{\lambda_2}-r^2} \dd r 
 =  \log(\lambda_1 \overline{\lambda_2}) - \log (\lambda_1 \overline{\lambda_2} - (|\lambda_1|- \epsilon)^2).
\end{multline}
The same type of expansion shows that the integral over $R_2$ vanishes and the contribution from
$R_4$ is
\begin{equation}\label{integr2}
\frac{1}{ \pi} \int_{R_4} \frac{1}{ ( \lambda_1 - z) (\overline{\lambda_2 - z})}  \dd m(z)
=  \log (1 - \lambda_1 \overline{\lambda_2}) - \log((|\lambda_2|+\epsilon)^2 - \lambda_1 \overline{\lambda_2}).
\end{equation}
As the expression is integrable, on the domains $R_1, R_3$  there is no contribution as $\epsilon\to0$. Summing (\ref{integr1}) and (\ref{integr2}) gives the result.
\end{proof}

\subsection{The quenched off-diagonal overlap: second moments.}\ 
The main result of this subsection is the following lemma, which gives the expectation of second moments of overlaps conditionally on the eigenvalues positions. 
For this, we define 
\begin{equation}\label{eqn:An}
X_n= \left(  \begin{array}{c}
 |B_{n}^{\rm t} \bar D_{n}|^2 \\
 \|B_n\|^2 \|D_n\|^2 \end{array} \right),\
 \gamma_{ij}=\frac{\alpha_{i,j}}{\sqrt{N}},\ 
A_n = \left( \begin{array}{cc}
 |1 + \gamma_{1,n} \overline{\gamma_{2,n}} |^2&|\gamma_{1,n} \gamma_{2,n}|^2 \\
|\gamma_{1,n} \gamma_{2,n}|^2&(1+|\gamma_{1,n}|^2) (1+|\gamma_{2,n}|^2 ) 
 \end{array} \right).
\end{equation}

\begin{lemma} For any $2\leq n\leq N-1$ we have
\begin{equation}\label{eqn:rec}
\E \left(X_{n+1} \mid \mathcal{F}_{n-1}\right) = A_{n+1} X_{n}.
\end{equation}
In particular,
\begin{equation}\label{eqn:rec2}
\left(
\begin{array}{c}
\E_T(|\mathscr{O}_{12}|^2\mid \mathcal{F}_1)\\
\E_T( \mathscr{O}_{11}\mathscr{O}_{22} \mid \mathcal{F}_1)
\end{array}
\right)
=
\left(
\begin{array}{cc}
|b_2|^2&0\\
0&1+|b_2|^2\\
\end{array}
\right)
\left(\prod_{i=3}^n A_i\right)
\left(
\begin{array}{cc}
|b_2|^2&0\\
0&1+|b_2|^2\\
\end{array}
\right)
\left(
\begin{array}{c}
1\\
1
\end{array}
\right).
\end{equation}
\end{lemma}

\begin{proof}
We recall the notation (\ref{eqn:Zn}) and the property $\E(Z_{n+1}\mid \mathcal{F}_{n-1})=1$. A short calculation also gives
$
\E(|Z_{n+1}|^2\mid \mathcal{F}_{n-1})=1+\frac{\|B_n\|^2\|D_n\|^2}{|B_n^{\rm t} \bar D_n|^2}$. Abbreviating $\gamma_k=\gamma_{k,n+1}$, this gives
\begin{align}
\E \left(|B_{n+1}^{\rm t}\bar D_{n+1}|^2\mid  \mathcal{F}_{n-1}\right) & = \E\left(| B_n^{\rm t} \bar D_n|^2 | 1+ \gamma_1 \overline{\gamma_2} Z_{n+1}|^2 \mid \mathcal{F}_{n-1} \right) \notag\\
&=  |B_n^{\rm t} \bar D_n|^2  \E\left(1+  \gamma_1 \overline{\gamma_2} Z_{n+1} +  \overline{\gamma_1} {\gamma_2} \overline{Z_{n+1}} +  |\gamma_1 {\gamma_2} Z_{n+1}|^2\mid \mathcal{F}_{n-1}\right) \notag\\
&= | 1+ \gamma_1 \overline{\gamma_2}|^2 |B_n^{\rm t} \bar D_n|^2 + |\gamma_1 {\gamma_2} |^2 \|B_n\|^2 \|D_n\|^2\label{eqn:calc1}.
\end{align}
We now denote $X=X_{n+1}= \sqrt{N} {B_n^{\rm t}T_n  \over  \| B_n \|}$, $Y=Y_{n+1}= \sqrt{N}{D_n^{\rm t}T_n  \over  \| D_n \|}$, so that
$\E(|X_{n+1}|^2 \mid \mathcal{F}_{n-1})=\E(|Y_{n+1}|^2 \mid \mathcal{F}_{n-1})=1$ and 
$
\E(|X_{n+1}Y_{n+1}|^2 \mid \mathcal{F}_{n-1})=1+\frac{|B_n^{\rm t} \bar D_n|^2}{\|B_n\|^2\|D_n\|^2}
$.
This yields
\begin{align}
\E \left(\|B_{n+1} \|^2 \|D_{n+1} \|^2 \mid  \mathcal{F}_{n-1}\right) & = \E\left(\|B_n\|^2 \|D_n\|^2 \big( 1+ |\gamma_1 X|^2\big)  \big( 1+ |\gamma_2 Y|^2\big)\mid\mathcal{F}_{n-1}\right)\notag\\
&=  \|B_n\|^2 \|D_n\|^2  \E\left(1+  |\gamma_1 X|^2 +  |{\gamma_2}Y|^2 +  |\gamma_1 {\gamma_2} XY |^2\mid  \mathcal{F}_{n-1}\right) \notag\\
&= \big( 1+ |\gamma_1|^2 \big) \big(1 + |\gamma_2|^2 \big) \|B_n\|^2 \|D_n\|^2+ |\gamma_1 {\gamma_2} |^2|B_n^{\rm t} \bar D_n|^2.\label{eqn:calc2}
\end{align}
Equations (\ref{eqn:calc1}) and (\ref{eqn:calc2}) together conclude the proof of (\ref{eqn:rec}).
Denoting $Y_n=\E_{T}(X_n\mid \mathcal{F}_1)$, we obtain
$$
Y_n=\left(\prod_{i=3}^n A_i\right)Y_2=\left(\prod_{i=3}^n A_i\right)\E\left(
\left(\begin{array}{c}
|B_2^{\rm t} \bar D_2|^2\\
\|B_2\|^2\|D_2\|^2
\end{array}
\right)\mid \mathcal{F}_1\right)
=
\left(\prod_{i=3}^n A_i\right)\left(
\begin{array}{c}
|b_2|^2\\
1+|b_2|^2
\end{array}
\right)
$$
and (\ref{eqn:rec2}) immediately follows.
\end{proof}

\subsection{The annealed off-diagonal overlap: second moments for $\la_1=0$.}\ 
We now want to integrate (\ref{eqn:rec2}) over the eigenvalues $\la_3,\dots,\la_N$, first in the special case $\lambda_1=0$.
This requires some new notations.
We abbreviate
$$\delta= N | \lambda_{1} - \lambda_{2} |^2, \ \
a= \frac{\delta}{2} + \sqrt{1 + \frac{\delta^2}{4}},\ \
b=-a^{-1}=\frac{\delta}{2} - \sqrt{1 + \frac{\delta^2}{4}}
$$
and will often use the property
\begin{equation}\label{eqn:abd}
\delta =  a-\frac{1}{a}=b-\frac{1}{b}.
\end{equation}
We also define the following rational fractions of $x$
\begin{align}\label{eqn:u}
u_k(x)&=1- {1-x^{-1} \over  k+3}, \\
 d_k(x,\delta)&=
\frac{ (k+2) (k+3)}{e_1^{(k+1)}(\delta)}
\left(
u_k(x) \frac{1}{(x-1)^2}e_3^{(k)}(\delta)-\frac{\delta u_k(x)}{2 x}
+ \frac{1}{(x-1)^2}\frac{x^2+(k+2)x+ (k+1)(k+3)}{(k+3)!}\delta^{k+1}
\right).
\notag
\end{align}
We can now state the following main proposition on which Theorem  \ref{thm:meso2} depends. The reason why such formulas exist relies on two algebraic facts.
\begin{enumerate}
\item The $A_n$'s from (\ref{eqn:An}) are diagonalizable in the same basis, see (\ref{eqn:diago}). Their commutation was expected, our choice of eigenvalues ordering being arbitrary, while
the left hand side of (\ref{eqn:rec2}) is intrinsic.
\item More surprisingly, the obtained holonomic sequence (\ref{eqn:recf}) is exactly solvable.
\end{enumerate}

\begin{proposition}\label{prop:theformula}
Conditionally on $\{\lambda_1=0,\lambda_2=z\}$, we have
\begin{align*}
\E(|\mathscr{O}_{12}|^2)&=\frac{1}{1+a^2}\left(\frac{a^2}{(a+1)^2}d_{N-2}(a,\delta)+\frac{a^2}{(a-1)^2}d_{N-2}(b,\delta)\right),\\
\E(\mathscr{O}_{11}\mathscr{O}_{22})&=\frac{1}{1+a^2}\left(\frac{1}{(a+1)^2}d_{N-2}(a,\delta)+\frac{a^4}{(a-1)^2}d_{N-2}(b,\delta)\right).
\end{align*}
\end{proposition}
\begin{proof}
Importantly, the matrices $A_n, 3\leq n\leq N$, can be diagonalized in the same basis: from (\ref{eqn:An})
an elementary calculation based on (\ref{eqn:ward}) gives
\begin{equation}\label{eqn:diago}
A_n = (1+|\gamma_{1,n}|^2) (1+|\gamma_{2,n}|^2 ) {\rm I}_2 + |\gamma_{1,n} \gamma_{2,n}|^2 \left( \begin{array}{cc}
 - N | \lambda_{1} - \lambda_{2} |^2\nc &  1 \\
1  &  0\nc
 \end{array} \right),
\end{equation}
so that its eigenvectors clearly do not depend on $n$. 
With these notations, the eigenvalues of $A_n$ are 
\begin{align*}
\lambda_+(n) & =(1+|\gamma_{1,n}|^2) (1+|\gamma_{2,n}|^2 ) - |\gamma_{1,n} \gamma_{2,n}|^2  a \\
\lambda_-(n) & = (1+|\gamma_{1,n}|^2) (1+|\gamma_{2,n}|^2 ) - |\gamma_{1,n} \gamma_{2,n}|^2 b,
\end{align*}
and the orthogonal basis
$$
U=\frac{1}{\sqrt{1+a^2}}
 \left( \begin{array}{cc}
a &  -1 \\
1  & a
 \end{array} \right)
$$
diagonalizes all $A_n$'s simultaneously: 
$UA_nU^{t}=\diag(\lambda_+(n),\lambda_-(n))$.
Then (\ref{eqn:rec2}), together with the codiagonalization of the $A_n$'s, $\E |T|^2=1$ and $\E(|T|^4)=2$, yields the following simple expression:
\begin{align}\label{expT}
&\left(
\begin{array}{c}
\E_T(|\mathscr{O}_{12}|^2)\\
\E_T(\mathscr{O}_{11}\mathscr{O}_{22})
\end{array}
\right)
=
\frac{1}{1+a^2}
\left(
\begin{array}{c}
\frac{a^2}{(a+1)^2}d_++\frac{a^2}{(a-1)^2}d_-\\
\frac{1}{(a+1)^2}d_++\frac{a^4}{(a-1)^2}d_-
\end{array}
\right),\\
&d_+=\prod_{n=3}^N\lambda_+(n),\ d_-=\prod_{i=3}^N\lambda_-(n)\label{notationd}.
\end{align}
Note that we have not yet used $z_1=0$: the above formula holds for any given $z_1,z_2$.

The remarkable fact is that in the specific case $z_1=0$, $\E(d_+)$ and $\E(d_-)$ can be calculated, as shown below (here $\E$ denotes integration over all variables except $\la_1,\la_2$).
We start with $\E(d_+)$.
From Theorem \ref{cond2}, the following representation holds:
$$ \E\left(\prod_{n=3}^{N} \lambda_{+}(n) \mid \lambda_1=0, \lambda_2=z\right) = \frac{ \det(f_{i,j})_{1\leq i,j\leq N-2} }{Z_N^{(0,z)}}$$
where $Z_N^{(0,z)}$ is given by (\ref{eqn:normalization}) and
\begin{align*}
f_{i,j} &= \frac{1}{ {( i+\nc1)}! } \int \lambda^{i-1} \bar{\lambda}^{j-1} |\lambda|^2 |z-\lambda|^2 \lambda_+(\lambda) \mu( \dd \lambda)\\
&= \frac{1}{ {( i+\nc1)}! } \int \lambda^{i-1} \bar{\lambda}^{j-1}( |\lambda|^2 |z-\lambda|^2 +N^{-1}  |\lambda|^2 + N^{-1}|z-\lambda|^2 -N^{-2}(a-1)) \mu( \dd \lambda).
\end{align*}
We expand $|\lambda-z|^2 = |\lambda|^2 + |z|^2 - \bar z \lambda - z \bar{\lambda}$. The resulting matrix is thus tridiagonal with
\begin{align*} 
f_{i,i-1}& = {-z }\frac{1}{iN^i},\\
f_{i,i} & = \frac{1}{N^{i+1}}+\frac{2N^{-1}+|z|^2}{(i+1)N^i}+\frac{N^{-1}|z|^2-N^{-2}(a-1)}{i(i+1)N^{i-1}},\\
f_{i,i+1} & = -\bar z\frac{i+2}{(i+1)N^{i+1}}.
\end{align*}
With the notation  $d_k=\det(f_{ij})_{1\leq i,j\leq k}$, the recurrence 
$d_{k} = f_{k,k} d_{k-1} - f_{k,k-1} f_{k-1,k} d_{k-2}$ holds,
so that defining $a_k=d_k N^{\frac{k(k+3)}{2}}$  we have
$$
a_k=\left(1+\frac{2+\delta}{k+1}+\frac{1-a^{-1}}{k(k+1)}\right)a_{k-1}
- \delta
\frac{k+1}{k^2}
a_{k-2},
$$
with the convention $a_0=1$ and 
$
a_1=2+\frac{\delta}{2}+\frac{1-a^{-1}}{2}.
$
Note that for $z=0$ we have $\delta=0$ and $a=1$, hence
$$
a_{k}^{0}:= a_k (z=0) = 2\prod_{j=2}^{k}\left(1+\frac{2}{j+1}\right)=\frac{(k+2)(k+3)}{6}.
$$
As a consequence, $g_k = {a_k \over a_{k}^{0} }$ satisfies $g_0=1$, $g_1=1+\frac{\delta}{4}+\frac{1-a^{-1}}{4}$, and
\begin{align}
g_{k}&=m_1(k,a) g_{k-1}\label{eqn:recf}
-m_2(k,a)
g_{k-2},\\
m_1(k,a)&=1+\frac{\delta}{k+3}+\frac{1-a^{-1}}{k(k+3)},\notag\\
m_2(k,a)&=\delta\frac{(k+1)^2}{k(k+2)(k+3)}.\notag
\end{align}
We cannot see an a priori reason why this equation could be solved, but it can be. We remark that the function $u_k=u_k(a)$ from (\ref{eqn:u})
also satisfies the induction
\begin{equation}\label{eqn:recu}
u_{k}=m_1(k,a) u_{k-1}
-m_2(k,a)
u_{k-2}.
\end{equation}
The reader who would like to check the above equation can substitute $\delta=a-\frac{1}{a}$ and verify that the Laurent series in $a$ on both sides of (\ref{eqn:recu}) coincide.
This equation  implies
\begin{equation}\label{eqn:recu2}
g_{k-1} { u_{k} \over u_{k-1}} = m_1(k,a) g_{k-1} -  m_2 (k,a) g_{k-1} { u_{k-2} \over u_{k-1}}
\end{equation}
Subtracting (\ref{eqn:recu2}) from (\ref{eqn:recf}) gives
$$\delta_{k}:=  g_{k}- g_{k-1} { u_{k}\over u_{k-1}} =  - m_2 (k,a) \Big( g_{k-2} - g_{k-1} { u_{k-2} \over u_{k-1}} \Big) = m_2(k,a) { u_{k-2} \over u_{k-1}} \delta_{k-1},$$
which yields
\begin{equation}
\delta_k = \left(\prod_{j=2}^{k} m_2(j,a) { u_{j-2}\over u_{j-1}}\right) \delta_{1}=36 \frac{1}{ (k+3)!} {k+1 \over k+2} \delta^{k-1} {u_0 \over u_{k-1}} \delta_{1}\label{eqn:mk}.
\end{equation}
Together with
${g_k \over u_k} = {g_{k-1} \over u_{k-1}}+ {\delta_k \over u_k}$, from (\ref{eqn:mk}) we obtain
$$ {g_k \over u_k}=36\sum_{j=1}^k \frac{1}{ (j+3)!} {j+1 \over j+2} \delta^{j-1} {u_0 \over u_{j-1}u_j} \delta_{1}+\frac{g_0}{u_0}.$$
A calculation gives $u_0\delta_1=\frac{1}{6}(a-a^{-1})(1+a^{-1})$, so that
\begin{equation}\label{eqn:interm}
{g_k \over u_k}=6(1+a^{-1})\sum_{j=1}^k \frac{1}{ (j+3)!} {j+1 \over j+2} \delta^{j} \frac{1}{ u_{j-1}u_j}+\frac{1}{u_0}.
\end{equation}
Note that
$$
\frac{1}{u_{j-1}u_j}=\frac{(j+2)(j+3)}{1-a^{-1}}\left(\frac{1}{u_{j-1}}-\frac{1}{u_j}\right),
$$
which simplifies (\ref{eqn:interm}) into
\begin{align*}
{g_k \over u_k}&=6\frac{1+a^{-1}}{1-a^{-1}}\sum_{j=1}^k \frac{1}{ j!(j+2)}\left(\frac{1}{u_{j-1}}-\frac{1}{u_j}\right) \delta^{j}+\frac{1}{u_0}\\
&=
6\frac{1+a^{-1}}{1-a^{-1}}\left(
\sum_{j=1}^{k-1}\frac{1}{u_j}\left(\frac{\delta}{(j+1)!(j+3)}-\frac{1}{j!(j+2)}\right)\delta^{j}
-\frac{1}{k!(k+2)}\frac{\delta^k}{u_k}+\frac{\delta}{3u_0}
\right)
+\frac{1}{u_0}\\
&=
6\frac{a+1}{a-1}\left(
\sum_{j=1}^{k-1}
\left(\frac{a}{(j+2)!}-\frac{1}{(j+1)!}\right)
\delta^{j}
-\frac{1}{k!(k+2)}\frac{\delta^k}{u_k}+\frac{\delta}{3u_0}
\right)
+\frac{1}{u_0}\\
&=
6\frac{a+1}{a-1}\left(
\frac{1}{a\delta^2}\sum_{j=3}^k\frac{\delta^j}{j!}+
\frac{a\delta^{k-1}}{(k+1)!}-\frac{\delta}{2}
-\frac{1}{k!(k+2)}\frac{\delta^k}{u_k}+\frac{\delta}{3u_0}
\right)
+\frac{1}{u_0}
\end{align*}
where we used $\delta=a-a^{-1}$ and (\ref{eqn:u}) at several steps. The above formula can be written in terms of (\ref{eqn:expo}) and further simplified as
$$
g_k=6u_k\frac{a+1}{a-1}\frac{1}{a\delta^2}e_3^{(k)}(\delta)-\frac{3u_k}{a}
+
6\frac{a+1}{a-1}
\frac{a+(k+2)+a^{-1}(k+1)(k+3)}{(k+3)!}\delta^{k-1}.
$$
By definition of $g_k$, $a_k = g_k a_{k}^{0}$, that is
$$ a_k = u_k\frac{1}{(a-1)^2 \delta } (k+2) (k+3)e_3^{(k)}(\delta)-\frac{u_k}{2a} (k+2) (k+3)
+
\frac{a+1}{a-1}
\frac{a+(k+2)+a^{-1}(k+1)(k+3)}{(k+1)!}\delta^{k-1}.$$
Then, using the normalizing constant (\ref{eqn:normalization}), we obtain
$$\E(d_+)= \frac{d_{N-2}}{Z_N^{(0,z)}}
=
\frac{a_{N-2}}{Z_N^{(0,z)} N^{\frac{(N-2) (N+1)}{2}}}
=
\frac{ \delta a_{N-2} }{ e_{1}^{(N-1)} (\delta)}  
$$
as $\delta=N|z|^2$. We find
\begin{equation}\label{eqn1:d+}
\E (d_+) = \frac{ N (N+1)}{e_1^{(N-1)}(\delta)}
\left(
\frac{u_{N-2}(a)}{(a-1)^2} e_3^{(N-2)}(\delta)-\frac{ \delta u_{N-2}(a)}{2 a}
+ \frac{1}{(a-1)^2}\frac{ a^2+ Na + (N-1)(N+1)}{(N+1)!}\delta^{N-1}
\right)
\end{equation}
which has been defined as $d_{N-2}(a, \delta)$. The formula for $\E(d_-)$ is obtained in the exact same way, with the only difference that $a$ is replaced by $b$. Finally, conditionally on $\{\lambda_1=0,\lambda_2=z\}$,
\eqref{expT} gives
$$
\left(
\begin{array}{c}
\E(|\mathscr{O}_{12}|^2)\\
\E(\mathscr{O}_{11}\mathscr{O}_{22})
\end{array}
\right)
=
\frac{1}{1+a^2}
\left(
\begin{array}{c}
\frac{a^2}{(a+1)^2}\E(d_+) + \frac{a^2}{(a-1)^2} \E(d_-)\\
\frac{1}{(a+1)^2}\E(d_+) + \frac{a^4}{(a-1)^2} \E(d_-)
\end{array}
\right).
$$
We can replace $\E(d_+)$ and $\E(d_-)$ by their exact expressions to obtain the claimed formula.
\end{proof}

\begin{proposition}\label{prop:at00} Let $\sigma\in(0,1/2)$.
Denoting $\omega=\sqrt{N}|z|$, uniformly in $|z|\in[0,N^{-\frac{1}{2}+\sigma}]$ we have
\begin{align}
\E\left(|\mathscr{O}_{12}|^2\mid \la_1=0,\la_2=z\right)&= \frac{N^2}{|\omega|^4}\left(1+\OO(N^{2\sigma-1})\right),\\
\E\left(\mathscr{O}_{11}\mathscr{O}_{22}\mid \la_1=0,\la_2=z\right)&= \frac{N^2}{|\omega|^4}\frac{1+|\omega|^4- e^{-|\omega|^2}}{1-e^{-|\omega|^2}}\left(1+\OO(N^{2\sigma-1})\right).
\end{align}
\end{proposition}

\begin{proof}
We consider asymptotics in Proposition \ref{prop:theformula}. First,
the term $\frac{\delta^{N-1}}{(N+1)!}$  is obviously negligible. Second, we always have $a\geq 1$ and $|b|\geq c/\delta\geq N^{-2\sigma}$, so that 
$u_{N-2}(a)= 1 + \OO(N^{-1})$, $u_{N-2}(b)  = 1 + \OO(N^{2\sigma-1})$. Moreover, 
$$
e_1^{(N-1)}(\delta) = (e^{\delta} -1)  \left(1 + \OO(N^{-1})\right),\ 
e_3^{(N-2)}(\delta) = (e^{\delta} -1- \delta - \frac{\delta^2}{2})  \left(1 + \OO(N^{-1})\right).
$$
From (\ref{eqn:abd}), we have
$ \frac{a^2}{(1+a)^2 (1-a)^2} = \frac{b^2}{(1+b)^2 (1-b)^2} = \frac{1}{\delta^2},$
so that
%\begin{align*}
%\frac{a^2}{(a+1)^2}\E(d_+) & =  \frac{ N^2 }{ |\omega|^4 (e^{|\omega|^2} -1)}
%\left( e^{|\omega|^2} -1- |\omega|^2 - \frac{|\omega|^4}{2} - \frac{ \delta^3 a}{2 (a+1)^2}
%\right)  \left(1 + \OO(N^{-1})\right) \\
%\frac{1}{(a-1)^2}\E(d_-) = \frac{b^2}{(b+1)^2}\E(d_-) & =  \frac{ N^2 }{ |\omega|^4 (e^{|\omega|^2} -1)}
%\left( e^{|\omega|^2} -1- |\omega|^2 - \frac{|\omega|^4}{2} - \frac{ \delta^3 a^2 }{2 b (a-1)^2}
%\right)  \left(1 + \OO(N^{-1})\right).
%\end{align*}
%\cor
(\ref{eqn1:d+}) and its analogue for $d_-$ give
\begin{align}
\frac{a^2}{(a+1)^2}\E(d_+) & =  \frac{ N^2 }{ \delta^2 (e^{\delta} -1)}
\left( e^{\delta} -1- \delta - \frac{\delta^2}{2} - \frac{ \delta (a-1)^2}{2a}
\right)  \left(1 + \OO(N^{-1})\right),\label{eqn:d+}\\
\frac{a^2}{(a-1)^2}\E(d_-) & =  \frac{ N^2 }{\delta^2 (e^{\delta} -1)}
\left( a^2\left(e^{\delta} -1- \delta - \frac{\delta^2}{2}\right) + \frac{\delta a}{2}(a+1)^2
\right)  \left(1 + \OO(N^{2\sigma-1})\right).\label{eqn:d-}
\end{align}
%
%\nc
%
%In order to combine the two, we simplify the $\delta^3$ part in the following way:
%\begin{align*}
%\frac{1}{1+a^2} \left(\frac{ \delta^3 a}{2 (a+1)^2} + \frac{ \delta^3 a^2 }{2 b (a-1)^2} \right)
%& = \frac{a \delta^3 }{2(1+a^2)} \left(\frac{ 1}{ (a+1)^2} - \frac{ a^2 }{  (a-1)^2} \right)\\
%&=  - \frac{a \delta^3 }{2(1+a^2)} \left(\frac{ (1+a^2) (a^2 +2a -1)}{ (a+1)^2 (a-1)^2}  \right) \\
%&=  - \frac{a \delta^3 }{2} \left(\frac{ a^2 +2a -1}{ (a+1)^2 (a-1)^2}  \right) \\
%%&=  - \frac{ \delta^2 }{2} \left(\frac{ a^2 +2a -1}{ (a+1) (a-1)}  \right) \\
%&= - \frac{\delta^2 + 2 \delta}{2} = - \frac{|\omega|^4}{2} - |\omega|^2
%\end{align*}
%
We observe that
$$
\frac{1}{a^2+1}\left(
- \frac{ \delta (a-1)^2}{2a}
+ \frac{\delta a}{2}(a+1)^2
\right)=\delta+\frac{\delta^2}{2},
$$
and the previous three equations give
\begin{multline*}
\E(|\mathscr{O}_{12}|^2\mid \la_1=0,\la_2=x) = \frac{1}{1+a^2} \left( \frac{a^2}{(a+1)^2}\E(d_+) + \frac{a^2}{(a-1)^2} \E(d_-) \right) \\
=  \frac{ N^2 }{ \delta^2 (e^{\delta} -1)}
\left( e^{\delta} -1- \delta - \frac{\delta^2}{2} + \delta + \frac{\delta^2}{2}
\right)  \left(1 + \OO(N^{2\sigma-1})\right) 
=  \frac{ N^2 }{ |\omega|^4}
 \left(1 + \OO(N^{2\sigma-1})\right).
\end{multline*}
%
%\cor
%\begin{align*}
%\E(|\mathscr{O}_{12}|^2) & =  \frac{ N^2 }{ |\omega|^4 (e^{|\omega|^2} -1)}\left(e^{|\omega|^2} -1- |\omega|^2 - \frac{|\omega|^4}{2}+|\omega|^2+\frac{|\omega|^4}{2}\right)
%\end{align*}
%\nc
%
The similar computation for $\E(\mathscr{O}_{11} \mathscr{O}_{22}\mid \la_1=0,\la_2=z)$ involves the two terms
\begin{align*}
\frac{1}{(a+1)^2}\E(d_+) & =  \frac{ N^2 }{ e^{\delta} -1} \left( \frac{1}{(a-1)^2 (a+1)^2}
\left( e^{\delta} -1- \delta - \frac{\delta^2}{2} \right) - \frac{ \delta }{2 a (a+1)^2} \right)
 \left(1 + \OO(N^{-1})\right), \\
\frac{a^4}{(a-1)^2}\E(d_-) &  = \frac{ N^2 }{ e^{\delta} -1} \left( \frac{a^6}{(a-1)^2 (a+1)^2}
\left( e^{\delta} -1- \delta - \frac{\delta^2}{2} \right) + \frac{ \delta a^5 }{2 (a-1)^2} \right)
 \left(1 + \OO(N^{2\sigma-1})\right).
 \end{align*}
Moreover, some algebra gives
\begin{align*}
&\frac{1}{1+a^2} \left( \frac{1}{(a-1)^2 (a+1)^2} + \frac{a^6}{(a-1)^2 (a+1)^2} \right)
=
 \frac{1 + \delta^2}{\delta^2},\\
&\frac{1}{1+a^2} \left( \frac{ \delta a^5 }{2 (a-1)^2} - \frac{ \delta }{2 a (a+1)^2} \right)
=
 \frac{ (\delta+1)( \delta^2 + \delta  +2) }{ 2  \delta }.
\end{align*}
Once combined, these  four equations yield
\begin{align*}
\E\left(\mathscr{O}_{11}\mathscr{O}_{22}\mid\la_1=0,\la_2=z\right) 
& = \frac{1}{1+a^2} \left( \frac{1}{(a+1)^2}\E(d_+) + \frac{a^4}{(a-1)^2} \E(d_-) \right) \\
& =
\frac{ N^2 }{ e^{\delta} -1} \left(  \frac{1 + \delta^2}{\delta^2}
\left( e^{\delta} -1- \delta - \frac{\delta^2}{2} \right) +  \frac{ (\delta+1)( \delta^2 + \delta  +2) }{ 2  \delta } \right)
 \left(1 + \OO(N^{2 \sigma-1})\right) \\ 
& = 
\frac{ N^2 }{ |\omega|^4} \frac{1+ |\omega|^4 - e^{- |\omega|^2}}{1-e^{-|\omega|^2} } \left(1 +\OO(N^{2\sigma-1})\right),
\end{align*}
which concludes the proof.
\end{proof}

\subsection{The annealed off-diagonal overlap: second moments in the general case.}\ 
We can now prove Theorem \ref{thm:meso2}. We closely follow the method developed first in our alternative proof of Theorem \ref{meandiag}, in Subsection 
\ref{subsec:1stpatch}, then in our proof of Theorem \ref{thm:expoff} in Subsection \ref{subsec:2ndpatch}. 
In particular, following the proof of Theorem \ref{thm:expoff},
we denote $\nu_{z_1,z_2}$ the measure (\ref{eqn:partitionfunction}) conditioned to $\lambda_1=z_1,\la_2=z_2$.

\begin{proof}[Proof of Theorem \ref{thm:meso2}]
Remember the notation (\ref{notationd}). Assume we can prove  (under the hypothesis of Theorem \ref{thm:meso2}) that
\begin{align}
\E_{\nu_{z_1,z_2}}(d_+)&=(1-|z_1|^2)(1-|z_2|^2)\E_{\nu_{0,z_2-z_1}}(d_+)\left(1+\OO(N^{-2\kappa+\e})\right),\label{eqn1}\\
\E_{\nu_{z_1,z_2}}(d_-)&=(1-|z_1|^2)(1-|z_2|^2)\E_{\nu_{0,z_2-z_1}}(d_-)\left(1+\OO(N^{-2\kappa+\e})\right).\label{eqn2}
\end{align}
From (\ref{eqn:d+}) and (\ref{eqn:d-}), a calculation gives $\E_{\nu_{0,z_2-z_1}}(d_+)>0$, $\E_{\nu_{0,z_2-z_1}}(d_-)>0$, so that (\ref{eqn1}), (\ref{eqn2}) together with (\ref{expT})
give
$$
\left(
\begin{array}{c}
\E_{\nu_{z_1,z_2}}(|\mathscr{O}_{12}|^2)\\
\E_{\nu_{z_1,z_2}}(\mathscr{O}_{11}\mathscr{O}_{22})
\end{array}
\right)
=(1-|z_1|^2)(1-|z_2|^2)
\left(
\begin{array}{c}
\E_{\nu_{0,z_2-z_1}}(|\mathscr{O}_{12}|^2)\\
\E_{\nu_{0,z_2-z_1}}(\mathscr{O}_{11}\mathscr{O}_{22})
\end{array}
\right)\left(1+\OO(N^{-2\kappa+\e})\right).
$$
Together with Proposition \ref{prop:at00}, this concludes the proof.
We therefore only need to show  (\ref{eqn1}). The proof for (\ref{eqn2}) is identical up to trivial adjustments.\\

\noindent{\it First step: small cutoff.}
Our test function of interest and its short-range cut version are
\begin{align}\label{eqn:testfct}
g_{z_1,z_2}(\lambda)&=\left(1+\frac{1}{N|z_1-\lambda|^2}\right)\left(1+\frac{1}{N|z_2-\lambda|^2}\right)-\frac{a}{N^2}\frac{1}{|z_1-\lambda|^2|z_2-\la|^2}\\
h_{z_1,z_2}(\lambda)&=\left(1+\frac{1}{N|z_1-\lambda|^2}\mathds{1}_{\la\not\in\mathscr{B}}\right)\left(1+\frac{1}{N|z_2-\lambda|^2}\mathds{1}_{\la\not\in\mathscr{B}}\right)-\frac{a}{N^2}\frac{1}{|z_1-\lambda|^2|z_2-\la|^2}\mathds{1}_{\la\not\in\mathscr{B}}\label{eqn:testfctII}
\end{align}
where $\mathscr{B}=\{|\la-z_1|<N^{-A}\}\cup\{|\la-z_2|<N^{-A}\}$.
We first prove
(\ref{eqn:elementaryII}) for our new definition of $g_{z_1,z_2},\, h_{z_1,z_2}$.
It follows from
\begin{align}
&\E_{N-2}\left(\prod_{i=3}^N\left((|z_1-\lambda_i|^2+\frac{1}{N})(|z_2-\lambda_i|^2 +\frac{1}{N})-\frac{a}{N^2}\right)e^{-N(|z_1|^2-1)-N(|z_2|^2-1)}\right)\\
&-
\E_{N-2}\left(\prod_{i=3}^N\left((|z_1-\lambda_i|^2 +\frac{\mathds{1}_{\la_i\not\in\mathscr{B}}}{N})(|z_2-\lambda_i|^2 +\frac{\mathds{1}_{\la_i\not\in\mathscr{B}}}{N})-\frac{a}{N^2}\mathds{1}_{\la_i\not\in\mathscr{B}}\right)e^{-N(|z_1|^2-1)-N(|z_2|^2-1)}\right)=\OO(N^{-A/2+C}),
\label{eqn:elem1III}
\end{align}
and (\ref{eqn:elem2II}), for some $C$ which does not depend on $A$.
Equation  (\ref{eqn:elem1III}) holds as the left hand side is bounded by
\begin{multline}
\E_{N-2}\left(\prod_{i=3}^N\left((|z_1-\lambda_i|^2 +\frac{1}{N})(|z_2-\lambda_i|^2 +\frac{1}{N})\right)(1-\prod_3^N\mathds{1}_{\la_i\not\in\mathscr{B}})e^{-N(|z_1|^2-1)-N(|z_2|^2-1)}\right)\\
=\OO(N^{-A/2})\E_{N-2}\left(\prod_{i=3}^N\left((|z_1-\lambda_i|^2 +\frac{1}{N})^2(|z_2-\lambda_i|^2 +\frac{1}{N})\right)^2e^{-2N(|z_1|^2-1)-2N(|z_2|^2-1)}\right)^{1/2},\label{eqn:boundIII}
\end{multline}
where we used the Cauchy-Schwarz inequality and 
$$
\left|(|z_1-\lambda_i|^2 +\frac{1}{N})(|z_2-\lambda_i|^2 +\frac{1}{N})-\frac{a}{N^2}\right|\leq (|z_1-\lambda_i|^2 +\frac{1}{N})(|z_2-\lambda_i|^2 +\frac{1}{N}).
$$
Indeed, after rescaling and shifting and introducing $z$ so that $|z|^2=\delta$, the above bound follows from 
$$
2(1+|\la|^2)(1+|z-\la|^2)-a= ((1+|\la|^2)(1+|z-\la|^2)-\delta)+ (1+|\la|^2)(1+|z-\la|^2)+(\delta-a)\geq 1+\delta-a=1-\frac{1}{a}\geq 0.
$$
We used $(1+|\la|^2)(1+|z-\la|^2)\geq \delta$, as proved by a simple optimization.
The last expectation in (\ref{eqn:boundIII}) has size order at most $N^{C}$ from Lemma \ref{lem:generalmoments}, which concludes the proof of
our initial short-range cutoff by choosing $A$ large enough.\\

\noindent{\it Second step:  the long-range contribution concentrates.} 
We smoothly separate the short-range from a long-range contributions in (\ref{eqn:testfctII}):\
\begin{align*}
h_{z_1,z_2}(\lambda)&=e^{h_{z_1,z_2}^{s}(\lambda)+h_{z_1,z_2}^{\ell}(\lambda)},\\
h_{z_1,z_2}^{s}(\lambda)&=  (\log h_{z_1,z_2}(\lambda)) \chi_{z,\delta} (\la),\\
h_{z_1,z_2}^{\ell}(\la)&=  (\log h_{z_1,z_2}(\lambda)) (1 - \chi_{z,\delta} (\la)),
\end{align*}
and, as earlier in this article, we denote $z=(z_1+z_2)/2$, recall $|z_1-z_2|<N^{-\frac{1}{2}+\sigma}$, $\chi_{z,\delta} (\la)= \chi \Big(N^{\frac{1}{2}-\delta}|z-\la| \Big)$, and choose  $\delta \in (\sigma, \kappa)$. Note that our notation $\delta$ in this step of the proof is unrelated to $\delta= N | \lambda_{1} - \lambda_{2} |^2$ in the previous step. 
We define $f^{\ell}_{z_1,z_2}(\la)= (\frac{1}{N |z_1 - \la|^2}+\frac{1}{N |z_1 - \la|^2})(1 - \chi_{z,\delta} (\lambda))$ and $\bar f_{z_1,z_2}=(N-2)\frac{1}{ \pi} \int_{\mathbb{D}} f^{\ell}_{z_1,z_2}(\la) \dd m(\lambda)$. Note that
$$
\left|\sum_{i=3}^Nh^\ell_{z_1,z_2}(\lambda_i)-\bar f_{z_1,z_2}\right|\leq \left|\sum_{i=3}^Nf^\ell_{z_1,z_2}(\lambda_i)-\bar f_{z_1,z_2}\right|+\left(\sum_{i=3}^N\frac{N^{2\sigma}}{N^2|z_1-\lambda_i|^4}(1 - \chi_{z,\delta} (\lambda_i))+
\sum_{i=3}^N\frac{N^{2\sigma}}{N^2|z_2-\lambda_i|^4}(1 - \chi_{z,\delta} (\lambda_i))
\right),
$$
where we used $a\leq N^{2\sigma}$.

With the exact same reasoning as from (\ref{eqn:concentration}) to (\ref{eqn:est1}) (with $\nu_z$ replaced by $\nu_{z_1,z_2}$), we obtain that $\sum_{i=3}^Nf^\ell_{z_1,z_2}(\lambda_i)-\bar f_{z_1,z_2}$ is exponentially concentrated
on scale $N^{-2\delta+\e}$. Moreover, similarly, to (\ref{eqn:est2}), we now have
$$
\E_{\nu_{z_1,z_2}}\left(e^{\alpha\sum_{i=2}^N\frac{N^{2\sigma}}{N^2|z_1-\lambda_i|^4}(1-\chi_{z,\delta}(\lambda_i))}\right)\leq e^{C\alpha^2 N^{4\sigma-8\delta+\e}+C\alpha N^{2\sigma-2\delta+\e}}.
$$
With $\alpha=c N^{2\delta}$, we therefore obtain
$$
\mathbb{P}_{\nu_{z_1,z_2}}(A)<e^{-c N^{\e}},\ \mbox{where}\  A=\left\{
\left|
\sum_{i=2}^Nh^\ell_{z_1,z_2}(\lambda_i)-\bar f_{z_1,z_2}
\right|>N^{2\sigma-2\delta+\e}
\right\}.
$$
This yields, for some $p=1+c(\kappa)$, $c(\kappa)>0$ and some $q,r>1$,
\begin{equation}\label{eqn:concIII}
\mathbb{E}_{\nu_{z_1,z_2}}\left(e^{\sum_i(h^s_{z_1,z_2}(\lambda_i)+h^\ell_{z_1,z_2}(\lambda_i))}\mathds{1}_{A}\right)
\leq 
\mathbb{E}_{\nu_{z_1,z_2}}\left(e^{p\sum_ih^s_{z_1,z_2}(\lambda_i)}\right)^{1/p}
\mathbb{E}_{\nu_{z_1,z_2}}\left(e^{q\sum_i h^\ell_{z_1,z_2}(\lambda_i)}\right)^{1/q}
\mathbb{P}(A)^{1/r}\leq e^{-c N^\e}.
\end{equation}
Here we used that the third term has size order $e^{-c N^{\e}}$, the second one is of order  $e^{q\bar f_{z_1,z_2}}=\OO(N^{C})$, and so is the first one from Lemma \ref{lem:momentII}.
Moreover,
\begin{multline*}
\mathbb{E}_{\nu_{z_1,z_2}}\left(e^{\sum_i(h^s_{z_1,z_2}(\lambda_i)+h^\ell_{z_1,z_2}(\lambda_i))}\mathds{1}_{A^c}\right)
=
(1+\OO(N^{2\sigma-2\delta+\e}))\mathbb{E}_{\nu_{z_1,z_2}}\left(e^{\sum_ih^s_{z_1,z_2}(\lambda_i)+\bar f_{z_1,z_2}}\right)\\-
(1+\OO(N^{2\sigma-2\delta+\e}))\mathbb{E}_{\nu_{z_1,z_2}}\left(e^{\sum_ih^s_{z_1,z_2}(\lambda_i)+\bar f_{z_1,z_2}}\mathds{1}_{A^c}\right),
\end{multline*}
Following similar arguments as (\ref{eqn:concII}), (\ref{eqn:useless}), still relying on Lemma \ref{lem:momentII}, we finally obtain 
$$
\mathbb{E}_{\nu_{z_1,z_2}}\left(\prod_{i=3}^Nh_{z_1,z_2}(\lambda_i)\right)
=
(1+\OO(N^{2\sigma-2\delta+\e}))e^{\bar f_{z_1,z_2}}\ \mathbb{E}_{\nu_{z_1,z_2}}\left(e^{\sum_ih^s_{z_1,z_2}(\lambda_i)}\right)
.
$$
%From (\ref{eqn:scales}) and Cauchy-Schwarz,
%we have
%$$
%\E(\mathscr{O}_{11}\mid\lambda_1=z) 
%=
%\E\left(e^{\sum_{n=2}^N h^{\text{loc}}(\lambda_k)+(N-1)\bar h}\right)
%+
%$$
%
%
%%Let us define $a_{kN} = {1 - \chi_{\epsilon} (\lambda_k) \over |z_1 - \lambda_k|^2} < N^{1-2 \epsilon}$. We were careful to take $\epsilon>{3 \over 8}$ so that $1-2\epsilon < \frac{1}{ 4}$, and $ |X_k|^2 $ are independent $\gamma_1$ variable with finite moments; we can therefore use lemma  \ref{BBC} and state the convergence 
%%$$ \frac{1}{ N} \sum_{k=1}^N  {1 - \chi_{\epsilon} (\lambda_k) \over |z_1 - \lambda_k|^2}  |X_k|^2 \xrightarrow[]{a.s.} \frac{1}{ N} \sum_{k=1}^N  {1 - \chi_{\epsilon} (\lambda_k) \over |z_1 - \lambda_k|^2}.$$
%%The error term converges as $1-2\epsilon < \frac{1}{ 4}$. 
%
%We then use rigidity (see appendix D) to justify
%$$ \sum_{k=2}^N    {1 - \chi_{z, \theta}(\lambda_k) \over N |z - \lambda_k|^2}  = \frac{1}{ \pi} \int_{\mathbb{D}} {1 - \chi \big({ { |\lambda - z|^2 \over \theta^2} \big)} \over |z - \lambda|^2} \dd m(\lambda) + \OO_{\prec} (N^{-2 \epsilon})$$
%and lemma \ref{integral1} gives us the asymptotics of this integral. We conclude
%$$\sum_{k=1}^N h_k^{\text{glob}} =  (1- 2 \epsilon) \log (N) + \log (1 - |z|^2 )  + C +O(?). $$
%$$e^{\sum_{k=1}^N h_k^{\text{glob}}} = e^C (1 - |z|^2 ) N^{1-2\epsilon} $$

\noindent{\it Third step: the local part is invariant.} For our test function $h^s_{z_1,z_2}$, the reader can easily check  the conditions of Lemma \ref{lem:invariance2II}:
the only new ingredient is
$$
\int_{\mathscr{B}^c} \frac{\rd m(\la)}{|z_1-\la|^2|z_2-\la|^2}\leq \frac{\log N}{|z_1-z_2|^2},
$$
so that in this setting, the existence of $r>2$, $d<\kappa$ such that
$
(N\|\nu\|_1)^r\leq N^{d}
$
means that there exists $\e>0$ such that
$
\left(\frac{\log N}{N|z_1-z_2|^2}\right)^{2+\e}\leq N^{2\kappa-\e},
$
i.e.
$
|z_1-z_2|\geq N^{-\frac{1}{2}-\frac{\kappa}{2}+\e},
$
as we assumed by hypothesis.
This gives
$$
\mathbb{E}_{\nu_{z_1,z_2}}\left(e^{\sum_ih_{z_1,z_2}^s(\lambda_i)}\right)=\mathbb{E}_{\nu_{0,z_2-z_1}}\left(e^{\sum_ih_{0,z_2-z_1}^s(\lambda_i)}\right)+\OO\left(e^{-c N^{2\kappa}}\right).
$$
This yields
$$
\mathbb{E}_{\nu_{z_1,z_2}}\left(\prod_{i=3}^Nh_{z_1,z_2}(\lambda_i)\right)
=
(1+\OO(N^{2\sigma-2\delta+\e}))e^{\bar f_{z_1,z_2}-\bar f_{0,z_2-z_1}}\ \mathbb{E}_{\nu_{0,z_2-z_1}}\left(\prod_{i=3}^Nh_{0,z_2-z_1}(\lambda_i)\right).$$
From Lemma \ref{integral1}, $e^{\bar f_{z_1,z_2}-\bar f_{0,z_2-z_1}}=(1-|z_1|^2)(1-|z_2|^2)$. With Proposition \ref{prop:at00}, this concludes the proof.
\end{proof}

\begin{lemma}\label{lem:momentII}
For any $0<\sigma<d<\kappa<1/2$, there exists  $p>1$,  $C>0$ such that for any $|z_1|<1-N^{-1/2+\kappa},|z_1-z_2|<N^{-\frac{1}{2}+\sigma}$, $z=(z_1+z_2)/2$, we have
$$
\mathbb{E}_{\nu_{z_1,z_2}}\left(\prod_{k=3}^N\left(\left(1+\frac{1}{N|\lambda_k-z_1|^2}\right)
\left(1+\frac{1}{N|\lambda_k-z_2|^2}\right)\right)^{p\chi_{z,\delta}(\lambda_k)}
\right)\leq N^C.
$$
\end{lemma}

\begin{proof}
The above left hand side is at most
$$
\frac{\mathbb{E}\left(\prod_{k=3}^N\left(\left(1+\frac{1}{N|\lambda_k-z_1|^2}\right)
\left(1+\frac{1}{N|\lambda_k-z_2|^2}\right)\right)^{(p-1)\chi_{z,\delta}(\lambda_k)}
(|z_1-\la_k|^2+\frac{1}{N})(|z_2-\la_k|^2+\frac{1}{N})e^{-N(|z_1|^2-1)-N(|z_2|^2-1)}
\right)}{
\E\left(\prod_{k=3}^N|z_1-\la_k|^2|z_2-\la_k|^2e^{-N(|z_1|^2-1)-N(|z_2|^2-1)}\right)
}
$$
With Lemma \ref{lem:generalmoments} and the Cauchy-Schwarz inequality, we therefore just need to prove that
$$
\mathbb{E}_{N-2}\left(\prod_{k=3}^N\left(\left(1+\frac{1}{N|\lambda_k-z_1|^2}\right)^{4(p-1)\chi_{z,\delta}(\lambda_k)}\right)\right)\leq N^C.
$$
We can apply Lemma \ref{lem:invariance1} (for $p$ small enough we have $N\|\nu\|_1=\OO(1)$) to compare it to the case $z_1=0$, which is easily shown to be $\OO(N^C)$ by Corollary \ref{kostbis}.
\end{proof}

%One can check that this is 1 for $\delta=0$, $a=1$.

%
%On the microscopic scale ($z=\omega/\sqrt{N}$), the above formula (and its analogue for $\lambda_-$) yields
%\begin{align*}
%\E(|\mathscr{O}_{12}|^2\mid \lambda_1=0,\lambda_2=z)
%&=
%N^2\left(1+\OO\left(\frac{1}{N}\right)\right)
%\left(
%\frac{e^{|w|^2}-1-\frac{|w|^2}{2}}{|w|^6}
%+\frac{|w|^2+2}{|w|^4}
%\right)
%\frac{|w|^2}{e^{|w|^2}-1}
%\\
%&=
%\frac{N^2}{|w|^4}
%\frac{1-(1+3|w|^2+4\frac{|w|^4}{2})e^{-|w|^2}}{1-e^{-|w|^2}}\left(1+\OO\left(\frac{1}{N}\right)\right)
%\end{align*}
%
%The above readily gives, for $1\ll a\ll N$, 
%$$
%f_{N-2}\sim \frac{6}{a^3}e^{\delta},
%$$
%so that
%$$
%\E(|\mathscr{O}_{12}|^2\mid \lambda_1=0,\lambda_2=z)\sim\frac{N^2}{a^4}
%=
%\left(\frac{1}{N|z|^4}\right)^2
%\sim|\E(\mathscr{O}_{12}\mid \lambda_1=0,\lambda_2=z)|^2
%$$
%
%
%
%
%For fixed $j$ we have $u_j(b)\sim_{a\to\infty}\frac{1}{(j+3)b}$ but as $a\ll N$ we have $u_N(b)\sim 1$, so that
%$$
%f_{N-2}(b)\sim -6
%b^{-1}\frac{1}{a^2}e^{\delta}
%\sim
%\frac{6}{a}
%e^{\delta}
%$$
%so that
%$$
%\E(|\mathscr{O}_{11}\mathscr{O}_{22}|\mid \lambda_1=0,\lambda_2=z)\sim N^2
%\sim
%\E(|\mathscr{O}_{11}|\mid \lambda_1=0,\lambda_2=z)
%\E(|\mathscr{O}_{22}|\mid \lambda_1=0,\lambda_2=z)
%$$

\subsection{Proof of Corollary \ref{cor:pseudo}.}\  Following a notation  from \cite{Tre2005}, let $\kappa(\la_j)=\mathscr{O}_{jj}^{1/2}$ be the condition number associated to $\lambda_j$.
As the spectrum of a $G$ is almost surely simple, from \cite[Equation (52.11)]{Tre2005} we know that
$$
\|(z-G)^{-1}\|=\frac{\kappa(\la_j)}{|z-\la_j|}+\OO\left(\sum_{k\neq j}\frac{\kappa(\la_k)}{|z-\la_k|}\right)
$$
as $z\to\lambda_j$. Together with $\sigma_\e(G)=\left\{z: \|z-G\|^{-1}>\e^{-1}\right\}$, this gives the following almost sure asymptotics:
$$
m(\sigma_\e(G)\cap \mathscr{B}_N)\underset{\e\to 0}{\sim}\sum_{\la_j\in\mathscr{B}_N}\pi(\kappa(\la_j)\e)^2.
$$
Denoting $c_N=N^2\int_{\mathscr{B}_N} (1-|z|^2)\frac{\rd m(z)}{\pi}$, we therefore have
$$
\lim_{\e\to0}\mathbb{P}\left(1-c<\frac{m(\sigma_{\e}(G)\cap \mathscr{B}_N)}{\pi\e^2 c_N}<1+c\right)
=
\mathbb{P}\left(\left|\frac{\sum_{\la_j\in\mathscr{B}_n}\mathscr{O}_{jj}}{c_N}-1\right|<c\right).
$$
From Corollary \ref{cor:extremes},
$$
\mathbb{P}(\exists i: \la_i\in \mathscr{B}_N: |\mathscr{O}_{ii}|>N^{10})=\oo(1),
$$
hence we only need to prove
\begin{equation}\label{eqn:suff}
\mathbb{P}\left(\left|\frac{\sum_{\la_j\in\mathscr{B}_n}\mathscr{O}_{jj}}{c_N}-1\right|<c, \mathscr{O}_{ii}\leq N^{10}\ {\rm for\ all}\ \la_i\in \mathscr{B}_N\right)\to 1.
\end{equation}
We proceed by bounding the second moment
\begin{align}
\E\left(\left|\sum_{\la_j\in\mathscr{B}_n}\mathscr{O}_{jj}-c_N\right|^2\prod_{i=1}^N\mathds{1}_{\mathscr{O}_{ii}\leq N^{10}}\right)&\leq 
\E\left(\sum_{\la_j\in\mathscr{B}_n}\mathscr{O}_{jj}^2\mathds{1}_{\mathscr{O}_{jj}\leq N^{10}}\right)\label{eqn:2nd1}\\
&+\E\left(\sum_{\la_i,\la_j\in\mathscr{B}_n,i\neq j}\mathscr{O}_{ii}\mathscr{O}_{jj}\mathds{1}_{\mathscr{O}_{ii},\mathscr{O}_{jj}\leq N^{10}}\right)\label{eqn:2nd2}\\
&+c_N^2-2c_N\E\left(\sum_{\la_i\in\mathscr{B}_n}\mathscr{O}_{ii}\right)+2c_N\E\left(\sum_{\la_i\in\mathscr{B}_n}\mathscr{O}_{ii}\mathds{1}_{\exists \la_j\in\mathscr{B}_N:\mathscr{O}_{jj}\geq N^{10}}\right)
\label{eqn:2nd3}
\end{align}
To bound the first term, note that
$$
\mathbb{E}\left(\mathscr{O}_{jj}^2\mathds{1}_{\la_j\in\mathscr{B}_N,\mathscr{O}_{jj}\leq N^{10}}\right)
\leq
C\int_{\mathscr{B}_N}\int_{0}^{N^{10}} x\,\mathbb{P}(\mathscr{O}_{jj}\geq x\mid\la_j=z)\rd x\rd m(z).
$$
Following the same reasoning as (\ref{eqn:boundby0}), we have
$$
\mathbb{P}\left(\mathscr{O}_{jj}\geq x\mid\la_j=z\right)\leq \mathbb{P}(\mathscr{O}_{jj}\geq x N^{-\e}(1-|z|^2)^{-1}\mid\la_j=0)+\OO(e^{-N^{\e}})
\leq
\frac{C N^{2+2\e}(1-|z|^2)^2}{x^2}+\OO(e^{-N^\e}),
$$
where we used Proposition \ref{prop:at0}. Denoting $a$ the center of $\mathscr{B}_N$,
we therefore bounded the right hand side of (\ref{eqn:2nd1}) by $N^{{3+3\e}} (1-|a|^2)^2 m(\mathscr{B}_N)\ll c_N^2$ because $m(\mathscr{B}_N)\geq N^{-1+2a}$ and $\e$ is arbitrary.

To bound (\ref{eqn:2nd2}), we first consider close eigenvalues and bound $\mathscr{O}_{ii}\mathscr{O}_{jj}\leq \frac{1}{2}(\mathscr{O}_{ii}^2+\mathscr{O}_{jj}^2)$:
\begin{multline*}
\E\left(\sum_{\la_i,\la_j\in\mathscr{B}_n,i\neq j}\mathscr{O}_{ii}\mathscr{O}_{jj}\mathds{1}_{\mathscr{O}_{ii},\mathscr{O}_{jj}\leq N^{10},|\la_i-\la_j|<N^{-\frac{1}{2}+\e}}\right)
\leq
C N 
\E\left(\mathscr{O}_{11}^2\mathds{1}_{\mathscr{O}_{11}\leq N^{10},\la_1\in\mathscr{B}_N}|\{j:|\la_j-\la_1|\leq N^{-\frac{1}{2}+\e}\}|\right)\\
\leq C N^{1+2\e}\E\left(\mathscr{O}_{11}^2\mathds{1}_{\mathscr{O}_{11}\leq N^{10},\la_1\in\mathscr{B}_N}\right)m(\mathscr{B}_N)
\leq C N^{3+4\e}(1-|a|^2)^2\ll c_N^2.
\end{multline*}
In the second estimate, we used the local law for Ginibre matrices: from \cite[Theorem 4.1]{BouYauYin2014II} the above number of close eigenvalues is at most $C N^{2\e}$ for some large $C$, with probability at least $1-N^{-D}$ for arbitrarily large $D$. 
The third estimate was obtained in the same way we bounded (\ref{eqn:2nd1}).

For eigenvalues at mesoscopic distance in $[N^{-\frac{1}{2}+\e},N^{-\frac{1}{2}+\kappa-a}]$, the contribution of (\ref{eqn:2nd2}) is obtained thanks to (\ref{eqn:decorrel}):
\begin{align*}
&\E\left(\sum_{\la_i,\la_j\in\mathscr{B}_n,N^{-\frac{1}{2}+\e}<|\la_i-\la_j|}\mathscr{O}_{ii}\mathscr{O}_{jj}\right)
=N(N-1)\E\left(\mathscr{O}_{11}\mathscr{O}_{22}\mathds{1}_{\la_1,\la_2\in\mathscr{B}_N,N^{-\frac{1}{2}+\e}<|\la_1-\la_2|}\right)\\
=&
N(N-1)\int_{\mathscr{B}_N^2\cap\{|z_1-z_2|>N^{-\frac{1}{2}+\e}\}}\E\left(\mathscr{O}_{11}\mathscr{O}_{22}\mid\la_1=z_1\la_2=z_2\right)\frac{\rd m(z_1)}{\pi}\frac{\rd m(z_2)}{\pi}
+
\OO(e^{-c N^{\e}})\\
=&
N(N-1)\int_{\mathscr{B}_N^2\cap\{|z_1-z_2|>N^{-\frac{1}{2}+\e}\}}N^2(1-|z_1|^2)(1-|z_2|^2)(1+\OO(N^{-\e}))\frac{\rd m(z_1)}{\pi}\frac{\rd m(z_2)}{\pi}
+
\OO(e^{-c N^{\e}})\\
=&c_N^2(1+\OO(N^{-\e})).
\end{align*}
Finally, the line (\ref{eqn:2nd3}) is easily shown to be of order $-c_N^2(1+N^{-\varepsilon})$ thanks to (\ref{eqn:O11}) and (\ref{eqn:tail}). We conclude that the left hand side of (\ref{eqn:2nd1}) is at most $N^{-\e}c_N^2$, which concludes the proof of (\ref{eqn:suff}) by Markov's inequality.

\section{Translation invariance for conditioned measures}\label{sec:translation}

Recall that the Ginibre kernel is
$$
K_N(z,w)=\frac{N}{\pi}e^{-N\left(\frac{|z|^2}{2}+\frac{|w|^2}{2}\right)}\sum_{k=0}^{N-1}\frac{(Nz\bar w)^k}{k!}.
$$
We also denote its bulk limit as $k_N(z,w)=\frac{N}{\pi}e^{-N(\frac{|z|^2}{2}+\frac{|w|^2}{2}-z\bar w)}$.

\begin{lemma}\label{lem:kernelapprox}
Let $\kappa>0$. There exists $c=c(\kappa)>0$ such that
for any $|z|,|w|\leq 1-N^{-\frac{1}{2}+\kappa}$, we have
$$
K_N(z,w)=k_N(z,w)+\OO(e^{-cN^{2\kappa}})
$$
\end{lemma}
\begin{proof}
This is a straightforward adaptation of the proof of \cite[Lemma 4.2]{BouYauYin2014II}.
\end{proof}

We denote $\mathscr{B}_{a,\delta}$ the ball with center $a$ and radius $N^{-\frac{1}{2}+\delta}$.
\begin{lemma}\label{lem:invariance1}
Let $0<\delta<\kappa<1/2$ be fixed constants. Consider any $\mathbb{C}$-valued measurable function $f$ supported on $\mathscr{B}_{0,\delta}$, $|a|\leq 1-N^{-\frac{1}{2}+\kappa}$, and $\nu(z)=e^{f(z)}-1$. 
For any $r>2$ there exists  $c, C>0$ depending only on $\kappa,\delta,r$ such that 
$$
\mathbb{E}\left(e^{\sum_{i=1}^N f(\la_i-a)}\right)=
\mathbb{E}\left(e^{\sum_{i=1}^N f(\la_i)}\right)+\OO\left(e^{-cN^{2\kappa}}
e^{C(N\|\nu\|_1)^r}\right).
$$
\end{lemma}

\begin{proof} 
Let $K_N^{(a)}(z,w)=K_N(z-a,w-a)$. We define $\|K\|=\sup_{z,w\in{\rm supp}(\nu)}|K(z,w)|$. We successively compare linear statistics for $K_N^{(a)},k_N^{(a)},k_N $ and $K_N$. First note that 
$k_N$ is the kernel of a translation invariant point process, so that comparison between $k_N^{(a)}$ and $k_N$ is trivial. For the other steps, we use 
 \cite[Lemma 3.4.5]{AndGuiZei2010} and obtain
\begin{multline}\label{eqn:initial}
\left|\mathbb{E}\left(e^{\sum_{i=1}^N f(\la_i-a)}\right)-\mathbb{E}\left(e^{\sum_{i=1}^N f(\la_i)}\right)\right|\\
\leq
\sum_{n=1}^\infty\frac{n^{1+\frac{n}{2}}}{n!}\|\nu\|_1^n
\left(
\max(\|K_N^{(a)}\|,\|k_N^{(a)}\|)^{n-1}\|K_N^{(a)}-k_N^{(a)}\|
+
\max(\|K_N\|,\|k_N\|)^{n-1}\|K_N-k_N\|
\right).
\end{multline}
Clearly, $\|K_N\|\leq \frac{N}{\pi}$ and we bound $\|K_N^{(a)}-K_N\|$ with Lemma \ref{lem:kernelapprox}. We conclude that for a universal large enough $C$, and $1/r+1/s=1$, we have
$$
\left|\mathbb{E}\left(e^{\sum_{i=1}^N f(\la_i-a)}\right)-\mathbb{E}\left(e^{\sum_{i=1}^N f(\la_i)}\right)\right|
\leq
e^{-cN^{2\kappa}}\sum_{n=1}^\infty\frac{n^{1+\frac{n}{2}}}{(n!)^{1/s}}\frac{1}{(n!)^{1/r}}\left(\frac{N}{\pi}\|\nu\|_1\right)^n
\leq
Ce^{-cN^{2\kappa}}
e^{C(N\|\nu\|_1)^r},
$$
where in the last inequality we used H{\"o}lder's inequality and $r>2$.
\end{proof}

\begin{lemma}\label{lem:invariance2}
Remember $\nu(z)=e^{f(z)}-1$ and $\mathscr{B}_{a,\delta}$ is the ball with center $a$ and radius $N^{-\frac{1}{2}+\delta}$. 
Let $0<\delta<\kappa<1/2$ be fixed constants. Consider
any $\mathbb{C}$-valued measurable function $f$ supported on $\mathscr{B}_{0,\delta}$,  and $|a|\leq 1-N^{-\frac{1}{2}+\kappa}$.
Assume also that either (i) or (ii) below holds:
\begin{enumerate}[(i)]
\item ${\rm Re}(f)=0$;
\item there exist $d>0$, $p>1$ and $r>2$  such that $rd<2\kappa$, $f=0$ on $|z|<e^{-N^{d}}$, and  ($f_+=\max({\rm Re} f,0)$)
\begin{align}
&(N\|\nu\|_1)^r\leq N^{d}\label{eqn:strange1},\\
&\log \|e^{\sum_{i=2}^Nf_+(\la_i-a)}\|_{{\rm L}^p} \leq N^d, \label{eqn:strange2} \\
& \log \|e^{\sum_{i=2}^Nf_+(\la_i)}\|_{{\rm L}^p} \leq N^d\label{eqn:strange3}.
\end{align}
where the $L^p$ norm is taken with respect to $\E_{N-1}$.
\end{enumerate}
Then for any $q<2\kappa$, uniformly in $f$ satisfying the above hypothese, we have
$$
\mathbb{E}\left(e^{\sum_{i=2}^N f(\la_i-a)}\mid \la_1=a\right)=
\mathbb{E} \left(e^{\sum_{i=2}^N f(\la_i)}\mid \la_1=0\right)\left(1+\OO\left(e^{-cN^{q}}\right)\right)+\OO\left(e^{-cN^{q}}\right).
$$
\end{lemma}

\begin{proof}
In this proof we first consider the most difficult case (ii), and we will finally mention the simple modifications required for (i). We start with
\begin{equation}\label{eqn:start}
\mathbb{E}_{N}\left(e^{\sum_{i=2}^N f(\la_i-a)}\mid \la_1=a\right)=
\frac{\mathbb{E}_{N-1}\left(e^{\sum_{i=2}^N f(\la_i-a)}\prod_{i=2}^{N}|\la_i-a|^2e^{-N(|a|^2-1)}\right)}
{\mathbb{E}_{N-1}\left(\prod_{i=2}^{N}|\la_i-a|^2e^{-N(|a|^2-1)}\right)}.
\end{equation}
Fix some constants $\kappa_1, \kappa_2$ such that $d<q<\kappa_2<\kappa_1<2\kappa$ and define $\chi^a_j(\la)=\mathds{1}_{|\la-a|<e^{-N^{\kappa_j}}}$, $j=1, 2$.
We first show we can afford imposing $\chi^a_2(\la_i)=0$: for some positive $q$, $r$ such that $p^{-1}+q^{-1}+r^{-1}=1$, we have
\begin{multline}
\mathbb{E}_{N-1}\left(e^{\sum_{i=2}^N f(\la_i-a)}\left(1-\prod_{i=2}^N(1-\chi^a_2(\la_i))\right)\prod_{i=2}^{N}|\la_i-a|^2e^{-N(|a|^2-1)}\right)\\
\leq 
\left\|e^{\sum_{i=2}^N f(\la_i-a)}\right\|_{{\rm L}^p}
\left\|1-\prod_{i=2}^N(1-\chi^a_2(\la_i))\right\|_{{\rm L}^q}
\left\|\prod_{i=2}^{N}|\la_i-a|^2e^{-N(|a|^2-1)}\right\|_{{\rm L}^r}\label{eqn:L3norm}
\end{multline}
where ${\rm L}^p={\rm L}^p(\mathbb{P}_{N-1})$. By hypotheses (\ref{eqn:strange2}) and (\ref{eqn:strange3}), the first norm is at most $e^{c N^d}$. The second is at most
$N\mathbb{P}_{N-1}(|\lambda_2|\leq e^{-N^{\kappa_2}})\leq e^{-cN^{\kappa_2}}$. The third norm is at most $N^{C}$, as a simple consequence of 
Lemma \ref{lem:generalmoments}.
These estimates also hold for $f=0$, so that we proved
\begin{equation}\label{eqn:second}
\mathbb{E}_{N}\left(e^{\sum_{i=2}^N f(\la_i-a)}\mid \la_1=a\right)=
\frac{\mathbb{E}_{N-1}\left(e^{\sum_{i=2}^N f(\la_i-a)}\prod_{i=2}^{N}(1-\chi^a_2(\la_i))|\la_i-a|^2e^{-N(|a|^2-1)}\right)+\OO(e^{-c N^{\kappa_2}})}
{\mathbb{E}_{N-1}\left(\prod_{i=2}^{N}(1-\chi^a_2(\la_i))|\la_i-a|^2e^{-N(|a|^2-1)}\right)+\OO(e^{-c N^{\kappa_2}})}.
\end{equation}
If $|\la_1-a|<e^{-N^{\kappa_1}}$ and $|\la_i-a|>e^{-N^{\kappa_2}}$, $2\leq i\leq N$, we have
$$
\prod_{i=2}^N|\la_i-a|=\left(1+\OO\left(e^{-cN^{\kappa_1}}\right)\right)\prod_{i=2}^N|\la_i-\la_1|,\ \ e^{-N|a|^2}=\left(1+\OO\left(e^{-cN^{\kappa_1}}\right)\right)e^{-N|\la_1|^2}.
$$
The expectation in the numerator of (\ref{eqn:second}) is therefore (in the first equation below $\la_1$ has distribution $U$, the uniform measure on the unit disk with center $a$ and radius $e^{-N^{\kappa_1}}$, with volume $b_N=\pi(e^{-2N^{\kappa_1}})$):
\begin{align}
&
\mathbb{E}_{\mathbb{P}_{N-1}\times U}\left(e^{\sum_{i=2}^N f(\la_i-a)}\prod_{i=2}^{N}(1-\chi^a_2(\la_i))|\la_i-\la_1|^2e^{-N(|\la_1|^2-1)}\right)\left(1+\OO\left(e^{-cN^{\kappa_1}}\right)\right)\notag\\
=&
\frac{e^N Z_N}{Z_{N-1}} \frac{1}{b_N}\mathbb{E}_{{N}}\left(e^{\sum_{i=2}^N f(\la_i-a)}\prod_{i=2}^{N}(1-\chi^a_2(\la_i))\chi^a_1(\la_1)\right)\left(1+\OO\left(e^{-cN^{\kappa_1}}\right)\right)
%=&
%\frac{e^N Z_N}{Z_{N-1}} \left(\frac{1}{b_N}\mathbb{E}_{{N}}\left(e^{\sum_{i=2}^N f(\la_i-a)}\prod_{i=2}^{N}(1-\chi^a_2(\la_i))\chi^a_1(\la_1)\right)+\OO\left(e^{-cN^{\kappa_1}}\right)\right)
\label{eqn:U}
\end{align}
%where for the last equation we used   (\ref{eqn:strange2}). 

We now want to remove the constraint on $(\la_i)_{i=2}^N$, i.e. prove
\begin{equation}\label{eqn:removecuttoff}
\frac{1}{b_N}\mathbb{E}_{{N}}\left(e^{\sum_{i=2}^N f(\la_i-a)}\prod_{i=2}^{N}(1-\chi^a_2(\la_i))\chi^a_1(\la_1)\right)
=
\frac{1}{b_N}\mathbb{E}_{{N}}\left(e^{\sum_{i=2}^N f(\la_i-a)}\chi^a_1(\la_1)\right)+\OO(e^{-cN^{\kappa_2}}).
\end{equation}
This requires a longer argument. Let $\mathscr{B}^a_{i}=\{|z-a|\leq e^{-N^{\kappa_i}}\}$, $i=1$ or $2$, $\xi=\sum_{1}^N\delta_{\la_i}$, $\wt\xi=\sum_{2}^N\delta_{\la_i}$.
Then,
\begin{multline}
\left|e^{\sum_{i=2}^N f(\la_i-a)}(1-\prod_{i=2}^{N}(1-\chi^a_2(\la_i)))\chi^a_1(\la_1)\right| \leq e^{\sum_{i=2}^N{\rm Re}f(\la_i)}\wt\xi(\mathscr{B}^a_{2})\chi^a_1(\la_1)\\
\leq
e^{\sum_{i=2}^N{\rm Re}f(\la_i)}\xi(\mathscr{B}^a_{2}-\mathscr{B}^a_{1})\chi^a_1(\la_1)
+
e^{\sum_{i=2}^N{\rm Re}f(\la_i)}\wt\xi(\mathscr{B}^a_{1})\chi^a_1(\la_1).\label{eqn:xi}
\end{multline}
To bound the first term, we use the negative association property of determinantal point processes for disjoint sets  (see e.g. \cite{Lyo2003}), using $f_+\geq 0$ and $f=0$ on $\mathscr{B}^a_{2}$:
\begin{equation}
\mathbb{E}_N\left(e^{\sum_{i=2}^N{\rm Re}f(\la_i)}\xi(\mathscr{B}^a_{2}-\mathscr{B}^a_{1})\xi(\mathscr{B}^a_{1})\right)\\
\leq
\mathbb{E}_N\left(e^{\sum_{i=2}^Nf_+(\la_i)}\right)\E_N\left(\xi(\mathscr{B}^a_{2}-\mathscr{B}^a_{1})\right)
\E_N\left(\xi(\mathscr{B}^a_{1})\right)\label{eqn:chi}.
\end{equation}
By (\ref{eqn:strange2}) and (\ref{eqn:strange3}), the first expectation above has size order at most $e^{cN^{d}}$. 
The second is of order $e^{-cN^{\kappa_2}}$ and the third one is bounded by $Nb_N(1+\oo(1))$, so that 
the first term in (\ref{eqn:xi}) gives an error $\OO(b_N e^{-c N^{\kappa_2}})$.

For the second term in (\ref{eqn:xi}), we also use the negative association property and $f=0$   on $\mathscr{B}^a_{2}$:
$$\E_N\left(e^{\sum_{i=2}^N{\rm Re}f(\la_i)}\wt\xi(\mathscr{B}^a_{1})\chi^a_1(\la_1)\right)
\leq 
\E_N\left(e^{\sum_{i=2}^Nf_+(\la_i)}\right)\E_N\left(\wt\xi(\mathscr{B}^a_{1})\chi^a_1(\la_1)\right)\leq 
e^{N^d}\E_N\left(\wt\xi(\mathscr{B}^a_{1})\chi^a_1(\la_1)\right).
$$
Together with
\begin{equation}\label{eqn:double1}
\E\left(\wt\xi(\mathscr{B}^a_{1})\chi^a_1(\la_1)\right)\leq \E\left(\xi(\mathscr{B}^a_{1})(\xi(\mathscr{B}^a_{1})-1)\right)=
\int_{(\mathscr{B}^a_{1})^2} |K_N(z_1,z_2)|^2\leq N^2 b_N^2,
\end{equation}
we have proved that the second term in (\ref{eqn:xi}) gives an error $\OO(b_N e^{-c N^{\kappa_2}})$.
This concludes the proof of (\ref{eqn:removecuttoff}), so that the numerator in (\ref{eqn:second}) is 
\begin{multline*}
\frac{e^N Z_N}{Z_{N-1}} \left(\frac{1}{b_N}\mathbb{E}_{{N}}\left(e^{\sum_{i=2}^N f(\la_i-a)}\chi^a_1(\la_1)\right)\left(1+\OO\left(e^{-cN^{q}}\right)\right)+\OO\left(e^{-cN^{q}}\right)\right)
+\OO\left(e^{-cN^{q}}\right)\\
=
\frac{e^N Z_N}{Z_{N-1}} \left(\frac{1}{b_N}\mathbb{E}_{{N}}\left(e^{\sum_{i=2}^N f(\la_i-a)}\chi^a_1(\la_1)\right)\left(1+\OO\left(e^{-cN^{q}}\right)\right)+\OO\left(e^{-cN^{q}}\right)\right),
\end{multline*}
where we used $\frac{e^N Z_N}{Z_{N-1}}\sim c_1 N^{c_2}$ for some $c_1,c_2$, as obtained from (\ref{eqn:partitionfunction}).
In the same way, the denominator in (\ref{eqn:second}) is 
$
\frac{e^N Z_N}{Z_{N-1}} \left(1+\OO\left(e^{-cN^{\kappa_2}}\right)\right)$,
so that we obtained 
\begin{align}\label{eqn:reduction}
\mathbb{E}_{N}\left(e^{\sum_{i=2}^N f(\la_i-a)}\mid \la_1=a\right)\notag
&=\frac{1}{b_N}\mathbb{E}_{N}\left(e^{\sum_{i=2}^N f(\la_i-a)}\chi^a_1(\la_1)\right)\left(1+\OO\left(e^{-cN^{q}}\right)\right)+\OO(e^{-cN^{q}})\\
&=\frac{1}{b_N}\mathbb{E}_{N}\left(e^{\sum_{i=1}^N f(\la_i-a)}\chi^a_1(\la_1)\right)\left(1+\OO\left(e^{-cN^{q}}\right)\right)+\OO(e^{-cN^{q}})\notag\\
&=\frac{1}{Nb_N}\mathbb{E}_{N}\left(e^{\sum_{i=1}^N f(\la_i-a)}\xi(\mathscr{B}^a_1)\right)\left(1+\OO\left(e^{-cN^{q}}\right)\right)+\OO(e^{-cN^{q}}),
\end{align}
where we successively used that fact that $f$ vanishes on $\mathscr{B}_1^a$ and symmetrized.

To conclude the proof, we therefore just need
\begin{equation}\label{eqn:final}
f_a'(0)=f_0'(0)+\OO(e^{-cN^{{2\kappa}}}),\ {\rm where}\ f_a(w)=\frac{1}{Nb_N}\mathbb{E}_{N}\left(e^{\sum_{i=1}^N f(\la_i-a)+w \xi(\mathscr{B}_1^a)}\right).
\end{equation}
From Lemma \ref{lem:invariance1} we know that uniformly on $|w|<1$ we have
$$
f_a(w)=f_0(w)+
\OO(e^{-N^{2\kappa}}),
$$
which proves (\ref{eqn:final}) by Cauchy's theorem, and therefore the lemma in case (ii).

Under the assumption (i), up to (\ref{eqn:reduction}) the results hold and the reasoning is simplified as all $L^p$ norms related to $f$ can be bounded by 1. To justify an analogue of (\ref{eqn:reduction}) and the end of the reasoning, we first replace $f$ by $\wt f=f\mathds{1}_{(\mathscr{B}_1^a)^c}$ and note that
$$
\left|\frac{1}{b_N}\mathbb{E}_{N}\left(\left(e^{\sum_{i=2}^N f(\la_i-a)}-e^{\sum_{i=2}^N \wt f(\la_i-a)}\right)\chi^a_1(\la_1)\right)\right|\leq
\frac{2}{b_N} \E_N(\xi(\mathscr{B}_1^a)(\xi(\mathscr{B}_1^a)-1))=\OO(N^2 b_N),
$$
so that by symmetrizing we now obtain
$$
\mathbb{E}_{N}\left(e^{\sum_{i=2}^N f(\la_i-a)}\mid \la_1=a\right)
=
\frac{1}{Nb_N}\mathbb{E}_{N}\left(e^{\sum_{i=1}^N \wt f(\la_i-a)}\xi(\mathscr{B}^a_1)\right)+\OO(e^{-cN^{\kappa_2}}).
$$
The rest of the proof is identical to case (i).
\end{proof}

We now state and prove an analogue of Lemma \ref{lem:invariance2} when conditioning on two points. We will only need case (ii), as we are interested in expectations in Section \ref{sec:offdiag}, not in
convergence in distribution.

\begin{lemma}\label{lem:invariance2II}
Let $0<\delta<\kappa<1/2$ be fixed constants and $C>0$ fixed, arbitrarily large. Consider
any $\mathbb{C}$-valued measurable function $f$ supported on $\mathscr{B}_{0,\delta}$,  $|a|,|b|\leq 1-N^{-\frac{1}{2}+\kappa}$, and $N^{-C}<|b-a|<N^{-1/2+d}$.
Assume that there exists $d<2\kappa$, $p>1$ and $r>2$ such that  $f=0$ on $|z|<e^{-N^{d}}$, on $|z-(b-a)|<e^{-N^{d}}$, and 
(\ref{eqn:strange1}), (\ref{eqn:strange2}) and (\ref{eqn:strange3}) hold.
Then for any $q<2\kappa$ we have
$$
\mathbb{E}_{N}\left(e^{\sum_{i=3}^N f(\la_i-a)}\mid \la_1=a,\la_2=b \right)=
\mathbb{E}_{N}\left(e^{\sum_{i=3}^N f(\la_i)}\mid \la_1=0,\la_2=b-a\right)\left(1+\OO\left(e^{-cN^{q}}\right)\right)+\OO\left(e^{-cN^{q}}\right).
$$
\end{lemma}

\begin{proof}
We start similarly to the proof of Lemma \ref{lem:invariance2}, by writing
\begin{equation}\label{eqn:startII}
\mathbb{E}_{N}\left(e^{\sum_{i=2}^N f(\la_i-a)}\mid \la_1=a,\la_2=b\right)=
\frac{\mathbb{E}_{N-2}\left(e^{\sum_{i=3}^N f(\la_i-a)}\prod_{i=3}^{N}|\la_i-a|^2e^{-N(|a|^2-1)}\prod_{i=3}^{N}|\la_i-b|^2e^{-N(|b|^2-1)}\right)}
{\mathbb{E}_{N-2}\left(\prod_{i=3}^{N}|\la_i-a|^2e^{-N(|a|^2-1)}\prod_{i=3}^{N}|\la_i-b|^2e^{-N(|b|^2-1)}\right)}.
\end{equation}
Again, we fix some constants $\kappa_1, \kappa_2$ such that $d<\kappa_2<\kappa_1<2\kappa$ and define $\chi^x_j(\la)=\mathds{1}_{|\la-a|<e^{-N^{\kappa_j}}}$, $j=1, 2$, $x=a,b$.
The strict analogue of (\ref{eqn:L3norm}) holds, so that the left hand side of (\ref{eqn:startII}) can be written
\begin{equation}\label{eqn:secondII}
\frac{\mathbb{E}_{N-2}\left(e^{\sum_{i=3}^N f(\la_i-a)}\prod_{i=3}^{N}(1-\chi^a_2(\la_i))|\la_i-a|^2e^{-N(|a|^2-1)}\prod_{i=3}^{N}(1-\chi^b_2(\la_i))|\la_i-b|^2e^{-N(|b|^2-1)}\right)+\OO(e^{-c N^{\kappa_2}})}
{\mathbb{E}_{N-2}\left(\prod_{i=3}^{N}(1-\chi^a_2(\la_i))|\la_i-a|^2e^{-N(|a|^2-1)}\prod_{i=3}^{N}(1-\chi^b_2(\la_i))|\la_i-b|^2e^{-N(|b|^2-1)}\right)+\OO(e^{-c N^{\kappa_2}})}.
\end{equation}
The analogue of (\ref{eqn:U}) then holds exactly in the same way:
the expectation in the numerator of (\ref{eqn:secondII}) is
$$
\frac{e^{2N} Z_N}{|a-b|^2Z_{N-2}} \frac{1}{b_N^2}\mathbb{E}_{{N}}\left(e^{\sum_{i=3}^N f(\la_i-a)}\prod_{i=3}^{N}(1-\chi^a_2(\la_i))(1-\chi^b_2(\la_i))\chi^a_1(\la_1)\chi^b_1(\la_2)\right)\left(1+\OO\left(e^{-cN^{\kappa_1}}\right)\right).
$$
Again, 
we want to remove the constraint on $(\la_i)_{i=3}^N$, i.e. prove
\begin{multline*}
\frac{1}{b_N^2}\mathbb{E}_{{N}}\left(e^{\sum_{i=3}^N f(\la_i-a)}\prod_{i=3}^{N}(1-\chi^a_2(\la_i))(1-\chi^b_2(\la_i))\chi^a_1(\la_1)\chi^b_1(\la_2)\right)\\
=
\frac{1}{b_N^2}\mathbb{E}_{{N}}\left(e^{\sum_{i=3}^N f(\la_i-a)}\chi^a_1(\la_1)\chi^b_1(\la_2)\right)+\OO(e^{-cN^{\kappa_2}}).
\end{multline*}
With the negative association property, the strict analogues of equation (\ref{eqn:xi}), (\ref{eqn:chi}) and (\ref{eqn:double1}) hold, so that
the numerator in (\ref{eqn:secondII}) is 
\begin{multline*}
\frac{e^{2N} Z_N}{|a-b|^2Z_{N-2}} \left(\frac{1}{b_N^2}\mathbb{E}_{{N}}\left(e^{\sum_{i=3}^N f(\la_i-a)}\chi^a_1(\la_1)\chi^b_1(\la_2)\right)+\OO\left(e^{-cN^{\kappa_2}}\right)\right)
+\OO\left(e^{-cN^{\kappa_2}}\right)\\
=
\frac{e^{2N} Z_N}{|a-b|^2Z_{N-2}} \left(\frac{1}{b_N^2}\mathbb{E}_{{N}}\left(e^{\sum_{i=3}^N f(\la_i-a)}\chi^a_1(\la_1)\chi^b_1(\la_2)\right)+\OO\left(e^{-cN^{\kappa_2}}\right)\right)
\end{multline*}
where we used $\frac{e^{2N} Z_N}{Z_{N-2}}\sim c_1 N^{c_2}$ for some $c_1,c_2$, as obtained from (\ref{eqn:partitionfunction}).
In the same way, the denominator in (\ref{eqn:secondII}) is 
$
\frac{e^{2N} Z_N}{|a-b|^2Z_{N-2}}  \left(1+\OO\left(e^{-cN^{\kappa_2}}\right)\right)$,
giving
\begin{align}\label{eqn:reductionII}
&\mathbb{E}_{N}\left(e^{\sum_{i=3}^N f(\la_i-a)}\mid \la_1=a,\la_2=b\right)
=\frac{1}{b_N^2}\mathbb{E}_{N}\left(e^{\sum_{i=3}^N f(\la_i-a)}\chi^a_1(\la_1)\chi^b_1(\la_2)\right)\left(1+\OO(e^{-cN^{q}})\right)+\OO(e^{-cN^{q}})\notag\\
&=\frac{1}{b_N^2}\mathbb{E}_{N}\left(e^{\sum_{i=1}^N f(\la_i-a)}\chi^a_1(\la_1)\chi^b_1(\la_2)\right)\left(1+\OO(e^{-cN^{q}})\right)+\OO(e^{-cN^{q}})\notag\\
&=\frac{1}{N(N-1)b_N^2}\mathbb{E}_{N}\left(e^{\sum_{i=1}^N f(\la_i-a)}\xi(\mathscr{B}^a_1)\xi(\mathscr{B}^b_1)\right)\left(1+\OO(e^{-cN^{q}})\right)+\OO(e^{-cN^{q}}),
\end{align}
where we successively used that fact that $f$ vanishes on $\mathscr{B}_1^a\cup \mathscr{B}_1^b$, $\mathscr{B}_1^a\cap \mathscr{B}_1^b=\varnothing$ (this holds because $|a-b|>N^{-C}$) and symmetrized.
To conclude the proof, we therefore just need $\partial_{z_1z_2}f_{a,b}(0,0)=\partial_{z_1z_2}f_{0,b-a}(0,0)+\OO(e^{-cN^{{q}}})$, where
\begin{equation}\label{eqn:finalII}
f_{a,b}(z_1,z_2)=\frac{1}{N(N-1)b_N^2}\mathbb{E}_{N}\left(e^{\sum_{i=1}^N f(\la_i-a)+z_1 \xi(\mathscr{B}_1^a)+z_2 \xi(\mathscr{B}_1^b)})\right)
\end{equation}
This follows  from  Lemma \ref{lem:invariance1} and Cauchy's Theorem, similarly to the end of the proof of Lemma \ref{lem:invariance2}.
\end{proof}

\section{Andr\'eief's identity and Kostlan's theorem}\label{eqn:Sec}

This section gives applications of Andr\'eief's identity to the conditioned measures of interest in this work. In particular, it proves some slight extensions of Kostlan's theorem (Corollary \ref{Kost1}), following a method from \cite{Dub2017}. The common main tool will be the following classical Lemma, by Andr\'eief \cite{Andreiev} (see \cite{DeiGio} for a short proof). Note that the original proof of Kostlan's theorem \cite{Kos1992} and some of its extensions \cite{HouKriPerVir2006} were based on different arguments.

\begin{lemma}[Andr\'eief's identity] On a measured space $(E, \mathcal{E}, \mu)$ For any functions $(\phi_i, \psi_i)_{i=1}^N \in L_2(\mu)^{2N}$,
$$ \frac{1}{ N! } \int_{E^N} \det \left(\phi_i(\lambda_j) \right) \ \det  \left( \psi_i (\lambda_j) \right) \ \mu (\dd \lambda_1) \dots  \mu ( \dd \lambda_N ) = \det \left(f_{i,j}\right) \quad \text{where} \ f_{i,j} = \int_E \phi_i(\lambda) \psi_j(\lambda) \mu(\dd \lambda).$$
 \end{lemma}
 
\begin{theorem}\label{uncond} Let $E=\mathbb{C}$, $g \in L_2(\mu)$, and  $\{ \lambda_1, \dots, \lambda_N\}$ eigenvalues from the  Ginibre ensemble. Then
$$ \E\left( \prod_{k=1}^N g (\lambda_k)\right)= N^{\frac{N(N-1)}{2}}\det (f_{i,j})_{i,j = 1}^{N} \quad \text{where} \ f_{i,j} = \frac{1}{ (j-1)!} \int \lambda^{i-1} \bar{\lambda}^{j-1} g(\lambda) \mu(\dd \lambda).$$ 
\end{theorem}

\begin{proof}
The following is elementary:
\begin{equation}\label{eqn:momentgam}
\int |\lambda|^{2i}\rd \mu(\lambda)=\frac{i!}{N^i}.
\end{equation}
The proof then follows from Andr\'eief's identity.
\end{proof}

\begin{theorem}\label{cond1}
We have (remember $\mu=\mu^{(N)}$)
$$\E \left(\prod_{k=2}^{N} g (\lambda_k) \mid  \lambda_1 = z\right)= \frac{1}{Z_N^{(z)}} \det(f_{i,j})_{i,j = 1}^{N-1} \quad \text{where} \ f_{i,j} = \frac{1}{ j!} \int \lambda^{i-1} \bar{\lambda}^{j-1} |z -\lambda |^2 g (\lambda)  \mu(\dd \lambda)$$
and $$Z_N^{(z)}=
N^{-\frac{N(N-1)}{2}}e^{(N-1)}\left(N|z|^2\right).$$
\end{theorem}

\begin{proof} 
Using Andr\'eief's identity with $\phi_i(\lambda) = \lambda^{i-1} g(\lambda) |z - \lambda|^2 $, $\psi_j(\lambda)= \lambda^{j-1}$, we find
$$  \E\left(\prod_{k=2}^{N} g (\lambda_k)\mid \lambda_1 = z\right)= \frac{1}{Z_N^{(z)}} \det(f_{i,j})_{i,j = 1}^{N-1}$$
where
$$Z_N^{(z)}=\det(M_{ij})_{i,j = 1}^{N-1},\ \  M_{ij}=\frac{1}{ i !} \int \lambda^{i-1} \bar{\lambda}^{j-1} |z - \lambda|^2 \mu(\dd \lambda).$$
By expanding $|z- \lambda|^2 = |z|^2 + |\lambda|^2 - z \overline{\lambda} - \overline{z} \lambda$, we see that that $M$ is tridiagonal, with entries (remember (\ref{eqn:momentgam}))
$$
M_{ii}=\frac{1}{N^i}+\frac{|z|^2}{iN^{i-1}},\ 
M_{i,i+1}=-\frac{\overline{z}}{N^i},\ 
M_{i,i-1}=- \frac{z}{i N^{i-1}}.
$$
Denoting $x=N |z|^2$ and $d_k=\det((M_{ij})_{1\leq i,j\leq k})$, with the convention $d_0=1$ we have
\begin{align*}
d_1&=\frac{1+x}{N}\\
d_k&=\left(1+\frac{x}{k}\right)\frac{1}{N^k}d_{k-1}-\frac{x}{k}\frac{1}{N^{2k-1}}d_{k-2}
\end{align*}
so that $a_k=d_k N^{\frac{k(k+1)}{2}}$ satisfies $a_0=1$, $a_1=1+x$,
$$
a_k=\left(1+\frac{x}{k}\right)a_{k-1}-\frac{x}{k}a_{k-2}.
$$
This gives $a_k=e^{(k)}(x)$ by an immediate induction.
\end{proof}

\begin{theorem}\label{cond2}
We have
$$ \E\left(\prod_{k=3}^{N} g(\lambda_k) \mid \lambda_1 = 0,  \lambda_2 = z\right)= \frac{1}{Z_N^{(0,z)}}\det(f_{i,j})_{i,j = 1}^{N-2} \quad \text{where} \ f_{i,j} = \frac{1}{ (i+1)!} \int \lambda^{i-1} \bar{\lambda}^{j-1} |\lambda|^2 |z - \lambda |^2 g (\lambda)  \mu(\dd \lambda)$$
and 
\begin{equation}\label{eqn:normalization}
Z_N^{(0,z)}=N^{-\frac{(N-2)(N+1)}{2}}\frac{e_1^{(N-1)}(N|z|^2)}{N|z|^2}.
\end{equation}
\end{theorem}

\begin{proof} 
By Andr\'eief's identity, the result holds with
$$Z_N^{(0,z)} = \det(M_{ij})_{i,j = 1}^{N-2},\  M_{ij}=\frac{1}{ (i+1) !} \int \lambda^{i-1} \bar{\lambda}^{j-1} |\lambda|^2 |z - \lambda|^2 \mu(\dd \lambda).$$
Expanding $|z - \lambda|^2 = |z|^2 + |\lambda|^2 - z \overline{\lambda} - \overline{z} \lambda$, we see that $M$ is tridiagonal with entries
$$
M_{ii}=\frac{1}{N^{i+1}}+\frac{|z|^2}{(i+1) N^i},\ 
M_{i,i+1}=-\frac{\overline{z}}{N^{i+1}},\ 
M_{i,i-1}=- \frac{z}{(i+1)N^{i}}.
$$
Denoting $x=N |z|^2$ and $d_k=\det((M_{ij})_{1\leq i,j\leq k})$, with the convention $d_0=1$ we have
\begin{align*}
d_1&=\frac{2+x}{2N^2},\\
d_k&=\left(1+\frac{x}{k+1}\right)\frac{1}{N^{k+1}}d_{k-1}-\frac{x}{k+1}\frac{1}{N^{2k+1}}d_{k-2}.
\end{align*}
so that $a_k=d_k N^{\frac{k(k+3)}{2}}$ satisfies $a_0=1$, $a_1=1+x/2$,
$$
a_k=\left(1+\frac{x}{k+1}\right)a_{k-1}-\frac{x}{k+1}a_{k-2}.
$$
This gives the expected result by an immediate induction.
\end{proof}

\noindent Kostlan's theorem now comes as a corollary, as well as a similar property for the Ginibre ensemble conditioned on $\lambda_1 =0$.

\begin{corollary}[Kostlan]\label{Kost1} The set $N\{ |\lambda_1|^2, \dots, |\lambda_N|^2 \} $ is distributed as $\{\gamma_1,\dots,\gamma_N\}$, a set of (unordered) independent Gamma variables of parameters $1,2, \dots, N$.\label{cor:Kos}
\end{corollary}

\begin{proof} Let $g \in \mathbb{C}[X]$ and use Theorem \ref{uncond} with the radially symmetric function $g(|\cdot|^2)$. The relevant matrix is then diagonal, with coefficients
$$ 
f_{i,i} = \frac{1}{ (i-1)!} \int |\lambda|^{2i-2} g(|\lambda|^2) \mu(\dd \lambda) 
= {N^{-(i-1)} \over (i-1)!} \int_{r=0}^{\infty} r^{i-1} g(r/N) e^{-r} \dd r 
 = N^{-(i-1)}\E(g(\gamma_i/N)).
$$
In other words, 
$$
\E\left(\prod_{i=1}^N g(|\lambda_i|^2)\right)=\E\left(\prod_{i=1}^N g(\gamma_{i}/N)\right). 
$$
Note that these statistics characterize the distribution of a set of unordered points, as such expressions with polynomial $g$ generate all symmetric polynomials, as shown in Lemma \ref{sympol}, and the gamma distributions are characterized by their moments. For more details, see \cite{Dub2017}. We conclude that $ N\{ |\lambda_1|^2, \dots, |\lambda_N|^2 \} \stackrel{d}{=} \{ \gamma_1, \dots, \gamma_N\}$.
\end{proof}

\begin{corollary}\label{kostbis} Conditioned on $\{ \lambda_1 =0\}$, $ \{ N |\lambda_2|^2, \dots, N |\lambda_N|^2 \} $ is distributed as $\{\gamma_2,\dots\gamma_N\}$, a set of (unordered) independent Gamma variables of parameters $2,3, \dots, N$.
\end{corollary}

\begin{proof} Similarly to the proof of Corollary \ref{cor:Kos}, we take $g \in \mathbb{C}[X]$ and the radially symmetric function $g(|\cdot|^2)$. In Theorem \ref{cond1}, we have
$$
f_{i,i}  = \frac{1}{ i!} \int |\lambda|^{2i} g(|\lambda|^2) \mu(\dd \lambda)
= {N^{-i} \over i!} \int_{r=0}^{\infty} r^{i} g(r/N) e^{-r} \dd r
 = N^{-i}\E \big[ g(\gamma_{i+1}/N) \big].
$$
This together with our expression for $Z_N^{(z=0)}$ in Theorem \ref{cond1} yields
\begin{align*} 
\E\left(\prod_{i=2}^N g(|\lambda_i|^2) \mid \lambda_1=0\right)= \E\left(\prod_{i=2}^N g(\gamma_{i})\right)
\end{align*}
and we conclude in the same way that $ N\{ |\lambda_2|^2, \dots, |\lambda_N|^2 \} \stackrel{d}{=} \{ \gamma_2, \dots, \gamma_N\} $.
\end{proof}

\noindent For the proof of the following lemma, we refer to \cite{Dub2017}.  We define the {\it product symmetric polynomials} as the symmetric polynomials given by products of polynomials in one variable:
$$ \mathrm{PS}_{\mathbb{C}}(N) = \left\{ \ \prod_{i=1}^N P(X_i) \ | \ P \in \mathbb{C} [X] \ \right\}$$ 
\begin{lemma}\label{sympol}
$\mathrm{PS}_{\mathbb{C}}(N)$ spans the vector space of symmetric polynomials of $N$ variables.
\end{lemma}

%\begin{proof} It is enough to see that such expressions span the {\it monomial} symmetric polynomials, defined for any $N$-tuple of integers $(a_1, \dots, a_N)$ as
%$$m_{(a_1,\dots ,a_N)} (X_1, \dots, X_N) = \sum_{\sigma \in \mathfrak{S}_N} \prod_{i=1}^N X_{\sigma(i)}^{a_{i}}$$
%If $b_1, \dots, b_k$ are the distinct integers involved in $(a_1, \dots, a_N)$, then for any parameter $t$ and any integer $M>N$ we expand the following element of $\mathrm{PSP}_N$:
%$$ Q_t(X_1, \dots, X_N) := \prod_{i=1}^N \Big( \sum_{j=1}^k t^{M^j} X_i^{b_j} \Big) = \sum_{\alpha_1+\dots+\alpha_k=N} t^{\sum \alpha_i M^i} m_{b_{\alpha}} (X_1, \dots, X_N) $$
%where $b_{\alpha}$ denotes the $N$-tuple where every $b_i$ is repeated $\alpha_i$ times. Note that $\sum \alpha_i M^i$ is an integer decomposition in base $M$ and thus characterizes the partition $\alpha$.
%
%Considering this equality for a number of distinct values of $t$ equal to the number of integer partitions of~$N$, the vector $(Q_{t_{\lambda}})_{\lambda \vdash N}$ is obtained from $(m_{b_{\alpha}})_{\alpha \vdash N}$ through the minor of an invertible Vandermonde determinant. The minor is itself invertible, and this gives us in turn an expression of $m_{(a_1,\dots ,a_N)} (X_1, \dots, X_N)$ as a linear combination of elements of $\mathrm{PSP}_N$. \end{proof}

\setcounter{equation}{0}
\setcounter{theorem}{0}
\renewcommand{\theequation}{A.\arabic{equation}}
\renewcommand{\thetheorem}{A.\arabic{theorem}}
\appendix
\setcounter{secnumdepth}{0}
\section{Appendix A\ \ \ Eigenvalues dynamics}\label{app:dyn}

This Appendix derives the Dyson-type dynamics for eigenvalues of nonnormal matrices. More precisely, we consider the Ornstein-Uhlenbeck version so that the equilibrium measure is the (real or complex) Ginibre ensemble. These dynamics take a particularly simple form in the case of complex Gaussian addition, where the drift term shows no interaction between eigenvalues: only the correlation of martingale terms is responsible for eigenvalues repulsion. 

We also describe natural dynamics with equilibrium measure given by the real Ginibre ensemble. Then, the eigenvalues evolution is more intricate.

It was already noted in \cite{BurGreNowTarWar2014} that  eigenvectors impact the eigenvalues dynamics for nonnormal matrices, and the full dynamics in the complex case have been written down in \cite{GrelaWarchol}.

\subsection{Complex Ginibre dynamics.}\  Let $G(0)$ be a complex matrix of size $N$, assumed to be diagonalized as $YGX = \Delta = \mathrm{Diag}(\lambda_1,\dots,\lambda_N) $, where $X,Y$ are the matrices of the right- and left-eigenvectors of $G(0)$. We also assume that $G(0)$ has simple spectrum, and $X,Y$ invertible. The right eigenvectors $(x_i)$ are the columns of $X$, and the left-eigenvectors $(y_j)$ are the rows of $Y$. They are chosen uniquely such that $XY=I$ and, for any  
$1\leq k\leq N$, 
$X_{kk}=1$.

We now consider the complex Dyson-type dynamics: for any $1\leq i,j\leq N$,
\begin{equation}\label{eqn:dynamics}
\rd G_{ij}(t)=\frac{\rd B_{ij}(t)}{\sqrt{N}}-\frac{1}{2}G_{ij}(t)\rd t,
\end{equation}
where the $B_{ij}$'s are independent standard complex Brownian motions: $\sqrt{2}\re(B_{ij})$ and $\sqrt{2}\im(B_{ij})$
are standard real Brownian motions. One can easily check that $G(t)$ converges to the Ginibre ensemble as $t\to\infty$,
with normalization (\ref{eqn:Gini}).

In the following, the bracket of two complex martingales $M,N$ is defined by bilinearity: $\langle M,N\rangle=\langle\re M,\re N\rangle-\langle\im M,\im N\rangle
+\ii\langle \re M,\im N\rangle+\ii\langle\im M,\re N\rangle$.

\begin{proposition}\label{prop:dynamics} The spectrum $(\lambda_1(t),\dots , \lambda_n(t))$ is a semimartingale satisfying the system of equations
$$ \dd \lambda_k(t) = \rd M_k(t)  -  \frac{1}{2} \lambda_k(t)\dd t  $$
where the martingales $(M_k)_{1\leq k\leq N}$ have brackets $\langle M_i,M_j\rangle=0$ and
$$
\rd\langle M_i,\overline{M_j}\rangle_t=\mathscr{O}_{ij}(t)\frac{\rd t}{N}.
$$ 
\end{proposition} 

\begin{remark}
As explained below, this equation (in particular the off-diagonal brackets) is coherent with the eigenvalues repulsion observed in (\ref{eqn:partitionfunction}).
Contrary to the Hermitian Dyson Brownian motion, all eigenvalues are martingales (up to the Ornstein Uhlenbeck drift term), so that their repulsion is not due to direct mutual interaction, but to correlations between these martingales at the microscopic scale.
 
For example, assume that $G(0)$ is already at equilibrium. Using physics conventions, for any bulk eigenvalues $\la_1,\la_2$ satisfying $w=\OO(1)$ (remember $w=\sqrt{N}(\lambda_1-\lambda_2)$), Proposition \ref{prop:dynamics} and Theorem \ref{thm:meso2} imply
$$
\E\left(\rd\lambda_1\rd \overline{\lambda_2}\mid \lambda_1=z_1,\lambda_2=z_2\right)\sim \E(\mathscr{O}_{12}\mid  \lambda_1=z_1,\lambda_2=z_2)\frac{\rd t}{N}\sim-(1-|z_1|^2)\frac{1}{|w|^4}\frac{1-(1+|\omega|^2)e^{-|\omega|^2}}{1-e^{-|\omega|^2}}\rd t$$
in the bulk.
By considering the real part in this equation and denoting $\rd\lambda_1=\rd x_1+\ii\rd y_1$, $\rd\lambda_2=\rd x_2+\ii \rd y_2$, we have in particular 
$\E(\rd x_1\rd x_2+\rd y_1\rd y_2)<0$, and this negative correlation is responsible for repulsion: the eigenvalues tend to move in opposite directions. Moreover, as eigenvalues get closer on the microscopic scale, $w\to 0$ and the repulsion gets stronger:
$$
\E\left(\rd\lambda_1\rd \overline{\lambda_2}\mid \lambda_1=z_1,\lambda_2=z_2\right)\sim -\frac{1-|z_1|^2}{|w|^2}\rd t.
$$

On the other hand, for  mesoscopic scale $N^{-1/2}\ll|\lambda_1-\lambda_2|$, Proposition \ref{prop:dynamics} and Theorem \ref{thm:meso2} give
$
\E\left(\rd\lambda_1\rd \overline{\lambda_2}\right)\sim -\frac{(1-|\lambda_1|^2)}{N^2|\lambda_1-\lambda_2|^4}\rd t=\oo(\rd t)$,
so that increments are uncorrelated for large $N$.
\end{remark}

For a given differential operator $f\mapsto f'$, we introduce the matrix $C=X^{-1}X'$. 
Along the following lemmas, all eigenvalues are assumed to be distinct. In our application, this spectrum simplicity will hold almost surely for any $t\geq 0$ as $G(0)$ has simple spectrum. 

\begin{lemma} We have $X'=XC$ and $Y'=-CY$. \end{lemma}
\begin{proof} The first equality is the definition of $C$. For the second one, $XY=I$ gives $XY'+X'Y=0$, hence $Y' = - X^{-1} X' Y = -CY.$ \end{proof}

\begin{lemma}\label{lem:firstorder} The first order perturbation of eigenvalues is given by $\lambda_k ' = y_k G' x_k$. \end{lemma}

\begin{proof} We have $\Delta' = (YGX)' = Y'GX + YG'X + YGX' = YG'X + YGXC - C YGX =  YG'X + \Delta C - C \Delta= YG'X + [\Delta,C]$. Therefore $\lambda_k' = (YG'X)_{kk} + [\Delta,C]_{kk} = y_k G' x_k$.  \end{proof}

\begin{lemma} For any  $i \neq j$, $C_{ij} = {y_i G' x_j \over \lambda_j - \lambda_i}$. \end{lemma}

\begin{proof} For such $i$, $j$, $\Delta'_{ij} = 0$. With the same computation as in the previous lemma, this gives
$(YG'X)_{ij} + [\Delta,C]_{ij}=0$.
Thus $(\lambda_i - \lambda_j) C_{ij} = - (YG'X)_{ij} = - y_i G' x_j$, from which the result follows.  
\end{proof}

\begin{lemma} For any $1\leq k\leq N$, $C_{kk} = - \sum_{l \neq k} X_{kl} {y_l G' x_k \over \lambda_k - \lambda_l}$. \end{lemma}

\begin{proof} We use the assumption $ X_{kk}=1$. From this, and the definition of $C$, we get
$$X'_{kk}=0=(XC)_{kk} = \sum_{l=1}^n X_{kl} C_{lk} = X_{kk} C_{kk} + \sum_{l \neq k} X_{kl} C_{lk}.$$
As a consequence, 
$C_{kk} = - \sum_{l \neq k} X_{kl} C_{lk}$
and  we obtain
the result thanks to the previous lemma. 
\end{proof}

From now on the differential operator will be either $\partial_{\re G_{ab}}$ ($G' = E_{ab}= \{ \delta_{ia} \delta_{jb} \}_{1\leq i,j\leq N}$), or $\partial_{\im G_{ab}}$, ($G' = \ii E_{ab})$. In both cases, $G''=0$. We denote $C^{\re}$ and $C^{\im}$ accordingly. In particular, for any $k$ and $i\neq j$ the following holds:
\begin{align}& \partial_{\re G_{ab}} \lambda_k  = Y_{ka} X_{b,k},\ 
\partial_{\im G_{ab}}\lambda_k = \ii Y_{ka} X_{b,k} \notag\\
&C_{ij}^{\re}  = {Y_{ia} X_{bj} \over \lambda_j - \lambda_i},\ 
 C_{kk}^{\re} = - \sum_{l \neq k} X_{kl} { Y_{la} X_{b,k} \over \lambda_k - \lambda_l}, \ 
C_{ij}^{\im}  = \ii {Y_{ia} X_{bj} \over \lambda_j - \lambda_i}, \ 
 C_{kk}^{\im} = - \ii \sum_{l \neq k} X_{kl} { Y_{la} X_{b,k} \over \lambda_k - \lambda_l}. \label{eqn:formulaC}
 \end{align}

\begin{lemma} We have 
\begin{align*}
&\partial_{\re G_{ab}} X_{ij} = \sum_{l \neq j}  (X_{il}- X_{ij} X_{jl} ) {Y_{la} X_{bj} \over \lambda_j - \lambda_l},\ \ 
\partial_{\im G_{ab}} X_{ij} = \ii \sum_{l \neq j}  (X_{il}- X_{ij} X_{jl} ) {Y_{la} X_{bj} \over \lambda_j - \lambda_l}. \\
\partial_{\re G_{ab}} Y_{ij} &= \sum_{l \neq i}  \frac{1}{ \lambda_i - \lambda_l} (X_{il}  Y_{la} X_{b,i} Y_{ij} +  Y_{ia} X_{bl} Y_{lj}),\ \ 
  \partial_{\im G_{ab}} Y_{ij} = \ii \sum_{l \neq i}  \frac{1}{ \lambda_i - \lambda_l} (X_{il}  Y_{la} X_{b,i} Y_{ij} +  Y_{ia} X_{bl} Y_{lj}).
\end{align*}
\end{lemma}

\begin{proof} Below is the computation for $ \partial_{\re G_{ab}}X_{ij} $. We use $X'=XC$ and (\ref{eqn:formulaC}):
\begin{multline*} X'_{ij} = (XC)_{ij} 
= \sum_{l=1}^n X_{il} C_{lj} 
=  \sum_{l \neq j} X_{il} {Y_{la} X_{bj} \over \lambda_j - \lambda_l} - X_{ij} \sum_{l \neq j} X_{jl} { Y_{la} X_{bj} \over \lambda_j - \lambda_l} 
= \sum_{l \neq j}  (X_{il}- X_{ij} X_{jl} ) {Y_{la} X_{bj} \over \lambda_j - \lambda_l}.
\end{multline*}
The case $\partial_{\im G_{ab}}X_{ij}$ is obtained similarly, as are the formulas for $Y$.
 \end{proof}
 
\begin{lemma}\label{lem:secondorder} The second order perturbation of eigenvalues is given by 
$$\partial_{\re G_{ab}}^2  \lambda_k  = 2 \sum_{l \neq k}  { Y_{ka} X_{bl} Y_{la}  X_{b,k} \over \lambda_k - \lambda_l},\ \ 
\partial_{\im G_{ab}}^2  \lambda_k  = -2 \sum_{l \neq k}  { Y_{ka} X_{bl} Y_{la}  X_{b,k} \over \lambda_k - \lambda_l}. $$
\end{lemma}

\begin{proof} We compute the perturbation for $ \partial_{\re G_{ab}} $. Differentiating $\lambda$ a second time gives
$$\lambda_k'' = y_k' G' x_k + y_k G'' x_k+ y_k G' x_k'
 = Y_{ka}' X_{b,k} + Y_{ka} X_{b,k}'.
$$
Replacing $X'$ and $Y'$ with their expressions yields
\begin{align*} \lambda_k'' & =  \sum_{l \neq k}  \frac{1}{ \lambda_k - \lambda_l} (X_{kl}  Y_{la} X_{b,k} Y_{ka} +  Y_{ka} X_{bl} Y_{la})  X_{b,k} + Y_{ka} \sum_{l \neq j}  (X_{bl}- X_{b,k} X_{kl} ) {Y_{la} X_{b,k} \over \lambda_k - \lambda_l} \\
& =  \sum_{l \neq k}  \frac{1}{ \lambda_k - \lambda_l} (X_{kl}  Y_{la} X_{b,k} Y_{ka} X_{b,k} +  Y_{ka} X_{bl} Y_{la}  X_{b,k} + Y_{ka} X_{bl} Y_{la} X_{b,k}-  Y_{ka}X_{b,k} X_{kl} Y_{la} X_{b,k}  ) \\
&= 2 \sum_{l \neq k}  { Y_{ka} X_{bl} Y_{la}  X_{b,k} \over \lambda_k - \lambda_l},
\end{align*}
which concludes the proof, the other cases being similar.
\end{proof}

For the proof of Proposition \ref{prop:dynamics}, we need the following elementary lemma.

\begin{lemma}\label{lem:Collison}
Let $\tau=\inf\{t\geq 0:\exists i\neq j, \lambda_i(t)=\lambda_j(t)\}$. Then $\tau=\infty$ almost surely.
\end{lemma}

\begin{proof}
The set of matrices $G$ with Jordan form of type
$$
\lambda_1\oplus\dots\oplus \lambda_{N-2}\oplus
\left(\begin{array}{cc}
\lambda_{N-1}&1\\
0&\lambda_{N-1}\end{array}\right)\
({\rm respectively}\ 
\lambda_1\oplus\dots\oplus \lambda_{N-2}\oplus \lambda_{N-1}\oplus\lambda_{N-1})
$$
is a submanifold $\mathcal{M}_1$ (resp. $\mathcal{M}_2$) of $\mathbb{C}^{N^2}$ with  complex codimension $1$ (resp. $3$),  see e.g. \cite{OveWom1988,Kel2008}. Therefore, almost surely, a Brownian motion
in $\mathbb{C}^{N^2}$ starting from a diagonalizable matrix with simple spectrum will not hit  $\mathcal{M}_1$ or $\mathcal{M}_2$. This concludes the proof.
\end{proof}

All derivatives can therefore
be calculated, as eigenvalues and eigenvectors are analytic functions of the matrix entries (see \cite{Kat1980}).

\begin{proof}[Proof of Proposition \ref{prop:dynamics}] 
In our context,the It\^o formula will take the following form: for a function $f$ from $\mathbb{C}^n$ to $\mathbb{C}$ of class $C^2$, where $B_t=(B_t^1,\dots,B_t^n)$ is made of independent  standard complex Brownian motions, we have
\begin{equation}\label{eqn:ItoCom} \dd f (B_t) = \sum_{i=1}^n \Bigg({\partial f \over \partial \re z_i } \dd \re B_t^{i} + {\partial f \over \partial \im z_i} \dd \im B_t^{i} \Bigg)+ \frac{1}{ 2} \Bigg(\sum_{i=1}^n {\partial^2 f \over \partial \re {z_i}^2 } + {\partial^2 f \over \partial {\im z_i}^2 } \Bigg) \dd t.\end{equation}
For any given $0<\e<\min\{|\lambda_i(0)-\lambda_j(0)|,i\neq j\}$, let
\begin{equation}\label{eqn:taue}
\tau_{\e}=\inf\{t\geq 0:\exists i\neq j, |\lambda_i(t)-\lambda_j(t)|<\e\}.
\end{equation}
Eigenvalues are smooth functions of the matrix coefficients on the domain $\cap_{i<j}\{|\lambda_i-\lambda_j|>\e\}$, so that
equation (\ref{eqn:ItoCom}) together with Lemmas \ref{lem:firstorder} and \ref{lem:secondorder} gives the following equality of stochastic integrals, with substantial cancellations of the drift term:
\begin{multline*} \dd \lambda_k(t\wedge\tau_{\e})  = \sum_{i,j=1}^n Y_{ki} X_{jk} \left(\frac{\dd B_{ij}(t\wedge\tau_{\e})}{\sqrt{N}}  -  \frac{G_{ij} }{2}\dd (t\wedge\tau_\e)\right) + 
\frac{1}{N} \sum_{i,j,\ell\neq k}\Bigg(   { Y_{ki} X_{jl} Y_{li}  X_{jk} \over \lambda_k - \lambda_l} -  { Y_{ki} X_{jl} Y_{li}  X_{jk} \over \lambda_k - \lambda_l} \Bigg) \dd (t\wedge\tau_{\e})  \\
=  \sum_{i,j=1}^n Y_{ki} X_{jk} \frac{\dd B_{ij}(t\wedge\tau_{\e})}{\sqrt{N}}  -\frac{1}{2}  \sum_{i,j=1}^n Y_{ki} G_{ij} X_{jk} \dd (t\wedge\tau_{\e}) 
=  \sum_{i,j=1}^n Y_{ki} X_{jk} \frac{\dd B_{ij}(t\wedge\tau_{\e})}{\sqrt{N}}  -\frac{1}{2}   \lambda_k \dd (t\wedge\tau_{\e}).
\end{multline*}
Taking $\e\to 0$ in the above equation together with Lemma \ref{lem:Collison} yields
$$
\dd \lambda_k(t)= \sum_{i,j=1}^n Y_{ki} X_{jk} \frac{\dd B_{ij}(t)}{\sqrt{N}}  -\frac{1}{2}   \lambda_k \dd t.
$$
The eigenvalues martingales terms are correlated. Their brackets are 
\begin{align} &\dd \langle \lambda_i, \bar\lambda_j \rangle_t = \frac{1}{N} \sum_{a,b,c,d=1}^n Y_{ia} X_{b,i} \overline{Y_{jc} X_{d,j}} \dd \langle B_{ab}, \overline{ \dd B_{cd}}\rangle_t=  (X^{\rm t}\overline{X})_{ij} (Y Y^*)_{ij} \frac{\dd t}{N} = \mathscr{O}_{ij}(t) \frac{\dd t}{N},\\ 
&\dd \langle \lambda_i, \lambda_j \rangle_t=0.
\end{align}
This concludes the proof.
\end{proof}

\subsection{Proof of Corollary \ref{cor:diffusive}.}\ 
Let $A_{t,\e}=\{\sup_{0\leq s\leq t}|\la_1(s)-\la_1(0)|<N^\e t^{1/2}\}$. We start by proving that 
\begin{equation}
\mathbb{P}(A_{t,\e})=1-\oo(1).\label{eqn:scaling1}
\end{equation}
From Proposition \ref{prop:dynamics} and It{\^ o}'s formula, we have
\begin{equation}\label{eqn:It1}
e^{\frac{t}{2}}\la_1(t)-\la_1(0)=\int_0^te^{\frac{s}{2}}\rd M_1(s),
\end{equation}
which is a local martingale. It is an actual martingale because 
\begin{equation}\label{eqn:It2}
\E\left(\langle \int_0^{\cdot}e^{\frac{s}{2}}\rd M_1(s)\rangle_t\right)=\int_0^t\E\left(e^{\frac{s}{2}}\frac{\mathscr{O}_{11}(s)}{N}\rd s\right)=\OO(t)<\infty,
\end{equation}
where in the last equality we used $\E(\mathscr{O}_{11}(s))=\OO(N)$, which follows from (\ref{eqn:exact}). The estimate (\ref{eqn:scaling1}) follows by \ Doob's  and Markov's inequalities.

For (\ref{eqn:scale}), we start  with
\begin{equation}\label{eqn:It}
|e^{\frac{t}{2}}\la_1(t)-\la_1(0)|^2=2{\rm Re}\int_0^t\overline{e^{\frac{s}{2}}\la_1(s)-\la_1(0)}e^{\frac{s}{2}}\rd M_1(s) +\int_0^te^s\frac{\mathscr{O}_{11}(s)}{N}\rd s.
\end{equation}
This implies
\begin{equation}\label{eqn:lalalo}
\E\left(e^t|\la_1(t)-\la_1(0)|^2  \mathds{1}_{\{\lambda_1(0)\in \mathscr{B}\}}\right)
=
\int_0^t\E\left(e^s\frac{\mathscr{O}_{11}(s)}{N} \mathds{1}_{\{\lambda_1(0)\in \mathscr{B}\}}\right)\rd s+\oo(t).
\end{equation}
Here, we used that $({\rm Re}\int_0^t\overline{e^{\frac{s}{2}}\la_1(s)-\la_1(0)}e^{\frac{s}{2}}\rd M_1(s) )_t$ is an actual martingale, because the expectation of its bracket is
$$
\int_0^te^{s}\mathbb{E}\left(|e^{\frac{s}{2}}\lambda_1(s)-\lambda_1(0)|^2\frac{\mathscr{O}_{11}(s)}{N}\mathds{1}_{\{\lambda_1(0)\in \mathscr{B}\}}\rd s\right)
\leq
2\int_0^te^{2s}\mathbb{E}\left(|\lambda_1(s)|^2+1)\frac{\mathscr{O}_{11}(s)}{N}\rd s\right)<\infty,
$$
where for the last inequality we used (\ref{eqn:exact}).

To evaluate the right hand side of (\ref{eqn:lalalo}), we would like to change $\lambda_1(0)\in\mathscr{B}$ into $\lambda_1(s)\in\mathscr{B}$. First,
\begin{equation}\label{eqn:lalalololulu}
\left|\E\left(\frac{\mathscr{O}_{11}(s)}{N} \mathds{1}_{A_{t,\e}}\left(\mathds{1}_{\lambda_1(0)\in \mathscr{B}}-\mathds{1}_{\lambda_1(s)\in \mathscr{B}}\right)\right)\right|
\leq \E\left(\frac{\mathscr{O}_{11}(s)}{N} \mathds{1}_{{\rm dist}(\lambda_1(s),\partial \mathscr{B})\leq N^\e t^{1/2}}\right)=\OO(N^\e t^{1/2}),
\end{equation}
where for the last inequality we used (\ref{eqn:exact}), again. 
Moreover, if $1/p+1/q=1$ with $p<2$. we have
\begin{equation}\label{eqn:lalalololululili}
\E\left(\frac{\mathscr{O}_{11}(s)}{N} \mathds{1}_{(A_{t,\e})^c}\right)\leq 
\E\left(\left(\frac{\mathscr{O}_{11}(s)}{N}\right)^p\right)^{1/p}
\mathbb{P}\left((A^{(1)}_{t,\e})^c\right)^{1/q}=\oo(1),
\end{equation}
where we used  \cite[Theorem 2.3]{Fyodorov2018} to obtain that  uniformly in the complex plane and in $N$, $\mathscr{O}_{11}/N$ has finite moment of order $p<2$. Equations 
(\ref{eqn:lalalo}), (\ref{eqn:lalalololulu}) and (\ref{eqn:lalalololululili})  imply 
$$
\E\left(|\la_1(t)-\la_1(0)|^2  \mathds{1}_{\lambda_1(0)\in \mathscr{B}}\right)
=
\int_0^t\E\left(\frac{\mathscr{O}_{11}(s)}{N} \mathds{1}_{\{\lambda_1(s)\in \mathscr{B}\}}\right)\rd s+\oo(t),
$$
and one concludes the proof of (\ref{eqn:scale}) with (\ref{eqn:exact}).

The proof of (\ref{eqn:indepen}) is identical, except that we rely on the off-diagonal bracket $\rd\langle \la_1,\bar \la_2\rangle_s=\mathscr{O}_{12}(s)\frac{\rd s}{N}$, the estimate (\ref{eqn:diagok}), and the elementary inequality
$$
|\mathscr{O}_{12}|=|(R_j^* R_i)(L_j^* L_i)|\leq  \|R_j\| \|R_i\|\|L_j\| \|L_i\|\leq \frac{1}{2}\left( \|R_i\|^2 \|L_i\|^2+\|R_j\|^2\|L_j\|^2\right)=\frac{1}{2}\left(\mathscr{O}_{11}+\mathscr{O}_{22}\right)
$$
to bound the (first and $p$-th) moment of $\mathscr{O}_{12}$ in the whole complex plane based on those of $\mathscr{O}_{11}$, $\mathscr{O}_{22}$.

\subsection{Real Ginibre dynamics.}\ 
We now consider $G(0)$ a real matrix of size $N$, again assumed to be diagonalized as $YGX = \Delta = \mathrm{Diag}(\lambda_1,\dots,\lambda_N) $, where $X,Y$ are the matrices of the right- and left-eigenvectors of $G(0)$. We also assume that $G(0)$ has simple spectrum, and $X,Y$ invertible. We keep the same notations for the right eigenvectors $(x_i)$, columns of $X$, and the left-eigenvectors $(y_j)$, rows of $Y$. They are again chosen such that $XY=I$ and, for any  
$1\leq k\leq N$, 
$X_{kk}=1$.

In this subsection, the real Dyson-type dynamics are ($1\leq i,j\leq N$),
\begin{equation}\label{eqn:real}
\rd G_{ij}(t)=\frac{\rd B_{ij}(t)}{\sqrt{N}}-\frac{1}{2}G_{ij}(t)\rd t,
\end{equation}
where the $B_{ij}$'s are independent standard  Brownian motions. One can easily check that $G(t)$ converges to the real Ginibre ensemble as $t\to\infty$.

Note that the real analogue of Lemma \ref{lem:Collison}
gives weaker repulsion:
the set of real matrices with Jordan form of type
$$
\lambda_1\oplus\dots\oplus \lambda_{N-2}\oplus
\left(\begin{array}{cc}
\lambda_{N-1}&1\\
0&\lambda_{N-1}\end{array}\right)
$$
is a submanifold $\mathcal{M}_1$ of $\mathbb{R}^{N^2}$, supported on $\lambda_{N-1}\in\mathbb{R}$, with real codimension $1$ (as proved by a straightforward adaptation of \cite[Theorem 7]{Kel2008}). Denoting  $\tau=\inf\{t\geq 0:\exists i\neq j, \lambda_i(t)=\lambda_j(t)\}$, under the dynamics (\ref{eqn:real}) for any $t>0$ we therefore have
$$
\mathbb{P}(\tau<t)>0,
$$
so that we can only state the real version of Proposition \ref{prop:dynamics} up to time $\tau$. In fact, collisions occur transforming pairs of real eigenvalues into pairs of complex conjugate eigenvalues,
a mechanism coherent with the random number of real eigenvalues in the real Ginibre ensemble \cite{LehSom1991,Ede1997}.

The overlaps (\ref{eqn:overlap}) are enough to describe the complex Ginibre dynamics, and so are they for the real Ginibre ensemble, up to the introduction of the following notation:
we define $\bar i\in\llbracket 1,N\rrbracket$ through
$\lambda_{\bar i}=\overline{\lambda_i}$, i.e. $\bar i$ is the index of the conjugate eigenvalue to $\lambda_i$. Note that $\bar i=i$ if $\lambda_i\in\mathbb{R}$.
For real matrices, if $L_j,R_j$ are eigenvectors associated to $\lambda_j$,  $\bar L_j,\bar R_j$ are eigenvectors for $\bar\lambda_j$, so that
$$
\mathscr{O}_{i\bar j}=(\bar R_j^* R_i)(\bar L_j^* L_i)=(R_j^{\rm t} R_i)(L_j^{\rm t} L_i).
$$

\begin{proposition}\label{prop:dynamicsReal} The spectrum $(\lambda_1(t),\dots , \lambda_n(t))$ evolves according to the following stochastic equations, up to the first collision:
$$ \dd \lambda_k(t\wedge\tau) = \rd M_k(t\wedge\tau)+ \left(
\sum_{l\neq k}\frac{ {\mathscr{O}}_{k\bar l}}{\lambda_k-\lambda_\ell} 
-\frac{1}{2} \lambda_k\right)\dd (t\wedge\tau) $$
where the martingales $(M_k)_{1\leq k\leq N}$ have brackets
$$
\rd \langle M_i,M_j\rangle_{t\wedge\tau}={\mathscr{O}}_{i\bar j}(t)\frac{\rd (t\wedge\tau)}{N},\ \ 
\rd\langle M_i,\overline{M_j}\rangle_{t\wedge\tau}=\mathscr{O}_{ij}(t)\frac{\rd (t\wedge\tau)}{N}.$$
\end{proposition} 

Note that  the real eigenvalues have associated real eigenvectors. For those, ${\mathscr{O}}_{k\bar l}=\mathscr{O}_{kl}$, and the variation is real: real eigenvalues remain real as long as they do not collide.\\

\begin{remark}
Proposition \ref{prop:dynamicsReal} is coherent with the attraction between conjugate eigenvalues exhibited in \cite{Mov2016}. In fact, if $\eta=\im(\lambda_k)>0$, the drift interaction term with $\bar \lambda_k$ is
$\mathscr{O}_{kk}/(\lambda_k-\bar\lambda_k)=-\ii \mathscr{O}_{kk}/(2\eta)$, so that these eigenvalues attract each other stronger as they approach the real axis.
\end{remark}

For the proof, we omit the details  and only mention the  differences with respect to Proposition \ref{prop:dynamics}. We apply the It\^o formula for a $\mathscr{C}^2$ function $f$ from $\mathbb{R}^n$ to $\mathbb{C}$, with argument $U_t=(U_t^1,\dots,U_t^n)$ is made of independent Ornstein-Uhlenbeck processes.
Together with the perturbation formulas for $\lambda_k$, Lemmas \ref{lem:firstorder} and \ref{lem:secondorder}, we obtain (remember the notation (\ref{eqn:taue}))
\begin{multline*} \dd \lambda_k(t\wedge\tau_{\e})  = \sum_{i,j=1}^n Y_{ki} X_{jk} \left(\frac{\dd B_{ij}(t\wedge\tau_{\e})}{\sqrt{N}}  -  \frac{1}{2}G_{ij} \dd (t\wedge\tau_\e\right) + \sum_{i,j} \sum_{l \neq k}  { Y_{ki} X_{jl} Y_{li}  X_{jk} \over \lambda_k - \lambda_l}  \dd (t\wedge\tau_{\e})  \\
=  \sum_{i,j=1}^n Y_{ki} X_{jk} \frac{\dd B_{ij}(t\wedge\tau_{\e})}{\sqrt{N}} +\left(\sum_{l\neq k}\frac{ (X^tX)_{lk} (YY^t)_{kl}}{\lambda_k-\lambda_\ell} -\frac{1}{2}   \lambda_k \right)\dd (t\wedge\tau_{\e}).
\end{multline*}
We can take $\e\to 0$ in the above formulas and the brackets are calculated as follows, concluding the proof:
\begin{align} &\dd \langle \lambda_i, \bar\lambda_j \rangle_{t\wedge\tau} = \frac{1}{N} \sum_{a,b,c,d=1}^n Y_{ia} X_{b,i} \overline{Y_{jc} X_{d,j}} \dd \langle B_{ab}, \overline{ \dd B_{cd}}\rangle_{t\wedge\tau}=  (X^{\rm t}\overline{X})_{ij} (Y Y^*)_{ij} \frac{\dd (t\wedge\tau)}{N} = \mathscr{O}_{ij}(t) \frac{\dd (t\wedge\tau)}{N},\\ 
&\dd \langle \lambda_i, \lambda_j \rangle_{t\wedge\tau}=
 \frac{1}{N} \sum_{a,b,c,d=1}^n Y_{ia} X_{b,i} {Y_{jc} X_{d,j}} \dd \langle B_{ab},{ \dd B_{cd}}\rangle_{t\wedge\tau}=  (X^{\rm t}{X})_{ij} (Y Y^{\rm t})_{ij} \frac{\dd(t\wedge\tau)}{N} = \mathscr{O}_{i\bar j}(t) \frac{\dd (t\wedge\tau)}{N}.
\end{align}

\setcounter{equation}{0}
\setcounter{theorem}{0}
\renewcommand{\theequation}{B.\arabic{equation}}
\renewcommand{\thetheorem}{B.\arabic{theorem}}
\appendix
\setcounter{secnumdepth}{0}
\section{Appendix B\ \ \ Normalized eigenvectors}\label{app:nor}

This paper focuses on the condition numbers and off-diagonal overlaps, but the Schur decomposition also easily gives information about other statistics such as the angles between eigenvectors.
We include these results for the sake of completeness.
We denote the complex angle as
$$
\arg(\lambda_1,\lambda_2)={R_1^{*} R_2  \over \| R_1 \| \|R_2 \|},
$$
where the phases of $R_1(1)$ and $R_2(1)$ can be chosen independent uniform on $[0,2\pi)$.
We also define
$$ \Phi (z) = {z \over \sqrt{1+|z|^2} }.$$
\begin{proposition}\label{prop:angle} Conditionally on $\lambda_1=z_1,\lambda_2=z_2$, we have 
$$ \arg(\lambda_1,\lambda_2) \overset{(\rm d)}{=}\Phi \left({X\over \sqrt{N} |z_1 - z_2| }\right)$$
where $X \sim \mathscr{N}_{\mathbb{C}}(0,\frac{1}{2} {\rm Id})$.
\end{proposition}
In particular, for $\lambda_1, \lambda_2$ at mesoscopic distance, the complex angle converges in distribution to a Dirac mass at $0$. Therefore in such a setting eigenvectors strongly tend to be orthogonal: matrices sampled from the Ginibre ensemble are not far from normal, when only considering eigenvectors angles.
The limit distribution becomes non trivial in the microscopic scaling $ |\lambda_1 - \lambda_2| \sim N^{-1/2}$, it is the pushforward of a complex Gaussian measure by $\Phi$.
\begin{proof}
From Proposition \ref{fund} we know that $R_1^*R_2=R_{T,1}^*R_{T,2}$, $\|R_1\|=\|R_{T,1}\|$ and $\|R_2\|=\|R_{T,2}\|$, where $R_{T,i}$ (and $L_{T,i}$) are the normalized bi-orthogonal bases of right and left eigenvectors 
for $T$, defined as (\ref{eqn:T}). The first eigenvectors are written $R_{T,1}=(1,0,\dots,0)$ and $R_{T,2}=(a,1,0\dots,0)$ where $a=-\bar b_2=-\frac{T_{12}}{\lambda_1-\lambda_2}$, with $T_{12}$ complex Gaussian $\mathscr{N}\left(0,\frac{1}{2N}{\rm Id}\right)$, independent of $\lambda_1$ and $\lambda_2$. This gives
$$
\arg(\lambda_1,\lambda_2) =-\frac{\bar b_2}{\sqrt{1+|b_2|^2}}
$$
and concludes the proof. 
\end{proof}
From Proposition \ref{prop:angle}, the distribution of the angle for fixed $\la_1$ and random $\la_2$ can easily be inferred.
For example, if $\la_2$ is chosen uniformly among eigenvalues in a macroscopic domain $\Omega\subset\{|z|<1\}$ with nonempty interior, we obtain the convergence in distribution ($X_\Omega$ is uniform on $\Omega$, independent of $\mathscr{N}$)
$$
N | \arg(\lambda_1 ,\lambda_2) |^2 \underset{N\to\infty}{\to}\frac{|\mathscr{N}|^2}{|z_1-X_\Omega|^2}.
$$
When $z_1=0$ and $z_2$ is free, the following gives a more precise distribution, for finite $N$ and in the limit.
\begin{corollary} Conditionally on $\{ \lambda_1 =0 \}$ we have
$$N | \arg(\lambda_1 ,\lambda_2) |^2 \overset{(\rm d)}{=} N\beta_{1,U_N}\overset{(\rm d)}{\longrightarrow}X$$
where $U_N$ is an independent random variable uniform on $ \{ 2,\dots,N\}$, and $X$ has density
$ {1 - (1+t) e^{-t} \over t^2} \mathbf{1}_{\mathbb{R}_+} (t).$
\end{corollary}
\begin{proof} From Corollary \ref{kostbis}, $N |\lambda_2|^2 \sim \gamma_{U_N}$. Together with Lemma \ref{lemmebeta1}, this gives
$$| \arg(\lambda_1 ,\lambda_2) |^2 = {  {|\mathscr{N}|^2 \over N | \lambda_2|^2 } \over 1+  {|\mathscr{N}|^2 \over N | \lambda_2|^2 } } \overset{(\rm d)}{=} {\gamma_1 \over \gamma_1 + \gamma_{U_N}} \overset{(\rm d)}{=} \beta_{1,U_N}.$$
The limiting density then follows from the explicit distribution of $\beta$ random variables.
\end{proof}

\end{document}